\documentclass[aos,preprint]{imsart}
\setattribute{journal}{name}{}

\RequirePackage[colorlinks,citecolor=blue,urlcolor=blue]{hyperref}

\usepackage{amssymb}
\usepackage{amsmath,amsthm}
\usepackage{algorithm}
\usepackage[noend]{algpseudocode}
\usepackage{enumitem}
\usepackage[bottom]{footmisc}

\usepackage{latexsym}
\usepackage{epsfig}
\usepackage[a4paper,left=4cm,right=3cm,top=3cm,bottom=3cm]{geometry}

\newcommand{\sign}{\mathrm{sign}}

\newcommand{\PP}{\mathbb{P}}
\newcommand{\EE}{{\mathbb{E}}}

\newcommand{\eps}{\varepsilon}

\newcommand{\FF}{\mathcal{F}}

\newcommand{\XX}{\mathcal{X}}

\newcommand{\YY}{\mathcal{Y}}

\newcommand{\EEE}{\mathcal{E}}

\renewcommand{\phi}{\varphi}

\newcommand{\given}{\,|\,}

\newcommand{\K}{{K}}
\newcommand{\pp}{\tilde p}
\makeatletter
\def\BState{\State\hskip-\ALG@thistlm}
\makeatother

\newtheorem{theorem}{Theorem}[section]

\newtheorem{lemma}[theorem]{Lemma}
\newtheorem{prop}[theorem]{Proposition}
\newtheorem{corollary}[theorem]{Corollary}
\theoremstyle{definition}
\newtheorem{remark}[theorem]{Remark}

\numberwithin{equation}{section}

\graphicspath{{.}}

\begin{document}

\begin{frontmatter}
\title{Adaptive distributed methods under communication constraints}
\runtitle{Adaptive distributed estimation}

\begin{aug}
\author{\fnms{Botond} \snm{Szab\'o}\thanksref{t1,t2}\ead[label=e1]{b.t.szabo@math.leidenuniv.nl}},
\and
\author{\fnms{Harry} \snm{van Zanten}\thanksref{t1}\ead[label=e2]{hvzanten@uva.nl}}

\thankstext{t1}{Research supported by the Netherlands Organization of Scientific Research NWO.}
\thankstext{t2}{The research leading to these results has received funding from the European Research Council under ERC Grant Agreement 320637.}
\runauthor{Szab\'o and Van Zanten}

\affiliation{Leiden University and University of Amsterdam}

\address{Mathematical Institute\\
Leiden University\\
Niels Bohrweg 1\\
2333 CA Leiden\\
The Netherlands\\
\printead{e1}\\
}

\address{Korteweg-de Vries Institute for Mathematics\\
University of Amsterdam\\
Science Park 107\\
1098 XG Amsterdam\\
The Netherlands\\
\printead{e2}}
\end{aug}

\begin{abstract}
We study distributed estimation methods under communication constraints in 
a distributed version of the nonparametric random design regression model. 
We derive minimax lower bounds and exhibit methods that attain those bounds. 
Moreover, we show that adaptive estimation is possible in this setting. 
\end{abstract}

%\begin{keyword}[class=MSC]
%\kwd[Primary ]{60K35}
%\kwd{60K35}
%\kwd[; secondary ]{60K35}
%\end{keyword}

%\begin{keyword}
%\kwd{sample}
%\kwd{\LaTeXe}
%\end{keyword}

\end{frontmatter}

\section{Introduction}

In this paper we take up the study of the fundamental possibilities and
limitations of distributed methods for high-dimensional, or nonparametric 
problems. The design and study of such methods has attracted substantial 
attention recently. This is for a large part motivated  by the ever increasing size of 
datasets, leading to the necessity to analyze data while distributed over
multiple machines and/or cores. Other reasons to consider distributed 
methods include privacy considerations or the simple fact that in some 
situations data are physically collected at multiple locations. 

By now a variety of methods are available for estimating nonparametric or
high-dimensional models to data in a distributed manner. A (certainly incomplete) 
list of recent references includes the papers \cite{scott:etal:2013,zhang:2013,kleiner:2014,deisenroth:2015,shang:cheng:2015,srivastava:2015a,Guhaniyogi:2017, battey:2018,jason:2017}. 
Some of these papers propose new methods, some study theoretical aspects of such methods,  and some do both.
The number of  theoretical papers on the fundamental performance of distributed methods
is still rather limited however. In the paper \cite{szabo:zanten:2017} we recently introduced
a distributed version of the canonical signal-in-white-noise model 
to serve as a benchmark model to study aspects like convergence rates
and optimal tuning of distributed methods. We used it to compare the 
performance of a number of distributed nonparametric methods recently
introduced in the literature. The study illustrated the intuitively obvious fact 
that in order to achieve an optimal bias-variance trade-off, or, equivalently, 
to find the correct balance between over- and under-fitting, distributed 
methods need to be tuned differently than methods that handle all data at once.
Moreover, our comparison showed that some of the proposed methods 
are more successful at this {than} others. 

A major challenge and fundamental question for nonparametric distributed 
methods is whether or not it is possible to achieve a form of adaptive
inference. In other words, whether we can design methods that do automatic, data-driven
tuning in order to achieve the optimal bias-variance trade-off. 
We illustrated by example in \cite{szabo:zanten:2017} that naively using 
methods that are known to achieve optimal adaptation in  non-distributed settings, 
can lead to sub-optimal performance in the distributed case. 
In the recent paper \cite{zhu:2018}, which considers the same distributed signal-in-white-noise
model and was written independently and at the same time as the present paper, 
it is in fact conjectured that adaptation in the considered particular distributed model is not possible.

In order to study convergence rates and adaptation for distributed methods
in a meaningful way the class of methods should be restricted somehow. 
Indeed, if there is no limitation on communication or computation, then we 
could simply communicate all data from the various local machines to a central machine, 
aggregate it, and  use some existing adaptive, rate-optimal procedure. 
In this paper we consider a setting in which the communication between the local and the global 
machines is restricted, much in the same way as the communication restrictions 
imposed in \cite{zhang:2013} in a parametric framework and recently in the 
simultaneously written paper \cite{zhu:2018}
in the context of the distributed signal-in-white-noise model we introduced in  \cite{szabo:zanten:2017}.

In the distributed nonparametric regression model with communication {constraints} 
that we consider we can derive
minimax lower bounds for the best possible rate that any distributed 
procedure can achieve under smoothness conditions on the true regression function.
Technically this essentially relies on an extension of the information 
theoretic approach of \cite{zhang:2013} to the infinite-dimensional setting (this is 
different from the approach taken in \cite{zhu:2018}, which relies on 
results from \cite{wang:2010}).
It turns out there are different regimes, depending on how much communication 
is allowed. On the one extreme end, and in accordance with intuition, 
if enough communication is allowed, it is possible to achieve the same 
convergence rates in the distributed setting as in the non-distributed case. 
The other extreme case is that there is so little communication allowed that 
combining different machines does not help. Then the optimal rate under the 
communication restriction can already be obtained by just using a single 
local machine and discarding the others. The interesting case is the 
intermediate regime. For that case we show there exists an optimal strategy 
that involves grouping the machines in a certain way and letting them 
work on different parts of the regression function.

These first results on rate-optimal distributed estimators are not adaptive, 
in the sense that the optimal procedures depend on the  regularity of the 
unknown regression function. The same holds true for the {procedure 
obtained} in parallel in \cite{zhu:2018}. 
In this paper we go a step further and show that 
contrary perhaps to intuition, and contrary
to the conjecture in \cite{zhu:2018},  adaptation is in fact possible.
Indeed, we exhibit in this paper an adaptive distributed method
which involves a very specific grouping of the local machines, in combination 
with a Lepski-type method that is carried out in the central machine.
We prove that the resulting distributed estimator
adapts to a range of smoothness levels of the unknown regression function and that, up to logarithmic factors, it attains the minimax lower bound.

Although our analysis is  theoretical, we believe it contains interesting
messages that are ultimately very relevant for the development of applied 
distributed methods in high-dimensional settings. 
First of all, we show that depending on the communication budget, it might be
advantageous to group local machines and let different groups work on 
different aspects of the high-dimensional object of interest. Secondly, 
we show that it is possible to have adaptation in communication restricted 
distributed settings, i.e.\
to have data-driven tuning that automatically achieves the correct bias-variance
trade-off. We note however that although our proof of this fact is constructive, 
the method we exhibit appears to be still rather unpractical. We view our adaptation result
primarily as a first proof of concept, that hopefully invites the development
of more practical adaptation techniques for distributed settings.
%\footnote{Maybe say that we can handle $L_2$ and $L_{\infty}$ as well... Or perhaps not here...}

\subsection{Notations}
For two positive sequences $a_n,b_n$ we use the notation $a_n\lesssim b_n$ if there exists an universal positive constant $C$ such that $a_n\leq C b_n$. Along the lines $a_n\asymp b_n$ denotes that $a_n\lesssim b_n$ and $b_n\lesssim a_n$ hold simultaneously. Furthermore we write $a_n\ll b_n$ if $a_n/b_n=o(1)$. Let us denote by $\lceil a\rceil$ and $\lfloor a \rfloor $ the upper and lower integer value of the real number $a$, respectively. The sum $\sum_{i=a}^b x_j$ for $a,b$ positive real number denotes the sum $\sum_{i\in\mathbb{N}: a\leq i \leq b}x_j$. For a set $A$ let $|A|$ denote the size of the set. For $f\in L_2[0,1]$ we denote the standard $L_2$-norm as $\|f\|_2^2=\int_0^1 f(x)^2dx$, while for bounded functions $\|f\|_{\infty}$ denotes the $L_\infty$-norm. The function $\sign:\, \mathbb{R}\mapsto \{0,1\}$ evaluates to {$0$}  on $(-\infty, 0)$ 
and {$1$} on $[0,\infty)$. Furthermore, we use the notation 
$\text{mean}\{a_1,\ldots,a_n\}=(a_1+\ldots+a_n)/n$. Throughout the paper, $c$ and $C$ denote global constants whose value may change one line to another.

%We note that in the paper in many instances (when it is clear from the context) we omit the notation $\lceil .\rceil$. This happens typically in cases when we are talking about number of servers or amount of transmitted bits which are of course natural numbers.

\section{Main results}
We work  with the distributed version of the 
random design regression model. 
We assume that we have $m$ `local' machines and in the $i$th machine we observe pairs of random variables $(T^{(i)}_{\ell},X^{(i)}_{\ell})$, $\ell=1,...,n/m$, (with $n/m\in\mathbb{N}$) satisfying 
\begin{align}\label{model: regression}
&X^{(i)}_\ell=f_0(T^{(i)}_\ell)+\sigma \eps_{\ell}^{(i)},\quad \text{where}\\
& T^{(i)}_{\ell}\stackrel{iid}{\sim} U(0,1), \eps_{\ell}^{(i)}\stackrel{iid}{\sim}N(0,1),\quad \ell=1,...,n/m,\,i=1,...,m,\nonumber
\end{align}
and $f_0 \in L_2[0,1]$ (which is 
the same for all machines) is the unknown functional parameter of interest. 
We denote the local distribution and  expectation corresponding to the $i$th machine in \eqref{model: regression} by $\mathbb{P}_{f_0,T}^{(i)}$ and $\mathbb{E}_{f_0,T}^{(i)}$, respectively, and the joint distribution and expectation over all machines $i=1,...,m$, by $\mathbb{P}_{f_0,T}$ and $\mathbb{E}_{f_0,T}$, respectively. 
We assume that the total sample size $n$ is known to every local machine.
For our theoretical results we will assume that the unknown true function $f_0$ belongs to some regularity class. We work in our analysis with Besov smoothness classes, more specifically we assume that 
for some degree of smoothness $s> 0$ we have 
$f_0\in B_{2,\infty}^s(L)$ or $f_0\in B_{\infty,\infty}^s(L)$. The first class is of Sobolev type, while the second one is of  H\"older type with minimax estimation rates $n^{-s/(1+2s)}$ and $(n/\log n)^{-s/(1+2s)}$, respectively. For precise definitions, see  Section \ref{sec: wavelets} in the supplementary material \cite{suppl: szabo:zanten:2018}. Each local machine carries out (parallel to the others) a local statistical procedure and transmits the results to a central machine, which produces an estimator for  the signal $f_0$ by somehow 
aggregating the messages received from the local machines.

We study these distributed procedures under communication constraints between the local machines 
and the central machine. We allow each local 
machine to send at most $B^{(i)}$ bits on average to the central machine. More formally, a distributed estimator $\hat f$ is a measurable function of 
$m$ binary strings, or messages,  passed down from the local machines to the central machine. We denote 
by $Y^{(i)}$ the finite binary string  transmitted from machine $i$ to the central machine, 
which is a measurable function of the local data $X^{(i)}$.
 For
a class of potential signals {$\FF \subset L_2[0,1]$}, we restrict the communication between the 
machines by assuming that for numbers $B^{(1)}, \ldots, B^{(m)}$, it holds that 
 $\mathbb{E}_{f_0,T}[l({Y^{(i)}} )]\leq B^{(i)}$ for every $f_0\in\mathcal{F}$ and 
 $i=1,\ldots,m$, where $l(Y)$ denotes the length of the string $Y$.  We denote the resulting class of communication restricted distributed estimators $\hat f$ by $\mathcal{F}_{dist}(B^{(1)},\ldots,B^{(m)};\mathcal{F})$.
The number of machines $m$ and the communication constraints $B^{(i)}$ 
are allowed to depend on the overall sample size $n$, in fact that is 
the interesting situation. To alleviate the notational burden somewhat we do not make 
this explicit in the notation however.

\subsection{Distributed minimax lower bounds for the $L_2$-risk} 
The first theorem we present gives a minimax lower bound for 
distributed procedures for the $L_2$-risk, uniformly over Sobolev-type 
Besov balls, see Section \ref{sec: wavelets} in the supplement for rigorous definitions.

\begin{theorem}\label{theorem: minimaxL2LB}
Consider   $s,L> 0$, $\log_2 n\leq m=O(n^{\frac{2s}{1+2s}}/\log^2 n)$ and communication constraints $B^{(1)}, \ldots, B^{(m)} > 0$.
Let the sequence $\delta_n=o(1)$ be defined as the solution to the equation
\begin{align}
\delta_n=\min\Big\{\frac{m}{n \log_2 n }, \frac{m}{n \sum_{i=1}^m[ (\log_2(n)\delta_n^{\frac{1}{1+2s}} B^{(i)}) \wedge 1]} \Big\}.\label{def: delta_n}
\end{align}
Then in distributed random design nonparametric regression model \eqref{model: regression} we have that
\begin{align*}
\inf_{\hat{f}\in\mathcal{F}_{dist}(B^{(1)},\ldots,B^{(m)};B_{2,\infty}^s(L))}\, \sup_{f_0\in B_{2,\infty}^s(L)} \mathbb{E}_{f_0,T}\|\hat{f}-f_0\|_2^2 \gtrsim L^{\frac{2}{1+2s}} \delta_n^{\frac{2s}{1+2s}}.
\end{align*}
\end{theorem}

\begin{proof}
See Section \ref{sec: minimaxL2LB}
\end{proof}

%\begin{corollary}
%\end{corollary}

We briefly comment on the derived result. First of all note that the quantity $\delta_n$ in \eqref{def: delta_n} is well defined, since the left-hand side of the equation is increasing, while the right-hand side is decreasing in $\delta_n$. The proof of the theorem is based on an application of a version of Fano's inequality, frequently used to derive minimax lower bounds. 
More specifically, as a first step we find as usual  a large enough finite subset of the functional space $B_{2,\infty}^s(L)$ over which the minimax rate is the same as over the whole space. This is  done by finding the `effective resolution level' $j_n$ in the wavelet representation of the 
function of interest and perturbing the corresponding wavelet coefficients, while setting the rest of the coefficients to zero. This effective resolution level for $s$-smooth functions is usually $({1+2s})^{-1}\log_2 n$ in case of the $L_2$-norm for non-distributed models
(e.g. \cite{gine:nickl:2016}). However, in our distributed setting the effective {resolution level} changes to 
$({1+2s})^{-1}\log \delta_n^{-1}$, which can be substantially different from the non-distributed case, as it strongly depends on the number of transmitted bits. 
The dependence on the expected number of transmitted bits enters the formula by using a variation 
of  Shannon's source coding theorem. Many of the information theoretic manipulations in the proof are an extended and adapted version of the approach introduced in \cite{zhang:2013}, where similar results were derived in context of distributed methods with communication constraints over parametric models.

To understand the result it is illustrative to consider the special 
case that the communication constraints are the same for all machines, i.e. $B^{(1)}
=\cdots=B^{(m)} = B$ for some $B>0$. We can then  distinguish three regimes: \ref{(i)} the case 
$B\geq n^{1/(1+2s)}/\log_2 n$; \ref{(ii)} the case  $(n\log_2 (n)/m^{2+2s})^{{1}/({1+2s})}\leq B <  n^{1/(1+2s)}/\log_2 n$; and \ref{(iii)} the case $B < (n\log_2 (n)/m^{2+2s})^{{1}/({1+2s})}$. 

In regime \ref{(i)}  we have a large  communication budget and by elementary computations we get that the minimum in \eqref{def: delta_n}  is taken in the second fraction 
and hence that $\delta_n=1/n$. This means that in this case the derived lower bound corresponds to the usual non-distributed minimax rate $n^{-2s/(1+2s)}$. 
In the other extreme case, regime \ref{(iii)}, the minimum is taken at the first term in  \eqref{def: delta_n} and $\delta_n=m/(n\log_2 n)$, so the lower bound is of the order $(n\log_2(n)/m)^{-2s/(1+2s)}$. This rate is, up to the $\log_2 n$ factor, equal to the minimax rate corresponding to the sample size $n/m$. Consequently, in this case it does not make sense to consider distributed methods, since by just using a single machine the best rate
can already be obtained (up to a logarithmic factor).
In  the intermediate case \ref{(ii)} it is straightforward to see that 
$\delta_n= (nB\log_2 n)^{({1+2s})/({2+2s})}$. It follows that if $B=o(n^{1/(1+2s)}/\log_2 n)$, 
i.e.\ if we are only allowed to communicate  `strictly' less than in case \ref{(i)}, 
then  the  lower bound is strictly worse  than the minimax rate corresponding to the non-distributed setting.

The findings above are summarized in the following corollary.

\begin{corollary}\label{cor: minimaxLB}
Consider   $s, L> 0$, a communication constraints $B^{(1)}=....=B^{(m)}=B> 0$ and assume that $\log_2 n \le m=O(n^{\frac{2s}{1+2s}}/\log^2 n)$. Then
\begin{enumerate}[label=(\roman*)]
\item if $B\geq n^{1/(1+2s)}/\log_2 n$, \label{(i)}
\begin{align*}
\inf_{\hat{f}\in\mathcal{F}_{dist}(B,\ldots,B; B_{2,\infty}^{s}(L))}\,
\sup_{f_0\in B_{2,\infty}^s(L)} \mathbb{E}_{f_0,T}\|\hat{f}-f_0\|_2^2 \gtrsim L^{\frac{2}{1+2s}}n^{-\frac{2s}{1+2s}};
\end{align*}
\item
if  $(n\log_2 (n)/m^{2+2s})^{{1}/({1+2s})}\leq B <  n^{1/(1+2s)}/\log_2 n$, \label{(ii)}
\begin{align*}
\inf_{\hat{f}\in\mathcal{F}_{dist}(B,\ldots,B; B_{2,\infty}^{s}(L))}\,
\sup_{f_0\in B_{2,\infty}^s(L)} \mathbb{E}_{f_0,T}\|\hat{f}-f_0\|_2^2 \gtrsim   L^{\frac{2}{1+2s}}\Big(\frac{n^{1/(1+2s)}}{B\log_2 n}\Big)^{\frac{2s}{2+2s}} n^{-\frac{2s}{1+2s}};
\end{align*}
\item
if $(n\log_2 (n)/m^{2+2s})^{{1}/({1+2s})}> B$, \label{(iii)}
\begin{align*}
\inf_{\hat{f}\in\mathcal{F}_{dist}(B,\ldots,B; B_{2,\infty}^{s}(L))}\, \sup_{f_0\in B_{2,\infty}^s(L)} 
\mathbb{E}_{f_0,T}\|\hat{f}-f_0\|_2^2 \gtrsim  L^{\frac{2}{1+2s}} \Big(\frac{n\log_2 n}{m}\Big)^{-\frac{2s}{1+2s}}.
\end{align*}
\end{enumerate}
\end{corollary}

\subsection{Non-adaptive rate-optimal distributed procedures for $L_2$-risk}

Next we show that the derived lower bounds are sharp
by exhibiting distributed procedures that attain the bounds (up to logarithmic factors).
We note that it is sufficient to consider only the case 
$B \geq (n\log_2 (n)/m^{2+2s})^{{1}/({1+2s})}$, since otherwise distributed techniques do not perform better than standard techniques carried out on one of the local servers. In case \ref{(iii)} therefore one would probably prefer to use a single local machine instead of a complicated distributed method with (possibly) worse performance.

As a first step let us consider Daubechies wavelets $\psi_{jk}$, $j=0,...$,  $k=0,1,...,2^j-1$ with at least $s$ vanishing moments (for details, see Section \ref{sec: wavelets} in the supplement). Then let us estimate the wavelet coefficients of the underlying function $f_0$ in each local problems, i.e. for every $j=0,...,$ and $k=0,1,...,2^j-1$ let us construct
\begin{align*}
\hat{f}_{jk}^{(i)}=\frac{m}{n}\sum_{\ell=1}^{n/m}X_\ell^{(i)}\psi_{jk}(T^{(i)}_\ell) 
\end{align*}
and note that
\begin{align*}
\mathbb{E}_{f_0,T}\hat{f}_{jk}^{(i)}=\int_{0}^1 f_0(t)\psi_{jk}(t)dt=f_{0,jk}.
\end{align*}
Since one can only transmit finite amount of bits we have to approximate the estimators of the wavelet coefficients. Let us take an arbitrary $x\in\mathbb{R}$ and write it in a scientific binary representation, i.e. $|x|=\sum_{k=-\infty}^{\log_2|x|} b_k2^k$, with $b_k\in\{0,1\}$, $k\in \mathbb{Z}$. Then let us take $y$ consisting the same digits as $x$ up to the $(D\log_2 n)th$ digits. for some $D>0$, after the binary dot (and truncated there), i.e. $|y|=\sum_{k=-D\log_2 n}^{\log_2|x|} b_k2^k$, see also Algorithm \ref{alg: transmit:number}.

\begin{algorithm}
\caption{Transmitting a finite-bit approximation of a number}\label{alg: transmit:number}
\begin{algorithmic}[1]
\Procedure{TransApprox($x$)}{}
\State \textit{Transmit:} $\sign(x)$, $b_{-\lfloor D\log_2 n\rfloor},...,b_{\lfloor\log_2|x|\rfloor} $.
\State \textit{Construct:} $y=(2\sign(x)-1)\sum_{k= -D\log_2 n}^{\log_2|x|}b_k2^{k}$.
\EndProcedure
\end{algorithmic}
\end{algorithm}

Observe that the length of $y$ (viewed
as a binary string) is bounded from above by $1+(1\vee \log_2 |x|)+D\log_2 n$ bits.
The following lemma asserts that if $\mathbb{E}(1\vee \log_2|X|)=o(\log_2 n)$, then the expected length  $\EE [l(Y)]$ of the  constructed binary string approximating $X$ is less than $\log_2 n$ (for sufficiently large $n$ and by choosing $D=1/2$)  and the approximation is sufficiently close to $X$.
\begin{lemma}\label{lem: approx2}
Assume that $\EE(1\vee \log_2  |X|)=o(\log_2  n)$. Then the approximation $Y$ of $X$ given in Algorithm \ref{alg: transmit:number} satisfies that
\begin{align*}
0\leq |X-Y|\leq n^{-D}\quad \text{and}\quad
\EE[ l(Y)]\leq (D+o(1))\log_2(n).
\end{align*}
%Furthermore we also have that
%$$\PP\big( l(Y)\leq 1+(D+1)\log_2 n\big)\geq 1-cn^{-3/2},$$
%for some large enough constant $c>0$.
\end{lemma}
\begin{proof}
See Section \ref{sec: approx}.
\end{proof}

After these preparations we can exhibit  procedures  attaining (nearly) the 
theoretical limits obtained in Corollary \ref{cor: minimaxLB}.

We first consider the case  \ref{(i)} that $B\geq n^{1/(1+2s)}/\log_2 n$. 
In this case each local machine $i=1,\ldots,m$ transmits the approximations $Y^{(i)}_{jk}$ (given in Algorithm \ref{alg: transmit:number}) of the first $n^{1/(1+2s)}\wedge (B/\log_2 n)$  wavelet coefficients $\hat{f}_{jk}^{(i)}$, i.e. for $2^j+k\leq n^{1/(1+2s)}\wedge (B/\log_2 n)$. Then in the central machine we simply average  the  transmitted approximations to obtain the estimated wavelet coefficients 
\begin{equation*}
\hat{f}_{jk}=
\begin{cases}
\frac{1}{m}\sum_{i=1}^m Y^{(i)}_{jk}, &\text{if $2^j+k\leq n^{1/(1+2s)}\wedge (B/\log_2 n)$},\\
0,& \text{else.}
\end{cases}
\end{equation*}
The final estimator $\hat f$ for $f_0$ is the function in $L_2[0,1]$ 
with these wavelet coefficients, i.e.\
$\hat f = \sum \hat f_{jk}\psi_{jk}$.
The  method is summarized as Algorithm \ref{alg: nonadapt:L2:case1} below.

\begin{algorithm}

\caption{Nonadaptive $L_2$-method, case \ref{(i)}}\label{alg: nonadapt:L2:case1}
\begin{algorithmic}[1]
\BState  \textbf{In the local machines:}
\For {$i=1$ to $m$}:
\For {$ 2^j+k=$  1 to $n^{1/(1+2s)}\wedge (B/\log_2 n)$}
\State \text{$Y_{jk}^{(i)}$ :=TransApprox($\hat{f}_{jk}^{(i)}$)}
\EndFor
\EndFor
\BState \textbf{In the central machine}:
\For {$2^j+k=1$ to $n^{1/(1+2s)}\wedge (B/\log_2 n)$}
\State \text{$\hat{f}_{jk}:= mean\{ Y_{jk}^{(i)}:\, 1\leq i\leq m\}$.}
\EndFor
\State Construct: $\hat f = \sum \hat f_{jk}\psi_{jk}$.
\end{algorithmic}
\end{algorithm}

We note again that the procedure outlined in Algorithm \ref{alg: nonadapt:L2:case1}
is just a simple averaging, sometimes called ``divide and conquer'' or ``embarrassingly
parallel'' in the learning literature (e.g.\ \cite{zhang2013}, \cite{rosen2016}). The following theorem asserts that the constructed estimator indeed attains 
the lower bound in case \ref{(i)} (up to a logarithmic factor for $B$ close to the threshold).

\begin{theorem}\label{theorem: minimaxL2UB1}
Let $s, L > 0$, $m\leq n$, and suppose that  $B\geq n^{1/(1+2s)}/\log_2 n$. 
Then  the distributed estimator $\hat f$ described in Algorithm 
\ref{alg: nonadapt:L2:case1} belongs to $\mathcal{F}_{dist}(B,\ldots,B; B_{2,\infty}^{s}(L))$ 
and satisfies
\begin{align*}
\sup_{f_0\in B_{2,\infty}^{s}(L),\,\|f_0\|_{\infty}\leq M}\mathbb{E}_{f_0,T}\|\hat{f}-f_0\|_2^2\lesssim L^{\frac{2}{1+2s}} \big(n^{-\frac{2s}{1+2s}})\vee (B/\log_2 n)^{-2s}\big).
\end{align*}
\end{theorem}

\begin{proof}
See Section \ref{sec: minimaxL2UB}
\end{proof}

Next we consider the case \ref{(ii)} of Corollary \ref{cor: minimaxLB}, i.e.\ 
the case that the communication restriction satisfies  $(n\log_2 (n)/m^{2+2s})^{{1}/({1+2s})}\leq   B<
n^{{1}/({1+2s})}/\log_2 n$. For technical reasons we also assume that $B\geq \log_2 n$. Using Algorithm \ref{alg: nonadapt:L2:case1} in this case 
would result in a highly sub-optimal procedure, as we prove at the end of Section \ref{sec: minimaxL2UB2}. It turns out that under this more severe communication 
restriction we can do much better if we form different groups of machines 
that work on different parts of the signal.

We introduce the notation $\eta = \lfloor (L^2n)^{\frac{1}{2+2s}}\big((\log_2 n)/B\big)^{\frac{1+2s}{2+2s}}\rfloor\wedge m$. Then  we group the local machines into $\eta$ groups 
and let the different groups work on different parts of wavelet domain as follows:
the  machines with numbers $1\leq i\leq  m/\eta$ each transmit the approximations $Y_{jk}^{(i)}$ of the estimated wavelet coefficients $\hat{f}_{jk}^{(i)}$ for $1\leq 2^{j}+k\leq \lfloor B/\log_2 n\rfloor$; the next machines, 
with numbers $m/\eta< i\leq 2 m/\eta$, each transmit the approximations $Y_{jk}^{(i)}$
for $\lfloor B/\log_2 n\rfloor< 2^{j}+k\leq 2 \lfloor B/\log_2 n\rfloor$,  and so on.  The last machines with numbers 
$(\eta-1) m/\eta< i\leq  m $  transmit the $Y_{jk}^{(i)}$ for $(\eta-1) \lfloor B/\log_2 n\rfloor< 2^{j}+k\leq \eta \lfloor B/\log_2 n\rfloor$. Then in the central machine we average the corresponding transmitted noisy coefficients in the obvious way. Formally, using the notation $\mu_{jk}= \big\lceil(2^{j}+k)\lfloor B/\log_2 n\rfloor^{-1}\big\rceil-1$,  the aggregated estimator $\hat f$ is the function with wavelet 
coefficients given by 
\begin{align*}
\hat{f}_{jk}=
\begin{cases}
mean\{Y_{jk}^{(i)}:\, \frac{\mu_{jk} m}{\eta} < i \leq \frac{(\mu_{jk}+1) m}{\eta}\},& \text{if $2^{j}+k\leq \eta\lfloor B/\log_2 n\rfloor$},\\
0, & \text{else}.
\end{cases}
\end{align*}
The procedure is summarized as Algorithm \ref{alg: nonadapt:L2:case2}.

\begin{algorithm}
\caption{Nonadaptive $L_2$-method, case \ref{(ii)}}\label{alg: nonadapt:L2:case2}
\begin{algorithmic}[1]
\BState \textbf{In the local machines}:
\For{$\ell= 1$ to $\eta$ }
\For {$i =\lfloor (\ell-1) m/\eta\rfloor +1$ to $\lfloor\ell m/\eta\rfloor$}
\For {$ 2^{j}+k=  (\ell-1) \lfloor B/\log_2 n\rfloor+1$ to $\ell\lfloor B/\log_2 n\rfloor$}
\State \text{$Y_{jk}^{(i)}$ :=TransApprox($\hat{f}_{jk}^{(i)}$)}
\EndFor
\EndFor
\EndFor
\BState \textbf{In the central machine}:
\For {$ 2^j+k=1$ to $\eta\lfloor B/\log_2 n\rfloor$}
\State \text{$\hat{f}_{jk}:=mean\{Y_{jk}^{(i)}:\, \mu_{jk}  m/\eta<  i \leq (\mu_{jk}+1)  m/\eta \}$}
\EndFor
\State Construct: $\hat f = \sum \hat f_{jk}\psi_{jk}$.
\end{algorithmic}
\end{algorithm}

The following theorem asserts that this estimator  attains 
the lower bound in case \ref{(ii)} (up to a logarithmic factor).

\begin{theorem}\label{theorem: minimaxL2UB2}
Let $s, L > 0$, $m\leq n$ and suppose that  $(n\log_2 (n)/m^{2+2s})^{{1}/({1+2s})}\vee \log_2 n \leq B
< n^{1/(1+2s)}/\log_2 n$. Then the distributed estimator $\hat f$ described 
 in Algorithm \ref{alg: nonadapt:L2:case2} belongs to 
 $\mathcal{F}_{dist}(B,\ldots,B; B_{2,\infty}^{s}(L))$  and satisfies 
\begin{align*}
\sup_{f_0\in B_{2,\infty}^{s}(L),\,\|f_0\|_{\infty}\leq M}\mathbb{E}_{f_0,T}\| \hat{f}_n-f_0\|_2^2\lesssim  M_n \Big(\frac{n^{1/(1+2s)}}{B\log_2 n}\Big)^{\frac{2s}{2+2s}}n^{-\frac{2s}{1+2s}},
\end{align*}
with $M_n=L^{\frac{4}{2+2s}}(\log_2n)^{2s} $.
\end{theorem}

\begin{proof}
See Section \ref{sec: minimaxL2UB2}
\end{proof}

%\begin{remark}
%Instead of diving the machines into groups and assign different tasks for them one could uniformly randomly select the transmitted coefficients from the first $\eta B/\log_2 n$ noisy wavelet coefficient. We conjecture that this procedure attains the same upper bound as the above described method. \textcolor{red}{[Question: Shall we prove this?]}
%\end{remark}

%We can conclude that the optimal number of transmitted bits are $n^{\frac{1}{1+2s}}$ up to a logarithmic factor. Transmitting less bits will result in (possibly polynomially) sub-optimal convergence rate for any distributed methods, while by transmitting at least the optimal amount of bits one can construct distributed estimators reaching (up to a multiplicative constant) the convergence rate of non-distributed techniques.

\subsection{Distributed minimax results for $L_\infty$-risk}

When we replace the $L_2$-norm by the  $L_{\infty}$-norm and correspondingly 
change the type of Besov balls we consider, we can derive 
a lower bound similar to Theorem \ref{theorem: minimaxL2LB} (see Section \ref{sec: wavelets} in the supplement for the rigorous definition of Besov balls).

\begin{theorem}\label{theorem: minimaxLinftyLB}
Consider   $s, L> 0$, communication constraints $B^{(1)}, \ldots, B^{(m)} > 0$ and assume that $\log_2n \le m=O(n^{\frac{2s}{1+2s}}/\log^2 n)$.
Let the sequence $\delta_n=o(1)$ be defined as the solution to the equation
\eqref{def: delta_n}.
Then in the distributed random design regression model \eqref{model: regression} we have that
\begin{align*}
\inf_{\hat{f}\in\mathcal{F}_{dist}(B^{(1)},\ldots,B^{(m)}; B_{\infty,\infty}^s(L))}
\, \sup_{f_0\in  B_{\infty,\infty}^s(L)} \mathbb{E}_{f_0,T}\|\hat{f}-f_0\|_{\infty} \gtrsim  \Big(\frac{n}{\log n}\Big)^{-\frac{s}{1+2s}} \vee \delta_n^{\frac{s}{1+2s}}.
\end{align*}
\end{theorem}

\begin{proof}
See Section \ref{sec: minimaxLinftyLB}
\end{proof}

The proof of the theorem is very similar to the proof of Theorem \ref{theorem: minimaxLinftyLB}. The first term on the right hand side follows from the usual non-distributed  
minimax  lower bound. For the second term we use the standard version of Fano's inequality. We again consider a large enough finite subset of $B_{\infty,\infty}^s{(L)}$. The effective resolution level for the $L_{\infty}$-norm in the non-distributed case is $(1+2s)^{-1}\log_2(n/\log_2 n)$. Similarly to the $L_2$ case the effective resolution level changes to $(1+2s)^{-1}\log\delta_n^{-1}$ in the distributed setting, which can be again substantially different from the non-distributed case. The rest of the proof follows the same lines reasoning as the proof of Theorem \ref{theorem: minimaxLinftyLB}. 

Similarly to the $L_2$-norm we consider again the specific case where all communication budgets are taken to be equal, i.e. $B^{(1)}=B^{(2)}=...=B^{(m)}=B$. One can easily see that there are again three regimes of $B$ (slightly different compared to the $L_2$-case).

\begin{corollary}\label{cor: minimaxLB_Linfty}
Consider   $s, L> 0$,  communication constraint $B^{(1)}=...=B^{(m)}=B> 0$ and assume that $\log_2 n \le m=O(n^{\frac{2s}{1+2s}}/\log^2 n)$.
\begin{enumerate}[label=(\roman*b)]
\item
If $B\geq \big(n/(\log_2 n)^{3+4s}\big)^{1/(1+2s)}$, then  \label{(ib)}
\begin{align*}
\inf_{\hat{f}\in\mathcal{F}_{dist}(B,\ldots,B; B_{\infty,\infty}^{s}(L))}\,
\sup_{f_0\in B_{\infty,\infty}^s(L)} \mathbb{E}_{f_0,T}\|\hat{f}-f_0\|_{\infty} \gtrsim (n/\log_2 n)^{-\frac{s}{1+2s}}.
\end{align*}
\item
If  $(n\log_2 (n)/m^{2+2s})^{{1}/({1+2s})}\leq B <  \big(n/(\log_2 n)^{3+4s}\big)^{1/(1+2s)}$, then \label{(iib)}
\begin{align*}
\inf_{\hat{f}\in\mathcal{F}_{dist}(B,\ldots,B; B_{\infty,\infty}^{s}(L))}\,
\sup_{f_0\in B_{\infty,\infty}^s(L)} \mathbb{E}_{f_0,T}\|\hat{f}-f_0\|_{\infty} \gtrsim   \big(\frac{n^{\frac{1}{1+2s}}}{B(\log_2 n)^{\frac{3+4s}{1+2s}}}\big)^{\frac{s}{2+2s}} (\frac{n}{\log_2 n})^{-\frac{s}{1+2s}}.
\end{align*}
\item
If $(n\log_2 (n)/m^{2+2s})^{{1}/({1+2s})}> B$\label{(iiib)}, then
\begin{align*}
\inf_{\hat{f}\in\mathcal{F}_{dist}(B,\ldots,B; B_{\infty,\infty}^{s}(L))}\, \sup_{f_0\in B_{\infty,\infty}^s(L)} 
\mathbb{E}_{f_0,T}\|\hat{f}-f_0\|_{\infty} \gtrsim   \Big(\frac{n\log_2 n}{m}\Big)^{-\frac{s}{1+2s}}.
\end{align*}
\end{enumerate}
\end{corollary}

Next we provide matching upper bounds (up to a $\log n$ factor) in the first two cases, i.e. \ref{(ib)} and \ref{(iib)}. In the third case the lower bound matches (up to a logarithmic factor) the minimax rate corresponding to a single local machine, hence it is not advantageous at all to develop complicated distributed techniques as a single server with only fraction of the total information performs at least as well. In the previous section dealing with $L_2$ estimation we have provided two algorithms (one where the machines had the same tasks and one where the machines were divided into groups and were assigned different tasks) to highlight the differences between the cases. Here for simplicity we combine the algorithms to a single one, but essentially the same techniques are used as before. 

In each local machine we compute the local estimators of the wavelet coefficients $\hat{f}_{jk}^{(i)}$ and transmit a finite digit approximations of them $Y_{jk}^{(i)}$, as in the $L_2$-case. Then let us divide the machines into $\eta = \lfloor\big(n(\log_2 n)^{2s}/B^{1+2s}\big)^{\frac{1}{2+2s}}\rfloor \wedge m\vee 1$ equal sized groups ($\eta=1$ corresponds to case \ref{(ib)}, while $\eta>1$ corresponds to case \ref{(iib)}). Similarly to before machines with numbers $1\leq i\leq  m/\eta$ transmit the approximations $Y_{jk}^{(i)}$ of the estimated wavelet coefficients $\hat{f}_{jk}^{(i)}$ for $1\leq 2^{j}+k\leq \lfloor B/\log_2 n\rfloor\wedge (n/\log_2 n)^{\frac{1}{1+2s}}$, and so on, the last machines with numbers 
$(\eta-1) m/\eta< i\leq  m $  transmit the approximations $Y_{jk}^{(i)}$ for $(\eta-1)\lfloor B/\log_2 n\rfloor\wedge (n/\log_2 n)^{\frac{1}{1+2s}}< 2^{j}+k\leq \eta\lfloor B/\log_2 n\rfloor\wedge (n/\log_2 n)^{\frac{1}{1+2s}}$. In the central machine we average the corresponding transmitted coefficients in the obvious way, i.e.  the aggregated estimator $\hat f$ is the function with wavelet 
coefficients given by 
\begin{align*}
\hat{f}_{jk}=
\begin{cases}
mean\{Y_{jk}^{(i)}:\, \frac{\mu_{jk} m}{\eta} < i \leq \frac{(\mu_{jk}+1) m}{\eta}\},& \text{if $2^{j}+k\leq \eta \lfloor\frac{ B}{\log_2 n}\rfloor\wedge (\frac{n}{\log n})^{\frac{1}{1+2s}}$},\\
0, & \text{else},
\end{cases}
\end{align*}
where $\mu_{jk}= \big\lceil(2^{j}+k) \lfloor B/\log_2 n\rfloor^{-1} \big\rceil-1$.
The procedure is summarized as Algorithm \ref{alg: nonadapt:Linfty} and the (up to a logarithmic factor) optimal behaviour is given in 
Theorem \ref{theorem: minimaxLinftyUB} below.

\begin{algorithm}
\caption{Nonadaptive $L_\infty$-method, combined}\label{alg: nonadapt:Linfty}
\begin{algorithmic}[1]
\BState \textbf{In the local machines}:
\For{$\ell= 1$ to $\eta$ }
\For {$i =\lfloor (\ell-1) m/\eta\rfloor +1$ to $\lfloor\ell m/\eta\rfloor$}
\For {$ 2^{j}+k= (\ell-1)\lfloor  B/\log_2 n\rfloor+1$ to $\ell \lfloor B/\log_2 n\rfloor$}
\State \text{$Y_{jk}^{(i)}$ :=TransApprox($\hat{f}_{jk}^{(i)}$)}
\EndFor
\EndFor
\EndFor
\BState \textbf{In the central machine}:
\For {$ 2^j+k=1$ to $\eta\lfloor B/\log_2 n\rfloor$}
\State \text{$\hat{f}_{jk}:=mean\{Y_{jk}^{(i)}:\, \mu_{jk}  m/\eta<  i \leq (\mu_{jk}+1)  m/\eta \}$}
\EndFor
\State Construct: $\hat f = \sum \hat f_{jk}\psi_{jk}$.
\end{algorithmic}
\end{algorithm}

\begin{theorem}\label{theorem: minimaxLinftyUB}
Let $s, L > 0$. Then the distributed estimator $\hat f$ described 
 in Algorithm \ref{alg: nonadapt:Linfty} belongs to 
 $\mathcal{F}_{dist}(B,\ldots,B; B_{\infty,\infty}^{s}(L))$  and satisfies
\begin{itemize}
\item for $B\geq n^{1/(1+2s)}(\log_2 n)^{2s/(1+2s)}$,
\begin{align*}
\sup_{f_0\in B_{\infty,\infty}^{s}(L)}\mathbb{E}_{f_0,T}\| \hat{f}_n-f_0\|_\infty\lesssim (n/\log_2 n)^{-\frac{s}{1+2s}};
\end{align*}
\item for $\big(n(\log_2 n)/m^{2+2s}\big)^{{1}/({1+2s})} \vee \log_2 n\leq B< n^{1/(1+2s)}(\log_2 n)^{2s/(1+2s)}$,
\begin{align*}
\sup_{f_0\in B_{\infty,\infty}^{s}(L)}\mathbb{E}_{f_0,T}\| \hat{f}_n-f_0\|_\infty\lesssim  M_n\Big(\frac{n^{\frac{1}{1+2s}} }{B(\log_2 n)^{\frac{3+4s}{1+2s}}}\Big)^{\frac{s}{2+2s}}(n/\log_2 n)^{-\frac{s}{1+2s}},
\end{align*}
with $M_n=(\log_2 n)^{s\vee \frac{3s}{2+2s}}$.
\end{itemize}
\end{theorem}

\begin{proof}
See Section \ref{sec: minimaxLinftyUB}
\end{proof}

We can draw similar conclusions for the $L_{\infty}$-norm as for the $L_{2}$-norm. If we do not transmit a sufficient amount of bits (at least $n^{1/(1+2s)}$ up to a $\log n$ factor) from the local machines to the central one then the lower bound from the theorem exceeds the minimax risk corresponding the non-distributed case. Furthermore by transmitting the sufficient amount of bits (i.e. $n^{1/(1+2s)}$ up to a $\log n$ factor) corresponding to the class $B_{\infty,\infty}^s(L)$, the lower bound will coincide with the non-distributed  minimax estimation rate.

\subsection{Adaptive distributed estimation}\label{sec: adapt:GWN}
The (almost) rate-optimal procedures  considered so far 
 have in common that they are non-adaptive, in the sense that they all use the knowledge of the regularity level $s$ of the unknown functional parameter of interest. 
In this section we  exhibit a distributed algorithm attaining the lower bounds 
(up to a logarithmic factor) across a whole range of regularities $s$ simultaneously. 
In the non-distributed setting it is well known that this is possible, and 
many adaptation methods exist, including for instance the block Stein method, Lepski's method, wavelet thresholding, 
and Bayesian adaptation methods, just to mention but a few (e.g. \cite{tsybakov:2009,gine:nickl:2016}). 
In the distributed case the matter is more complicated. Using the usual adaptive 
tuning methods in the local machines will typically not work (see \cite{szabo:zanten:2017})
and in fact it was recently conjectured that adaptation, if at all possible,  
would require more communication than is allowed in our model (see \cite{zhu:2018}).

We will show, however, that in our setting, if all machines have the same communication restriction 
given by $B\geq \log_2 n$, it is possible to adapt to regularities
$s$ ranging in the interval $[s_{\min}, s_{\max})$, where 
\begin{align}
s_{\min}=\arg\inf_{s>0}\lim\inf_n \Big\{(n(\log_2 n)^2/m^{2+2s})^{\frac{1}{1+2s}}\leq B\Big\}\label{def: smin}
\end{align}
and  $s_{\max}$ is the regularity of the considered Daubechies wavelet and can be chosen arbitrarily large. Note  that $s_{\min}$ is well defined. 
If $s\in [s_{\min}, s_{\max})$, then we are in one of the non-trivial cases 
\ref{(i)} or \ref{(ii)} of Corollary \ref{cor: minimaxLB}.
We will construct a distributed method which, up to logarithmic factors, attains 
the corresponding lower bounds, without using knowledge about the regularity level $s$.

\begin{remark}
We provide some examples for the value of $s_{\min}$ for different choices of $B$ and $m$. Taking $m=\sqrt{n}$ we have for all $B\geq \log_2 n$ that $s_{\min}=0$. For $m=\log n$ and $B=\sqrt{n}$ we get $s_{\min}=1/2$. For $m=\log n$ and $B=\log_2 n$ we have that $s_{\min}=\infty$. Note that it is intuitively clear that in case the number of machines is large, then it is typically advantageous to use a distributed method compared to a single local machine as we would lose too much information in the later case. However, if we have a small number of machines and can transmit only a very limited amount of information, then it might be more advantageous to use only a single machine to make inference.
\end{remark}

In the non-adaptive case we saw that different strategies were required to attain 
the optimal rate, case \ref{(ii)} requiring a particular grouping of the local machines. 
The cut-off between cases \ref{(i)} and \ref{(ii)} depends, however, on the value of $s$, 
so in the present adaptive setting we do not know beforehand in which of the two cases
we are. 
%\footnote{Also note that the $\#$ of transmitted coefficients depend on $s$ as well, which is not known... I think it is even a more important problem than not knowing if we are in case \ref{(i)} or \ref{(ii)}.}
In order to tackle this problem we introduce a somewhat more involved grouping of the 
machines, which basically gives us the possibility to carry out both  strategies simultaneously.
This is combined with a modified version of Lepski's method, carried out in the 
central machine, ultimately leading to (nearly) optimal distributed concentration rates for every regularity class $s\in [s_{\min}, s_{\max})$, simultaneously. (We note that in our 
distributed regression setting, deriving an appropriate version of 
Lepski's method requires some non-standard technical work, see Section \ref{sec: adaptive}).

As a first step in our adaptive procedure we divide the machines into groups. To begin with, let us take the first $\lfloor m/2\rfloor$ machines and denote the set of their index numbers by $I$. Then 
the remaining $\lceil m/2\rceil$ machines are split into $\tilde{\eta}_n=j_{\max}-j_{B,n}$  equally sized groups (for simplicity each group has $\lfloor \lceil m/2 \rceil/\tilde\eta \rfloor$ machines and the leftovers are discarded), where
\begin{align*}
j_{B,n}&:= \lfloor \log_2 \lfloor B/\log_2 n\rfloor\rfloor\\
j_{\max}&:=\lceil  (2+2s_{\min})^{-1}\log_2 (nB) \rceil\wedge \lceil (1+2s_{\min})^{-1}\log_2 n  \rceil.
\end{align*}
The corresponding sets of indexes
are denoted by $I_0,I_1,\ldots, I_{\tilde\eta-1}$. Note that $|I_t|\asymp m/\log_2 n$, for $t\in\{0,...,\tilde\eta-1\}$. Then the machines in the group $I$ transmit the approximations $Y_{jk}^{(i)}$ (with $D=1/2$ in Algorithm \ref{alg: transmit:number}) of the local estimators of the wavelet coefficients  $\hat{f}_{jk}^{(i)}$, for $0\leq j\leq j_{B,n}-1, k=0,...,2^j-1$, with 

 to the central machine.  The machines in group $I_t$, $t\in\{0,...,\tilde\eta-1\}$, will be responsible for transmitting the coefficients at resolution level $j=j_{B,n}+t$. First for every $t\in\{0,\ldots,\tilde\eta-1\}$, the machines in group $I_t$ are split again into $2^{t}$ equal size groups (for simplicity each group has $\lfloor2^{-t}\lfloor \lceil m/2 \rceil/\tilde\eta \rfloor\rfloor\geq 1$ machines and the leftovers are discarded again), denoted by $I_{t,1},I_{t,2},\ldots,I_{t,2^{t}}$.  A machine $i$ in one of the groups  $I_{t,\ell}$
for $\ell\in\{1,...,2^{t}\}$ transmits the approximations $Y_{jk}^{(i)}$ (again with $D=1/2$ in Algorithm \ref{alg: transmit:number})  of the local estimators of the wavelet coefficients  $\hat{f}_{jk}^{(i)}$, for $j=j_{B,n}+t$ and $(\ell-1)2^{j_{B,n}}\leq   k<  \ell 2^{j_{B,n}}$ to the central machine. 

In the central machine we first average  the transmitted approximations of the corresponding coefficients. We define
\begin{equation}\label{def: y_jk}
\hat{f}_{jk}=
\begin{cases}
|I|^{-1} \sum_{i\in I} Y_{jk}^{(i)} & \text{if  $j<j_{B,n}, k=0,...,2^{j}-1$,}\\
 |I_{t,\ell}|^{-1} \sum_{i\in I_{t,\ell}} Y_{jk}^{(i)}& \text{if $j_{B,n}\leq j\leq j_{B,n}+\tilde{\eta}$, $k=0,...,2^{j}-1$.}
\end{cases}
\end{equation}
Using these coefficients we can construct for every $j$ the preliminary estimator 
\begin{align}
\tilde{f}(j)=\sum_{l\leq j-1}\sum_{k=0}^{2^l-1}  \hat{f}_{jk}\psi_{jk}.\label{def:tilde:f}
\end{align}
This gives us a sequence of estimators from which we
select the appropriate one 
using a  modified version of Lepski's method.  
We consider 
$\mathcal{J}=\{0,...,j_{\max}\}$ and  define $\hat{j}$ as
\begin{align}
\hat{j}=\min\big\{ j\in \mathcal{J}:   \|\tilde{f}(j)-\tilde{f}(l)\|_2^2\leq \tau 2^l/n_l,\, \forall l>j,\, l\in\mathcal{J}  \big\},
\end{align}
for some sufficiently large parameter $\tau>1$ and $n_j=|I_{j-j_{B,n},1}| n/m\asymp \frac{nB}{2^{j}(\log_2 n)^2}$, for $j\geq j_{B,n}$ and $n_j=|I|n/m\asymp n$ for $j<j_{B,n}$. Then we construct our final estimator $\hat{f}$ simply by taking $\hat{f}=\tilde{f}(\hat{j})$.

We summarize the above procedure (without discarding servers for achieving exactly equal size subgroups) in Algorithm \ref{alg: adapt:L2}, below.

\bigskip

\begin{algorithm}
\caption{Adaptive $L_2$-method}\label{alg: adapt:L2}
\begin{algorithmic}[1]
\BState \textbf{In the local machines}:
\For {$i=1$ to $\lfloor m/2\rfloor$}
\For {$ j= 0$ to $j_{B,n}-1$}
\For {$k=0$ to $2^j-1$}
\State \text{$Y_{jk}^{(i)}$ :=TransApprox($\hat{f}_{jk}^{(i)}$)}
\EndFor
\EndFor
\EndFor
\For{$t= 0$ to $\tilde\eta-1$ }
\For{$\ell =1$ to $2^{t}$}
\For {{$i =\lfloor m/2 \rfloor +t\Big\lfloor \frac{\lceil m/2 \rceil}{\tilde{\eta}}\Big\rfloor+ (\ell-1)\Big\lfloor2^{-t}\Big\lfloor\frac{\lceil m/2 \rceil}{\tilde{\eta}}\Big\rfloor\Big\rfloor+1$ to $\lfloor m/2 \rfloor +t\Big\lfloor \frac{\lceil m/2 \rceil}{\tilde{\eta}}\Big\rfloor+ \ell\Big\lfloor2^{-t}\Big\lfloor\frac{\lceil m/2 \rceil}{\tilde{\eta}}\Big\rfloor\Big\rfloor$}}
\For {$j=j_{B,n}$  to $j_{B,n}+\tilde{\eta}-1$}
\For {$k=0$ to $2^j-1$}
\State \text{$Y_{jk}^{(i)}$ :=TransApprox($\hat{f}_{jk}^{(i)}$)}
\EndFor
\EndFor
\EndFor
\EndFor
\EndFor
\BState \textbf{In the central machine}:
\BState \textit{(1) Averaging the local observations}:
\For{$j=0$ to $j_{B,n}-1$}
\For {$k=0$ to $2^j-1$}
\State \text{$\hat{f}_{jk}:=mean\{Y_{jk}^{(i)}: i\leq m/2$\}}
\EndFor
\EndFor
\For{$t= 0$ to $\tilde\eta-1$ }
\State{$j:=j_{B,n}+t$}
\For{$\ell =1$ to $2^{t}$}
\For{$k=(\ell-1)2^{j_{B,n}}$ to   $\ell 2^{j_{B,n}}-1$}

\State {\text{$\hat{f}_{jk}:=mean\Big\{Y_{jk}^{(i)}:\,  \lfloor \frac{m}{2} \rfloor +t\Big\lfloor \frac{\lceil m/2 \rceil}{\tilde{\eta}}\Big\rfloor+ (\ell-1)\Big\lfloor2^{-t}\Big\lfloor\frac{\lceil m/2 \rceil}{\tilde{\eta}}\Big\rfloor\Big\rfloor<  i\leq$}
\State \qquad\qquad\qquad\qquad\qquad\text{ $\leq  \lfloor \frac{m}{2} \rfloor +t\Big\lfloor \frac{\lceil m/2 \rceil}{\tilde{\eta}}\Big\rfloor+ \ell\Big\lfloor2^{-t}\Big\lfloor\frac{\lceil m/2 \rceil}{\tilde{\eta}}\Big\rfloor\Big\rfloor\Big\}$}}

\EndFor
\EndFor
\EndFor
\BState \textit{(2) Lepski's method}:
\For{ {$j=0$} to $j_{\max}$}
\State \text{ $\tilde{f}(j):=\sum_{l\leq j-1}\sum_{k=0}^{2^j-1} \hat{f}_{jk}\psi_{jk}$}
\EndFor
\State{ Let $\hat{j}:=j_{\max}$, $stop:=FALSE$}
\While{ $stop==FALSE$ and $\hat{j}\geq  0$}
\State{Let $l:=\hat{j}+1$}
\While{ $stop== FALSE$ and $l\leq j_{\max}$}
\If{$\|\tilde{f}(j)-\tilde{f}(l)\|_2^2\leq \tau 2^l/ n_l$}
\State{$l:=l+1 $}
\Else{ $stop:=TRUE$}
\EndIf
\EndWhile
\If{$stop==FALSE$}
\State{$\hat{j}:=\hat{j}-1$}
\EndIf
\EndWhile
\State Construct: $\hat f = \tilde{f}(\hat{j})$.
\end{algorithmic}
\end{algorithm}

\newpage

\begin{theorem}\label{theorem: adaptive}
For every $L,s>0$ the distributed method $\hat f$ described above belongs to $\mathcal{F}_{dist}(B,\ldots,B; B_{2,\infty}^{s}(L))$ and for all $s\in [s_{\min},s_{\max})$

\begin{equation*}
\sup_{f_0\in B_{2,\infty}^s(L)}\mathbb{E}_{f_0,T}\|\hat{f}- f_0\|_2 \lesssim 
\begin{cases}
n^{-s/(1+2s)},&\text{if  $B\geq  C_L n^{1/(1+2s)}\log_2 n$,}\\
M_n\Big(\frac{n^{1/(1+2s)} }{B\log_2 n}\Big)^{\frac{s}{2+2s}}n^{-\frac{s}{1+2s}}, &\text{if $B< C_L n^{1/(1+2s)}\log_2 n$,}\\
\end{cases}
\end{equation*}
with $C_L= 4L^{2/(1+2s)}$ and $M_n=(\log_2 n)^{3s/(2+2s)}$.
\end{theorem}

\begin{proof}
See Section \ref{sec: adaptive}.
\end{proof}

\begin{remark}
Compared to the lower bound in Corollary \ref{cor: minimaxLB} one can observe that in case $B\geq  n^{1/(1+2s)}\log_2 n$ the upper bound is sharp, for $B<  n^{1/(1+2s)}\log_2 n$ we might get an extra slowly varying term of order at most $O((\log n)^{6s/(2s+2)})$.
\end{remark}

A slight modification of the above algorithm also leads to a (up to a logarithmic factor) minimax adaptive  estimation rate in the $L_\infty$-norm. We construct the truncation estimator $\tilde{f}(j)$ as in Algorithm \ref{alg: adapt:L2}, see \eqref{def:tilde:f}. The only difference to the $L_2$-case is that we introduce an extra $l$ term in the definition of $\hat{j}$, i.e.
\begin{align*}
\hat{j}=\min\big\{ j\in \mathcal{J}:   \|\tilde{f}(j)-\tilde{f}(l)\|_\infty\leq \tau \sqrt{l2^l/n_l},\, \forall l>j,\, l\in\mathcal{J}  \big\}.
\end{align*}
Finally we define $\hat{f}=\tilde{f}(\hat{j})$ and show below that it attains the nearly optimal minimax rate adaptively.

\begin{theorem}\label{theorem: adaptive:Linfty}
For every $L,s>0$ the distributed method $\hat f$ described above belongs to $\mathcal{F}_{dist}(B,\ldots,B; B_{\infty,\infty}^{s}(L))$. Furthermore for all $s\in [s_{\min},s_{\max})$,
\begin{equation*}
\sup_{f_0\in B_{\infty,\infty}^s(L)}\mathbb{E}_{f_0,T}\|\hat{f}- f_0\|_\infty \lesssim 
\begin{cases}
(n/\log_2 n)^{-s/(1+2s)},&\text{if  $B\geq  t_n$,}\\
M_n\Big(\frac{n^{1/(1+2s)} }{B(\log_2 n)^{\frac{3+4s}{1+2s}}}\Big)^{\frac{s}{2+2s}}(\frac{n}{\log_2 n})^{-\frac{s}{1+2s}} , &\text{if $ B<t_n $,}\\
\end{cases}
\end{equation*}
with $t_n= C_L (n/\log_2 n)^{\frac{1}{1+2s}}\log_2 n$, $C_L= 12 L^{2/(1+2s)}$, and $M_n=(\log_2 n)^{\frac{(5+8s)s}{(1+2s)(2+2s)}}$.
\end{theorem}

\begin{proof}
See Section \ref{sec: adaptive:Linfty}.
\end{proof}

\section{Proofs for the $L_{2}$-norm}

\subsection{Proof of Theorem \ref{theorem: minimaxL2LB}}\label{sec: minimaxL2LB}

Note that without loss of generality we can multiply $\delta_n$ with an arbitrary constant.
In the proof we define $\delta_n$ as the solution to
\begin{align}
\delta_n =\bar{C}^{-1}\min\Big\{\frac{m}{n \log_2 n }, \frac{m}{n \sum_{i=1}^m [(\delta_n^{\frac{1}{1+2s}}\log_2 (n) B^{(i)}) \wedge  1]} \Big\},\label{def: delta_n_new}
\end{align}
for some sufficiently large $\bar{C}\geq 1$ to be specified later. We prove the desired lower bound for the minimax risk using a 
modified version of Fano's inequality, given in Theorem \ref{lem: Fano:Wainwright}.
As a first step we construct a  finite subset $\FF_0 \subset B_{2,\infty}^s(L)$. 
We use the wavelet notation outlined in Section \ref{sec: wavelets} of the supplement and consider Daubechies wavelets with at least $s$ vanishing moments. Define $j_n=\lfloor(\log_2 \delta_n^{-1})/(1+2s)\rfloor$. Next we divide the interval $[0,1]$ into the partition of $2^{j_n}/\tilde{C}$ disjoint intervals $I_1,...,I_{2^{j_n}/\tilde{C}}$ (without loss of generality we have assumed that $2^{j_n}/\tilde{C}\in\mathbb{N}$), such that each interval $I_k$ contains a support of a wavelet basis function $\psi_{j_n,\ell}$, $\ell\in\{0,...,2^{j_n}-1\}$ (for Daubechies wavelets with $s$ vanishing moments this is possible for $\tilde{C}\geq 2s+2$). Slightly abusing our notations let us denote a basis function corresponding to the $k$th interval $I_k$ by $\psi_{j_n,k}$ and by $K_{j_n}=\{1,2,...,2^{j_n}/\tilde{C}\}$ the index set of the intervals (and basis functions). Note that the basis functions $\psi_{j_n,k}$, $k\in K_{j_n}$, have disjoint support.

For $\beta \in\{-1,1\}^{|K_{j_n}|}$, let $f_\beta \in L_2[0,1]$ be 
the function with  wavelet coefficients  
 \begin{align}\label{eq: function_minimax}
 f_{\beta, jk}=
 \begin{cases}
L\beta_{k}\delta_n^{1/2}, & \text{if}\quad j=j_n, k\in  K_{j_n},\\
 0, & \text{else},
 \end{cases}
\end{align}
and take $\bar{C}=2^{17}\tilde{C}L^2\|\psi\|_\infty^2$. Now define $\FF_0 = \{f_\beta: \beta \in\{-1,1\}^{|K_{j_n}|}\}$.
 Note that $\FF_0 \subset B_{2,\infty}^s(L)$, since
\begin{align*}
\|f_\beta\|_{B_{{2},\infty}^s}^2=\sup_{j}2^{2sj}\sum_{k=0}^{2^j-1}f_{\beta, jk}^2 \leq L^22^{2s j_n}|K_{j_n}|\delta_n \leq  L^2.
\end{align*}
For this set of functions $\FF_0$, the maximum and minimum number of elements
in balls of radius $t >0$, given by  
\begin{align*}
N_{t}^{\max}= \max_{f_\beta\in\FF_0}\big|\{ f_{\beta'}\in\FF_0:\, 
\|f_\beta-f_{\beta'}\|_2\leq t\} \big|,\\
N_{t}^{\min}= \min_{f_\beta\in\FF_0}\big|\{ f_{\beta'}\in\FF_0:\, 
\|f_\beta-f_{\beta'}\|_2\leq t\} \big|,
\end{align*}
satisfy $N_{t}^{\max}=N_t^{\min}=\sum_{i=0}^{\tilde{t}}{|K_{j_n}| \choose i}<|\FF_0|/2$ for $\tilde{t}=\frac{t^2}{4\delta_n L^2}<|K_{j_n}|/2$ (and therefore $N_{t}^{\max}<|\FF_0|-N_t^{\min}$).

Let $F$ be a uniform random variable over $\{-1,1\}^{|K_{jn}|}$. Note that the design $T$ is independent of $F$, while the data $X$ depends on $F$. In each local machine $i$ we observe the pair of random variables $(T^{(i)},X^{(i)})$ and we transmit a measurable function $Y^{(i)}$ of this local data to the central machine. This provides us the Markov chains $F\rightarrow (T^{(i)},X^{(i)})\rightarrow Y^{(i)}$, $i=1,...,m$ or by jointly writing them in the form
\begin{align}
F\rightarrow (T,X)\rightarrow Y.\label{eq: markov:chain}
\end{align}
Then in view of Theorem \ref{lem: Fano:Wainwright} (with {$t^2=2L^2\delta_n|K_{j_n}|/3$, $d(f,g)=\|f-g\|_2$}) that
\begin{align}
\inf_{\hat{f}\in \mathcal{F}_{dist}(B^{(1)},\ldots, B^{(m)};B_{2,\infty}^s(L))}\sup_{f_0\in B_{2,\infty}^{{s}}(L)}\mathbb{E}_{f_0,T}\|\hat{f}-f_0\|_2^2
\gtrsim L^2\delta_n |K_{j_n}|\Big(1-\frac{I(F;Y)+\log 2}{\log(|\FF_0|/N_{t}^{\max})}\Big),\label{eq: LB:risk2}
\end{align}
where $I(F;Y)$ is the mutual information between the random variables $F$ and $Y$.

%Next let us take take a uniform random variable $M$ on the set $\mathcal{M}$ and consider the Markov chain $M\rightarrow X\rightarrow Y\rightarrow \hat{f}_n$. Let us denote by $\mathcal{H}(B^{(1)},..,B^{(m)};\mathcal{F})$ the collection of $X=(X^{(1)},...,X^{(m)})$ measurable, binary random vectors $Y=(Y^{(1)},...,Y^{(m)})$, such that $Y^{(i)}$ given $X^{(i)}$ is independent from $Y^{(-i)},X^{(-i)}$ and over the class $\mathcal{F}$ the expected value of bits decoding $Y^{(i)}$ is bounded above by $B^{(i)}$, i.e. for all $f\in\mathcal{F}$ we have $\mathbb{E}_{f}[l(C_{Y^{(i)}})]\leq B^{(i)}$, $i=1,...,m$. Then by applying Lemma 

To lower bound the right-hand side, first
 note that $N_{t}^{\max}=\sum_{i=0}^{{\tilde{t}}}{|K_{j_n}| \choose i}< 2 {|K_{j_n}| \choose {\tilde{t}}}\leq 2(e|K_{j_n}|/{\tilde{t}})^{\tilde{t}}$ and therefore, for ${\tilde{t}}=|K_{j_n}|/6$ {(i.e. $t^2=2L^2\delta_n|K_{j_n}|/3$)},
\begin{align*}
\log(|\FF_0|/N_t^{\max})\geq |K_{j_n}| \log (2 (6e)^{-1/6}2^{-1/|K_{j_n}|})\geq |K_{j_n}|/6.
\end{align*}
Hence,  
to derive the statement of the theorem from \eqref{eq: LB:risk2} it {is sufficient} to 
show that 
%can be reformulated as
%\begin{align*}
%&\inf_{\hat{f}\in \mathcal{F}_{dist}(B^{(1)},...B^{(m)};B_{2,\infty}^s(L))}\sup_{f_0\in B_{2,\infty}^s(L)}\mathbb{E}_{f_0}\|\hat{f}_n-f_0\|_2^2\nonumber\\
%&\qquad\qquad\gtrsim \inf_{Y\in \mathcal{H}(B^{(1)},..,B^{(m)};B_{2,\infty}^s(L))}\delta_n^{\frac{2s}{1+2s}}\Big( 1-\frac{I(M;Y)+\log 2}{\delta_n^{-1/(1+2s)}/6}\Big).
%\end{align*}
%
%
%
%Hence by noting that $\mathcal{M}\subset B_{2,\infty}^{s}(1)$ it is sufficient to verify the inequality
\begin{align}
I(F;Y)\leq |K_{j_n}|/8+O(1)\label{eq: hulp:endproof}.
\end{align}
The proof of the next lemma is deferred to Section \ref{sec: proof: lem:reg}.
\begin{lemma}\label{lem: mutual_regression}
For the Markov chain $F\rightarrow (T,X)\rightarrow Y$ introduced in \eqref{eq: markov:chain} we have for $m=O(n^{\frac{2s}{1+2s}}/\log_2^2 n)$ that
\begin{align}
I({F};{Y})\leq  \frac{4L^2\tilde{C} \|\psi\|_\infty^2\delta_n |K_{j_n}| n}{m}\sum_{i=1}^m\Big( (2^{12}\log(n)|K_{j_n}|^{-1} B^{(i)})\wedge 1\Big)+O(1). \label{eq: help:mutual:info}
\end{align} 

\end{lemma}

Since in view of the definition of $\delta_n$ we have that
$$\delta_n\leq \frac{2^{12}\bar{C}^{-1}m}{ n \sum_{i=1}^m\big[\big(2^{12}\log_2(n)\delta_n^{\frac{1}{1+2s}} B^{(i)}\big)\wedge 1\big] },$$
the right-hand side of \eqref{eq: help:mutual:info} is further bounded by $2^{-3}|K_{j_n}|+O(1)$, finishing the proof of assertion \eqref{eq: hulp:endproof} and concluding the proof of the theorem.

\subsection{Proof of Theorem \ref{theorem: minimaxL2UB1}}\label{sec: minimaxL2UB}
First note that by using Cauchy-Schwartz inequality we get that
\begin{align*}
\mathbb{E}_{f_0,T}(\log_2 |\hat{f}_{jk}^{(i)}|\vee1)&\leq 1+\mathbb{E}_{f_0,T} |\hat{f}_{jk}^{(i)}| 
= 1+\mathbb{E}_{f_0,T} |X_{1}^{(i)}\psi_{jk}(T_1^{(i)})|\\
& \leq 1+\|f_0\|_2\|\psi_{jk}\|_2+ \|\psi_{jk}\|_2\mathbb{E}_{f_0}|\eps_1^{(i)}| =O(1).
\end{align*}
Hence  in view of Lemma \ref{lem: approx2} (with $D=1/2$) the approximation satisfies 
\begin{align*}
0\leq |\hat{f}_{jk}^{(i)}-Y_{jk}^{(i)}|\leq 1/\sqrt{n}\quad \text{and}\quad \mathbb{E}_{f_0,T}[l({Y_{jk}^{(i)}})]\leq(1/2+o(1))\log_2 n.
\end{align*}
 Therefore  we need at most $(1/2+o(1))B$ bits in expected value to transmit
  $\{Y_{jk}^{(i)}:\, 2^j+k\leq n^{1/(1+2s)}\wedge \lfloor B/\log_2 n\rfloor\}$, hence $\hat{f}_n\in\mathcal{F}_{dist}(B,...,B;B_{2,\infty}^s(L))$. 

Next for convenience we introduce the notation for the approximation error $W_{jk}^{(i)}=Y_{jk}^{(i)}- \hat{f}_{jk}^{(i)}$, satisfying $|W_{jk}^{(i)}|\leq n^{-1/2}$. The estimator $\hat{f}$ is given by its wavelet coefficients $\hat{f}_{jk}$, $j\in\mathbb{N}, k\in \{0,1,...,2^j-1\}$. For $2^j+k>n^{1/(1+2s)}\wedge \lfloor B/\log_2 n\rfloor$ we have $ \hat{f}_{jk}=0$, 
while for $2^j+k\leq n^{1/(1+2s)}\wedge \lfloor B/\log_2 n\rfloor$,
\begin{align*}
\hat{f}_{jk}=\frac{1}{m}\sum_{i=1}^m Y_{jk}^{(i)}=\frac{1}{m}\sum_{i=1}^m (\hat{f}_{jk}^{(i)}+W_{jk}^{(i)})
 = f_{0,jk}+Z_{jk}+W_{jk},
\end{align*}
where $Z_{jk}=m^{-1}\sum_{i=1}^{m} (\hat{f}_{jk}^{(i)}- \mathbb{E}_{f_0,T}\hat{f}_{jk}^{(i)})$ and $|W_{jk}|=|m^{-1}\sum_{i=1}^m W_{jk}^{(i)}|\leq n^{-1/2}$. Note that in view of assumption $\|f_0\|_{\infty}\leq M$
\begin{align*}
\mathbb{E}_{f_0,T} Z_{jk}^2& \leq 2\mathbb{E}_{f_0,T} \Big( \frac{1}{n}\sum_{i=1}^m\sum_{\ell=1}^{n/m}f_0(T_\ell^{(i)})\psi_{jk}(T^{(i)}_\ell)- \mathbb{E}_{f_0,T} f_0(T_\ell^{(i)})\psi_{jk}(T^{(i)}_\ell) \Big)^2\\
&\qquad+ 2\mathbb{E}_{f_0,T}\Big( \frac{1}{n}\sum_{i=1}^m\sum_{\ell=1}^{n/m} \eps_{\ell}^{(i)}\psi_{jk}(T_\ell^{(i)})\Big)^2\\
&\leq 2n^{-1}\mathbb{E}_{T}\Big( f_0(T_1^{(1)})\psi_{jk}(T^{(1)}_1)- \mathbb{E}_{T} f_0(T_1^{(1)})\psi_{jk}(T^{(1)}_1) \Big)^2\\
&\qquad +2n^{-1} \mathbb{E}_{f_0}(\eps_{1}^{(1)})^2  \mathbb{E}_T \psi_{jk}^2(T_1^{(1)})\\\
&\leq 2n^{-1} \int_{0}^1 f_0^2(t)\psi_{jk}^2(t)dt+2n^{-1}\leq 2(M^2+1)/n.
\end{align*}

For convenience we also introduce the notation $j_n=\big\lfloor \log_2\big(L^\frac{2}{1+2s}(n^{\frac{1}{1+2s}}\wedge \lfloor B/\log_2 n\rfloor)\big)\big\rfloor$. Then by combining the above  inequalities we get that the risk is bounded from above by
\begin{align}
\mathbb{E}_{f_0,T}\|\hat{f}-f_0\|_2^2&\leq \sum_{j\geq j_n}\sum_{k=0}^{2^j-1}f_{0,jk}^2+2\sum_{j=0}^{j_n}\sum_{k=0}^{2^j-1}\mathbb{E}_{f_0,T}(Z_{jk}^2+W_{jk}^2)\nonumber\\
&\lesssim  L^2\sum_{j\geq j_n}2^{-2js}\sup_{j\geq j_n} 2^{2js}\sum_{k=0}^{2^j-1}f_{0,jk}^2+
\sum_{j=0}^{j_n}\sum_{k=0}^{2^j-1} n^{-1}\nonumber \\
&\lesssim L^22^{-2j_ns} +2^{j_n}/n\lesssim L^{\frac{2}{1+2s}} (n^{-2s/(1+2s)}\vee (B/\log_2 n)^{-2s}).\label{eq: risk:large:B}
\end{align}

\subsection{Proof of Theorem \ref{theorem: minimaxL2UB2}}\label{sec: minimaxL2UB2}
Similarly to the proof of Theorem \ref{theorem: minimaxL2UB1} we get that $\mathbb{E}_{f_0,T}[l({Y_{jk}^{(i)}})]\leq (1/2+o(1))\log_2 n$ and since each machine transmits at most $\lfloor B/\log_2 n\rfloor$ coefficients, the total amount of transmitted bits per machine is bounded from above by $B$ (for large enough $n$), hence $\hat{f}\in\mathcal{F}_{dist}(B,\ldots,B; B_{2,\infty}^s(L))$. 

Next let $A_{jk}=\{ \lfloor \mu_{jk} m/\eta\rfloor +1,...,\lfloor (\mu_{jk}+1) m/\eta\rfloor \}$ be the collection of machines transmitting the $(j,k)$th approximated wavelet coefficient $Y_{jk}^{(i)}$ {and note that the size of the set satisfies $|A_{jk}|\asymp m/\eta$.} Then our aggregated estimator $\hat{f}$ satisfies for $2^j+k\leq \eta  \lfloor B/\log_2 n \rfloor$ (i.e. the total number of different coefficients transmitted) that 
\begin{align*}
\hat{f}_{jk}&=\frac{1}{|A_{jk}|}\sum_{i\in A_{jk}}Y_{jk}^{(i)}= f_{0,jk}+Z_{jk}+W_{jk},
\end{align*}
where  $|W_{jk}|=\frac{1}{|A_{jk}|} |\sum_{i\in A_{jk}}W_{jk}^{(i)}|\leq n^{-1/2}$ and $Z_{jk}=\frac{1}{|A_{jk}|}\sum_{i\in A_{jk}} (\hat{f}_{jk}^{(i)}- \mathbb{E}_{f_0,T}\hat{f}_{jk}^{(i)})$. Note that similarly to above
\begin{align}
\mathbb{E}_{f_0,T} Z_{jk}^2& \leq 2\mathbb{E}_{f_0,T} \Big( \frac{m}{n|A_{jk}|}\sum_{i\in A_{jk}}\sum_{\ell=1}^{n/m}f_0(T_\ell^{(i)})\psi_{jk}(T^{(i)}_\ell)- \mathbb{E}_{f_0,T} f_0(T_\ell^{(i)})\psi_{jk}(T^{(i)}_\ell) \Big)^2\nonumber\\
&\qquad+ 2\mathbb{E}_{f_0,T}\Big( \frac{m}{n|A_{jk}|}\sum_{i\in A_{jk}}\sum_{\ell=1}^{n/m} \eps_{\ell}^{(i)}\psi_{jk}(T_\ell^{(i)})\Big)^2\nonumber\\
&\leq  \frac{2m}{n|A_{jk}|}\mathbb{E}_{T}\Big( f_0(T_1^{(1)})\psi_{jk}(T^{(1)}_1)- \mathbb{E}_{T} f_0(T_1^{(1)})\psi_{jk}(T^{(1)}_1) \Big)^2\nonumber\\
&\qquad + \frac{2m}{n|A_{jk}|} \mathbb{E}_{f_0}(\eps_{1}^{(1)})^2  \mathbb{E}_T \psi_{jk}^2(T_1^{(1)})\leq  \frac{2(M^2+1)m}{n|A_{jk}|}.\label{eq: UB:square:noise}
\end{align}

Let  $j_n=\lfloor\log_2 (\eta \lfloor B/\log_2 n\rfloor )\rfloor$. Then similarly to \eqref{eq: risk:large:B} the risk of the aggregated estimator  is bounded as
\begin{align}
\mathbb{E}_{f_0,T}\|\hat{f}-f_0\|_2^2&\leq \sum_{j\geq j_n}\sum_{k=0}^{2^{j}-1}f_{0,jk}^2+2\sum_{j=0}^{j_n}\sum_{k=0}^{2^j-1}\mathbb{E}_{f_0,T}(Z_{jk}^2+W_{jk}^2)\nonumber\\
&\lesssim  \sum_{j\geq j_n}2^{-2js}\sup_{j\geq j_n} 2^{2js}\sum_{k=0}^{2^j-1}f_{0,jk}^2+
\sum_{j=0}^{j_n}\sum_{k=0}^{2^j-1} \eta/n \nonumber \\
&\lesssim  L^2\big(\frac{B\eta}{\log _2n}\big)^{-2s}+\frac{B\eta^{{2}}}{n\log_2 n}\asymp L^{\frac{4}{2+2s}}(nB/\log_2 n)^{-\frac{2s}{2+2s}}\nonumber\\
&= L^{\frac{4}{2+2s}}(\log_2 n)^{\frac{4s}{2+2s}}\Big(\frac{n^{1/(1+2s)}}{B\log_2 n}\Big)^{\frac{2s}{2+2s}}n^{-\frac{2s}{1+2s}}\vee L^2\big(\frac{Bm}{\log_2n}\big)^{-2s},\label{UB: alg2}
\end{align}
 concluding the proof of the theorem.

Finally we show that Algorithm \ref{alg: transmit:number} is in general 
suboptimal in this case. Consider the function $f_0\in B_{2,\infty}^s(1)$ with wavelet coefficients $f_{0,jk}=2^{-j(s+1/2)}$, $j\in\mathbb{N}$, $k=0,...,2^{j}-1$, and take $j_n=\lfloor\log_2 \lfloor B/\log_2 n\rfloor \rfloor$, then
\begin{align*}
\mathbb{E}_{f_0,T}\|\hat{f}-f_0\|_2^2&\geq \sum_{j\geq j_n}\sum_{k=0}^{2^j-1}f_{0,jk}^2\geq \sum_{k=0}^{2^j-1}2^{-j_n (2s+1)}\\
& \gtrsim \Big(\frac{B}{\log_2 n}\Big)^{-2s}=\tilde{M}_n \Big(\frac{n^{1/(1+2s)}}{B\log_2 n}\Big)^{\frac{2s}{2+2s}}n^{-\frac{2s}{1+2s}}
\end{align*}
where the multiplication factor $\tilde{M_n}=\Big(\frac{n(\log_2 n)^{3+2s}}{B^{1+2s}}\Big)^{\frac{2s}{2+2s}}$ tends to infinity and can be of polynomial order, yielding a highly sub-optimal rate.
%
%\begin{remark}
%Note that in \eqref{UB: alg2} the bias term is of order $(nB\log_2 n)^{-2s/(2+2s)}$, while the variance term is a multiple of $(\log_2 n)^{-1}(B/n^{1+2s})^{1/(2+2s)}$, which is of smaller order than the bias for $B\leq n^{1/(1+2s)}/\log_2 n$. This computation also shows that in case we can not transmit enough information between the local servers and the central one, we will not achieve optimal bias and variance trade-off (we are over smoothing) and as a consequence we can not obtain the minimax non-distributed (single server) convergence rate.
%\end{remark}

\subsection{Proof of Lemma \ref{lem: approx2}}\label{sec: approx}
 One can easily see by construction that 
\begin{align}
0\leq |X-Y|\leq n^{-D}.\label{eq: error}
 \end{align}
Next note that the expected number of transmitted bits is bounded from above by
\begin{align*}
\EE\big(1+(1\vee \log_2 |X|)+D\log_2 n\big)&= 1+D\log_2 (n)+\EE(1\vee \log_2 |X|)\\
&=(D+o(1))\log_2 n.
\end{align*}
%The expected value $\EE(1\vee \log_2 |X|)$ is further bounded from above by
%\begin{align*}
%&2\int_0^{2} \frac{1}{\sqrt{2\pi m/n}} e^{-\frac{(x-\mu)^2}{2m/n}}dx +2\int_2^{\infty} \frac{\log_2 |x|}{\sqrt{2\pi m/n}} e^{-\frac{(x-\mu)^2}{2m/n}}dx \\
%&\lesssim 1+\int_1^{\infty} \frac{x}{\sqrt{2\pi m/n}} e^{-\frac{(x-\mu)^2}{2m/n}}dx \lesssim 1+\mu=O(1),
%\end{align*}
%hence we can conclude that $\EE l(Y)\leq (D+o(1))\log n$.

%Finally, note that by similar computation
%\begin{align*}
%\PP\big( 1\vee\log_2|X|>\log_2 (n^2)\big)&\leq2\int_{n^2}^{\infty}\sqrt{\frac{n}{2\pi m}}e^{-\frac{(x-\mu)^2}{2m/n}}dx\lesssim\sqrt{n}e^{-n^3/(4m)}=o(e^{-n}),
%\end{align*}
%finishing the proof of the lemma.

\subsection{Proof of Theorem \ref{theorem: adaptive}}\label{sec: adaptive}
First recall that for every $s,L>0$ and $f_0\in B_{2,\infty}^s(L)$ we have $f_{0,jk}^2\leq L^2$, $j\geq 0, k\in \{0,1,...,2^j-1\}$. Therefore, in view of Lemma \ref{lem: approx2} (with $D=1/2$) we have $\mathbb{E}_{f_0,T}[l({Y_{jk}^{(i)}})]\leq (1/2+o(1))\log_2 n$. Since the machines in group $I$ and the machines in $I_{t,\ell}$, $t\in\{0,...,\tilde{\eta}-1\}$, $\ell\in\{1,...,2^{t}\}$ transmit at most $ \lfloor B/\log_2 n\rfloor$ coefficients we have that in expected value at most
\begin{align*}
\lfloor B/\log_2 n\rfloor \big(1/2+o(1)\big)\log_2 n\leq  B
\end{align*}
bits are transmitted  per machine (for $n$ large enough). Therefore the estimator indeed belongs to 
$\mathcal{F}_{dist}(B,\ldots,B; B_{2,\infty}^{s}(L))$.

Next we show that the estimator $\hat{f}$ achieves the minimax rate. First let us introduce the notation $|W_{jk}^{(i)}|=|Y_{jk}^{(i)}-\hat{f}_{jk}^{(i)}|\leq n^{-1/2}$. Then note that for $j\leq j_{\max}$ and $k\in \{0,1,...,2^j-1\}$ {the aggregated quantities $\hat{f}_{jk}$ defined in \eqref{def: y_jk} are equal to} 
\begin{align}
\hat{f}_{jk}=\frac{1}{|A_{jk}|}\sum_{i\in A_{jk}} Y_{jk}^{(i)}=  f_{0,jk}+Z_{jk}+W_{jk},\label{eq: model:mGWN}
\end{align}
where
\begin{equation*}
A_{jk}=\begin{cases}
I, & \text{if $j<j_{B,n}$, $k=0,1,...,2^{j}-1$},\\
I_{j-j_{B,n} ,\ell}, & \text{if $j\geq j_{B,n}$, $(\ell-1)2^{j_{B,n}} \leq k<\ell2^{j_{B,n}}$},
\end{cases}
\end{equation*}
where $|W_{jk}|=n_j^{-1}|\sum_{i\in A_{jk}}W_{jk}^{(i)}|\leq n^{-1/2}$, $Z_{jk}=|A_{jk}|^{-1}\sum_{i\in A_{jk}}(\hat{f}_{jk}^{(i)}-\mathbb{E}_{f_0,T}\hat{f}_{jk}^{(i)})$, and recall that $n_j=n|A_{jk}|/m$ for every $j\leq j_{\max}$, $k\in \{0,..,2^{j}-1\}$. Recall also that $n_j\asymp nB/(2^j(\log_2 n)^2)$ for $j\geq j_{B,n}$ and $n_j\asymp n$ for  $j<j_{B,n}$.  

Note that the squared bias satisfies
\begin{align*}
\|\mathbb{E}_{f_0,T}\tilde{f}(j)-f_0\|_2^2\lesssim \|K(f_0,j)-f_0\|_2^2+ 2^j/n
\lesssim 2^{-2js}\|f_0\|_{B_{2,\infty}^s}^2+2^j/n,
\end{align*}
where $\quad K(f_0,j)=\sum_{l=0}^{j-1}\sum_{k=0}^{2^l-1} f_{0,lk}\psi_{lk}$.  Furthermore, also note that {for} $\ell\leq j$ we have $n_\ell\geq n_j$ and hence in view of \eqref{eq: UB:square:noise}
\begin{align*}
\mathbb{E}_{f_0,T}\| \tilde{f}(j)-\mathbb{E}_{f_0,T}\tilde{f}(j)\|_2^2&\lesssim
\sum_{\ell\leq j-1}\sum_{k=0}^{2^{\ell}-1} \big(  \mathbb{E}_{f_0,T}Z_{\ell k}^2+  \mathbb{E}_{f_0,T}W_{\ell k}^2\big) \\
&\lesssim \sum_{\ell\leq j-1}\sum_{k=0}^{2^{\ell}-1} n_\ell^{-1}\leq 2^j/n_j.
\end{align*}

Let us introduce the notation $B(j,f_0)=2^{-2js} \|f_0\|_{B_{2,\infty}^s}^2$ and define the optimal choice of the parameter $j$ (the optimal resolution level) as
\begin{align*}
j^*=\min\big\{ j\in\mathcal{J}:\,  B(j,f_0)\leq  2^j/n_j \big\},
\end{align*}
balancing out the squared bias and variance terms. Note that since the right hand side is monotone increasing and the left hand side is monotone decreasing in $j$,  we have that
\begin{align*}
B(j,f_0)\leq 2^j/n_j,\,\,\text{for $j\geq j^*$}\quad\text{and}\quad B(j,f_0)> 2^j/n_j,\,\,\text{for $j< j^*$}.
\end{align*}
Therefore
\begin{align*}
2^{j^*-1}/n_{j^*-1}< B(j^*-1,f_0)= 2^{2s} B(j^*,f_0)\leq  2^{2s}  2^{j^*}/n_{j^*}.
\end{align*}

Let us distinguish three cases according to the value of $j^*$. If $j^*< j_{B,n}$ then $n_{j^*-1}=n_{j^*}\asymp n$ and therefore $2^{j^*}\asymp n^{1/(1+2s)}$ (using the definition $B(j^{*},f_0)=2^{-2j^{*}s} \|f_0\|_{B_{2,\infty}^s}^2$). Note that the inequality $j^*< j_{B,n}$ is implied by $B(j_{B,n}-1,f_0)\leq 2^{j_{B,n}-1}/n_{j_{B,n}-1}$, which in turns holds if $2^{j_{B,n}-1}\geq (n\|f_0\|_{B_{2,\infty}^s}^2)^{1/(1+2s)}$. Therefore we can conclude that $B\geq 4 L^{2/(1+2s)}n^{1/(1+2s)}\log_2 n$ implies the inequality $j^*< j_{B,n}$ (by recalling that $2^{j_{B,n}}\geq B/(2\log_2 n)$). If $j^*= j_{B,n}$, then $2^{j^*}\asymp B/\log_2 n$, {$n_{j^*}\asymp n/\log_2 n$, $n_{j^*-1}\asymp n$} and therefore $(n/\log_2 n)^{1/(1+2s)}\lesssim 2^{j^*}\lesssim n^{1/(1+2s)}$. Finally, if $j^*> j_{B,n}$, then $n_{j^*-1}\asymp n_{j^*}\asymp nB/(2^{j^*}\log_2^2 n)$ and therefore $2^{j^*}\asymp (nB/\log_2^2 n)^{1/(2+2s)}$. We summarize these findings in the following displays
\begin{equation}\label{eq: opt:resolution}
2^{j^*}\asymp\begin{cases}
n^{1/(1+2s)}, &\text{if $B\geq  C_L n^{1/(1+2s)}\log_2 n$,}\\
B/\log_2 n, &\text{if $ n^{\frac{1}{1+2s}}(\log_2 n)^{\frac{2s}{1+2s}}\leq  B< C_L n^{\frac{1}{1+2s}}\log_2 n$,}\\
(nB/\log_2^2 n)^{1/(2+2s)}&\text{if  $B< n^{1/(1+2s)}(\log_2 n)^{\frac{2s}{1+2s}}$,}
\end{cases}
\end{equation}
and
\begin{equation}\label{eq: opt:resolution_noise}
n_{j^*}\gtrsim\begin{cases}
n, &\text{if $B\geq  C_L n^{1/(1+2s)}\log_2 n$,}\\
n/\log_2 n, &\text{if $ n^{\frac{1}{1+2s}}(\log_2 n)^{\frac{2s}{1+2s}}\leq  B< C_L n^{\frac{1}{1+2s}}\log_2 n$,}\\
(nB/\log_2^2 n)^{\frac{1+2s}{2+2s}}&\text{if  $B\leq n^{1/(1+2s)}(\log_2 n)^{\frac{2s}{1+2s}}$,}
\end{cases}
\end{equation}
where $C_L= 4 L^{2/(1+2s)}$. Note that in all cases $j^*\leq j_{\max}$ holds.

Let us split the risk into two parts
\begin{align}
\mathbb{E}_{f_0,T}\|f_0-\hat{f}\|_2= \mathbb{E}_{f_0,T}\|f_0-\tilde{f}(\hat{j})\|_2 1_{\hat{j}>j^*}+ \mathbb{E}_{f_0,T}\|f_0-\tilde{f}(\hat{j})\|_2 1_{\hat{j}\leq j^*},\label{eq: risk:two:parts}
\end{align}
and deal with each term on the right-hand side separately. First note that
\begin{align*}
\mathbb{E}_{f_0,T}\|f_0-\tilde{f}(\hat{j})\|_2^2 1_{\hat{j}\leq j^*}&\leq  2\mathbb{E}_{f_0,T}\|\tilde{f}(j^*)-\tilde{f}(\hat{j})\|_2^2 1_{\hat{j}\leq j^*}+2 \mathbb{E}_{f_0,T}\|\tilde{f}(j^*)-f_0\|_2^2\\
&\lesssim \tau 2^{j^*}/n_{j^*}+ \|\mathbb{E}_{f_0,T}\tilde{f}(j^*)-f_0\|_2^2+\mathbb{E}_{f_0,T}\|\tilde{f}(j^*)-\mathbb{E}_{f_0,T} \tilde{f}(j^*)\|_2^2\\
&\lesssim  2^{j^*}/n_{j^*}+2^{-2j^*s},
\end{align*}
which implies together with \eqref{eq: opt:resolution} and \eqref{eq: opt:resolution_noise} that
\begin{equation}\label{eq: risk:UB:cases}
\mathbb{E}_{f_0,T}\|f_0-\hat{f}\|_2^2 1_{\hat{j}\leq j^*}\lesssim
\begin{cases}
n^{-2s/(1+2s)},&\text{if  $B\geq  C_L n^{1/(1+2s)}\log_2 n$,}\\
B/n, &\text{if $ n^{\frac{1}{1+2s}}(\log_2 n)^{\frac{2s}{1+2s}}\leq  B\leq C_L  n^{\frac{1}{1+2s}}\log_2 n$,}\\
\Big(nB/\log_2^2 n\Big)^{-\frac{2s}{2+2s}},&\text{if  $B\leq n^{1/(1+2s)}(\log_2 n)^{\frac{2s}{1+2s}}$.}
\end{cases}
\end{equation}

Next we deal with the first term on the right hand side of \eqref{eq: risk:two:parts}. By Cauchy-Schwarz inequality and Lemma \ref{lem: Lem:8.2.1redo} we get that
\begin{align*}
\mathbb{E}_{f_0,T}\|f_0-\hat{f}\|_2 1_{\hat{j}> j^*}&\leq \sum_{j=j^*+1}^{j_{\max}} \mathbb{E}_{f_0,T}^{1/2}\|f_0-\tilde{f}(j)\|_2^2\, \mathbb{P}_{f_0,T}^{1/2}(\hat{j}=j)\\
& \lesssim \sum_{j=j^*+1}^{j_{\max}}  \mathbb{P}_{f_0,T}^{1/2}(\hat{j}=j)\lesssim j_{\max}e^{-(cn^{\delta}\wedge \sqrt{n_r})}+\sum_{k=1}^{\infty} e^{-{(c/2) 2^{j^*}k}}\\
&=o(n^{-1})+o( 2^{- j^*s}),
\end{align*}
resulting in the required upper bound in view of \eqref{eq: risk:UB:cases}, concluding the proof of our statement.

\begin{lemma}\label{lem: Lem:8.2.1redo}
Assume that $f_0\in B_{2,\infty}^s(L)$, for some $s,L>0$. Then there exists a universal constants $c,\delta>0$ such that for every $j>j^*$ we have
\begin{align*}
\mathbb{P}_{f_0,T}(\hat{j}=j)\lesssim e^{-(c 2^{j}\wedge n^{\delta}\wedge \sqrt{n_r})}.
\end{align*}
\end{lemma}

\begin{proof}
Let us introduce the notation $j^{-}=j-1$ and note that for every $j>j^*$ we have $j^{-}\geq j^*$. Then by the definition of $\hat{j}$
\begin{align*}
\mathbb{P}_{f_0,T}(\hat{j}=j)\leq \sum_{l=j}^{j_{\max}}\mathbb{P}_{f_0,T}(\|\tilde{f}(j^-)- \tilde{f}(l)\|_2^2> \tau 2^l/n_l ).
\end{align*}
Note that the left hand side term in the probability in view of Parseval's inequality can be given in the form
\begin{align*}
\|\tilde{f}(j^-)- \tilde{f}(l)\|_2^2
&=\sum_{r =j^-}^{l-1}\sum_{k=0}^{2^{r}-1} \Big( f_{0,r k}+Z_{r k}+W_{r k} \Big)^2\\
&\leq 3\sum_{r =j^-}^{l-1}\sum_{k=0}^{2^{r}-1}\Big(
 f_{0,rk}^2+ Z_{r k}^2+ W_{r k}^2\Big).
\end{align*}
We deal with the three terms on the right hand side separately. Note that the functions $j\mapsto B(j,f_0)$ and $j\mapsto n_j$ are monotone decreasing, hence by the definition of $j^*$ we get for $l\geq j^-\geq j^*$
\begin{align*}
\sum_{r=j^-}^{l-1}\sum_{k=0}^{2^{r}-1}f_{0,r k}^2\leq B({j^{-}},f_0) \leq B(j^*,f_0)\leq 2^{j^*}/n_{j^*}\leq 2^l/n_l.
\end{align*}
Furthermore $\sum_{r =j^-}^{l-1}\sum_{k=0}^{2^{r}-1}W_{r k}^2\leq 2^l/n\leq 2^l/n_l$. 

Let $S(r)=\{\sum_{l=0}^{r}\sum_{k=0}^{2^l-1}b_{lk}\psi_{lk}:\, \sum_{l=0}^{r}\sum_{k=0}^{2^l-1}b_{lk}^2=1 \}$ denote the unite sphere in the linear subspace spanned by the basis functions $\psi_{lk}$, $l\leq r$, $0\leq k\leq 2^{l}-1$. Then in view of Lemma 5.3 of \cite{chagny:2013}, see also Lemma \ref{lem:5.3} in the supplement, and the inequality $(a+b)^2\leq 2a^2+2b^2$ we get that
\begin{align}
\sum_{k=0}^{2^{r}-1}Z_{r k}^2&= \sum_{k=0}^{2^{r}-1}\Big(\frac{1}{n_r}\sum_{i\in A_{r k}}\sum_{\ell=1}^{n/m}\big(Y_{\ell}^{(i)}\psi_{jk}(T^{(i)}_\ell)-\mathbb{E}_{f_0,T}Y_{\ell}^{(i)}\psi_{jk}(T^{(i)}_\ell)\big)\Big)^2\nonumber\\
&\leq 2\sup_{g\in S(r)}\Big(\frac{1}{n_r}\sum_{i\in A_{r k}}\sum_{\ell=1}^{n/m} \big(f_0(T_{\ell}^{(i)})g(T^{(i)}_\ell)-\mathbb{E}_{T}f_0(T_{\ell}^{(i)})g(T^{(i)}_\ell)\big)\Big)^2
\nonumber\\
&\qquad+2\sum_{k=0}^{2^{r}-1}\Big(\frac{1}{n_r}\sum_{i\in A_{r k}}\sum_{\ell=1}^{n/m}\big(\eps_{\ell}^{(i)}\psi_{jk}(T^{(i)}_\ell)\big)\Big)^2.\label{eq: UB:chi:square:term}
\end{align}

We deal with the two terms on the right hand side separately, starting with the first one. Note that for every $g\in S(r)$ the inequality $\|g\|_{\infty}\leq C 2^{r/2}$ holds, for some universal constant $C>0$ and 
\begin{align*}
 {\sup_{g\in S(r)}}  V_T\big( f_0(T_1^{(1)})g(T_1^{(1)})\big)\leq \|f_0\|_\infty^2.
\end{align*}
Next for convenience let us introduce the notation 
$$\nu(g)=\frac{1}{n_r}\sum_{i\in A_{r k}}\sum_{\ell=1}^{n/m}\Big(f_0(T_{\ell}^{(i)})g(T^{(i)}_\ell)-\mathbb{E}_{T}f_0(T_{\ell}^{(i)})g(T^{(i)}_\ell)\Big).$$
 Then by the definition of $S(r)$ and Cauchy-Schwarz inequality
\begin{align*}
\mathbb{E}_T {\sup_{g\in S(r)}} \big|\nu(g)\big|\leq  \sum_{k=0}^{2^{r}-1} \mathbb{E}_T \big(\nu(\psi_{r,k})^2\big)
=\sum_{k=0}^{2^{r}-1}\frac{1}{n_r}V_T\big( f_0(T_{1}^{(1)})\psi_{rk}(T^{(1)}_1)\big)\leq \frac{\|f_0\|_\infty^2 2^r}{n_r}.
\end{align*}
Therefore in view of Lemma 5 of \cite{lacour:08}, see also Lemma \ref{lem4:chagny} in the supplement, there exist constants $c_1,c_2,c_2>0$ such that
\begin{align}
\mathbb{E}_{T}& \sup_{g\in S(r)}\Big[\Big(\frac{1}{n_r}\sum_{i\in A_{r k}}\sum_{\ell=1}^{n/m} \big(f_0(T_{\ell}^{(i)})g(T^{(i)}_\ell)-\mathbb{E}_{T}f_0(T_{\ell}^{(i)})g(T^{(i)}_\ell)\big)\Big)^2- c_1 2^r/n_r \Big]_+\nonumber\\
&\quad\leq c_2\frac{1}{n_r}e^{-c_3 2^{r}}+c_4\frac{2^{r}}{n_r^2}e^{-\sqrt{n_r}}\lesssim \frac{1}{n_r}e^{-(c_3 2^{r}\wedge \sqrt{n_r})}.\label{eq: bound:mean}
\end{align}
Therefore by Markov's inequality we get that
\begin{align}
\mathbb{P}_{T}\Big(\sup_{g\in S(r)}\Big(\frac{1}{n_r}\sum_{i\in A_{r k}}\sum_{\ell=1}^{n/m} \big(f_0(T_{\ell}^{(i)})g(T^{(i)}_\ell)-\mathbb{E}_{T}f_0(T_{\ell}^{(i)})g(T^{(i)}_\ell)\big)\Big)^2  \geq \frac{2c_1 2^r}{n_r} \Big)\lesssim 2^{-r}e^{-(c_3 2^{r}\wedge \sqrt{n_r})}.\label{eq: bound:mean:Markov}
\end{align}

Next we deal with the second term on the right hand side of \eqref{eq: UB:chi:square:term}. Let us introduce the shorthand notation $\tilde{Z}_{rk}=n_r^{-1}\sum_{i\in A_{r k}}\sum_{\ell=1}^{n/m}\eps_{\ell}^{(i)}\psi_{rk}(T^{(i)}_\ell)$. Note that $cov(\tilde{Z}_{rk},\tilde{Z}_{rk'}|T)=0$ for $|k-k'|\geq C$, for some large enough constant $C$, following from the disjoint support of the wavelet basis functions $\psi_{rk}$ and $\psi_{rk'}$, and 
$$\tilde{Z}_{rk}|T\sim N\big(0, \frac{1}{n_r^2}\sum_{i\in A_{r k}}\sum_{\ell=1}^{n/m}\psi_{rk}(T^{(i)}_\ell)^2\big).$$ 
Furthermore, let us denote by $\mathcal{B}_r$ the event that in each bin $I_{r,l}=[(l-1)2^{-r},l2^{-r}]$, at most $2n_r/2^{r}$ observations $T_{\ell}^{(i)}$, $i\in A_{rk}$, $\ell=1,...,m$, $k=0,...,2^{r}-1$ fall. Since there are $2^{r-j_{B_n}}\leq 2^r$ subgroups of machines at resolution level $r$ we   note that in view of Lemma \ref{lem: design} we have that $\mathbb{P}_T(\mathcal{B}_r^c)\leq   2^{2r+1} e^{-n_r2^{-r-3}}$. Then by recalling that for $r\leq j_{B,n}$, $n_r\asymp n$,  while for $r>j_{B,n}$,  $n_r=nB/(2^r\log^2n)$, we get that $n_r/2^r\gtrsim (nB/\log_2^2 n)^{\frac{2s_{\min}}{2+2s_{\min}}}\wedge n^{\frac{2s_{\min}}{1+2s_{\min}}}$, hence 
\begin{align}
\mathbb{P}_T(\mathcal{B}_r^c)\lesssim e^{-n^{\delta}},\quad \text{for any $\delta< 2s_{\min}/(2+2s_{\min})$},\label{eq: UB:Bcomp}
\end{align}
and on $\mathcal{B}_r$ the inequality $ n_r^{-2}\sum_{i\in A_{r k}}\sum_{\ell=1}^{n/m}\psi_{rk}(T^{(i)}_\ell)^2\leq C n_r^{-1}$ holds, for some sufficiently large $C>0$. Let us denote the covariance matrix of the random vector $ (\tilde{Z}_{r0},...,\tilde{Z}_{r(2^{r}-1)})|T$  by $\Sigma_T$. In view of the preceding argument the in absolute value largest entry of $\Sigma_T$ is bounded from above by $C n_r^{-1}$ on the event $T\in \mathcal{B}_r$ and by noting that $\Sigma_T$ has band size $C$, in view of Gershgorin circle theorem \cite{gershgorin:1931}, see also Lemma \ref{lem:Gershgorin} in the supplement, the eigenvalues of $\Sigma_T$  satisfy that $0<\lambda_i\leq C n_r^{-1}$, $i=1,..,2^{r}$. Then by the tail bounds of chi-square distributions, see for instance Theorem 4.1.9 of \cite{gine:nickl:2016} (or Lemma \ref{lem:chisquare} in the supplement),
\begin{align*}
 \mathbb{P}_{f_0} \big(\sum_{k=0}^{2^{r}-1} \tilde{Z}_{rk}^2 \geq \frac{C_1 2^{r}}{n_r} |T=t\big)=  \mathbb{P}\big(\sum_{i=1}^{2^{r}}\lambda_i\zeta_i^2\geq  \frac{C_1 2^{r}}{n_r}\big)\leq \mathbb{P}\big(\sum_{i=1}^{2^{r}}\zeta_i^2\geq  C_2 2^{r}\big)\lesssim e^{-C_3 2^r},
\end{align*}
for some sufficiently large constants $C_1,C_2>0$ and small $C_3>0$, where $\zeta_i\stackrel{iid}{\sim}N(0,1)$. Hence we can conclude that
\begin{align*}
 \mathbb{P}_{f_0,T} &\Big(\sum_{k=0}^{2^r-1} \big( \frac{1}{n_r}\sum_{i\in A_{r k}}\sum_{\ell=1}^{n/m}\eps_{\ell}^{(i)}\psi_{rk}(T^{(i)}_\ell)\big)^2\geq \frac{C_12^r}{n_r}\Big)\\
&\leq \int_{t\in \mathcal{B}_r}  \mathbb{P}_{f_0}\big(\sum_{k=0}^{2^{r}-1}\tilde{Z}_{rk}^2 \geq \frac{C_12^r}{n_r} |T=t \big)dt+\mathbb{P}_T(\mathcal{B}_r^c)\lesssim e^{-(C 2^{r}\wedge n^{\delta})},
\end{align*}
finishing the proof of the lemma.
\end{proof}

\section{Proofs for the $L_{\infty}$-norm}

\subsection{Proof of Theorem \ref{theorem: minimaxLinftyLB}}\label{sec: minimaxLinftyLB}
First of all we note that in the non-distributed case, where all the information is available in the central machine, the minimax $L_{\infty}$-risk is $(n/\log n)^{-\frac{s}{1+2s}}$. Since the class of distributed estimators is clearly a {subset of the class of all} estimators this will be also a lower bound for the distributed case. The rest of the proof goes similarly to the proof of Theorem \ref{sec: minimaxL2LB}. 

We consider the same subset of functions $\FF_0$ as in the proof of Theorem \ref{sec: minimaxL2LB}, with functions given by \eqref{eq: function_minimax}. Note that each function $f_\beta\in\FF_0$ belongs to the set $B_{\infty,\infty}^s(L)$, since
\begin{align*}
\|f_\beta\|_{B_{\infty,\infty}^s}=\sup_{j}2^{(s+1/2)j}\sup_{k=0,...,2^j-1}f_{\beta, jk}=2^{(s+1/2)j_n}L\delta_n^{1/2}\leq  L.
\end{align*} 
Furthermore, if $f_\beta \not = f_{\beta'}$, {then there} exists a $k\in K_{j_n}$ such that $\beta_{k}\neq \beta'_{k}$. Then due to the disjoint support of the corresponding 
Daubechies wavelets $\psi_{j_n,k}$, $k\in K_{j_n}$ the $L_{\infty}$-distance between 
the two functions  is bounded from below by
\begin{align*}
\|f_\beta-f_{\beta'}\|_{\infty}\geq  |f_{\beta,j_nk}-f_{\beta',j_nk}|\cdot \|\psi_{j_n,k}\|_{\infty}
{\gtrsim} 2^{j_n/2}\delta_n^{1/2}\geq \delta_n^{\frac{s}{1+2s}}.
\end{align*}

Now let $F$ be a uniform random variable on the set $\FF_0$. Then in view of Fano's inequality 
(see for instance Theorem \ref{thm: fano}  in the supplement with $\delta=\delta_n^{s/(1+2s)}$ and $p=1$) we get that
\begin{align*}
\inf_{\hat{f}\in\mathcal{F}_{dist}(B^{(1)},\ldots,B^{(m)};B_{\infty,\infty}^s(L) )}\sup_{f_0\in  B_{\infty,\infty}^s(L)} \mathbb{E}_{f_0,T}\Big(\|\hat{f}-f_0\|_{\infty} \Big)
\gtrsim
\delta_n^{\frac{s}{1+2s}}\Big(1-\frac{I(F;Y)+\log 2}{\log_2 |\FF_0|} \Big).
\end{align*}
We conclude the proof by noting that the term in the bracket on the right hand side of the preceding display is bounded from below by a constant, see the proof of Theorem \ref{sec: minimaxL2LB}.

\subsection{Proof of Theorem \ref{theorem: minimaxLinftyUB}}\label{sec: minimaxLinftyUB}
Similarly to the proof of Theorem \ref{theorem: minimaxL2UB1} we get that $\mathbb{E}_{f_0,T}[l({Y_{jk}^{(i)}})]\leq (1/2+o(1))\log_2 n$, hence we need at most  $(1/2+o(1))B$ bits in expected value to transmit the $\lfloor B/\log_2 n\rfloor\wedge (n/\log_2 n)^{1/(1+2s)}$ approximated coefficients. Therefore the total amount of transmitted bits per machine is bounded from above by $B$ (for large enough $n$), hence $\hat{f}\in\mathcal{F}_{dist}(B,\ldots,B; B_{\infty,\infty}^s(L))$. 

Similarly to the proof of Theorem \ref{theorem: minimaxL2UB2}, let $A_{jk}=\{ \lfloor \mu_{jk} m/\eta\rfloor +1,...,\lfloor (\mu_{jk}+1) m/\eta\rfloor \}$ be the collection of machines transmitting the $(j,k)$th approximated wavelet coefficient {and note that the size of the set satisfies $|A_{jk}|\asymp m/\eta$.} And recall that the aggregated estimator $\hat{f}$ satisfies for $2^j+k\leq (\eta   \lfloor B/\log_2 n\rfloor)\wedge (n/\log_2 n)^{1/(1+2s)} $ (i.e. the total number of different coefficients transmitted) that 
\begin{align*}
\hat{f}_{jk}&=\frac{1}{|A_{jk}|}\sum_{i\in A_{jk}}Y_{jk}^{(i)}= f_{0,jk}+Z_{jk}+W_{jk},
\end{align*}
where  $|W_{jk}|=|A_{jk}|^{-1}|\sum_{i\in A_{jk}}W_{jk}^{(i)}|\leq n^{-1/2}$ and $Z_{jk}=|A_{jk}|^{-1}\sum_{i\in A_{jk}} (\hat{f}_{jk}^{(i)}- \mathbb{E}_{f_0,T}\hat{f}_{jk}^{(i)})$. 
We show below that for all $2^j\leq n/\eta$,
\begin{align}
\mathbb{E}_{f_0,T}\sup_k |Z_{jk}|\lesssim \sqrt{(\log_2 n)\eta/n}.\label{eq: UB:supZjk}
\end{align}

Next note that by triangle inequality
\begin{align*}
\mathbb{E}_{f_0,T} \|f_0-\hat{f}\|_{\infty}\leq  \|f_0-\mathbb{E}_{f_0,T}\hat{f}\|_{\infty}+ \mathbb{E}_{f_0,T}\|\hat{f}-\mathbb{E}_{f_0,T}\hat{f}\|_{\infty}.
\end{align*}
We deal with the two terms on the right hand side separately. Let us introduce the notation 
$$j_n=\lfloor\log_2\big( (\eta \lfloor B/\log_2 n\rfloor)\wedge (n/\log_2 n)^{1/(1+2s)} \big)\rfloor\leq \log_2 (n/\eta).$$
Then by triangle inequality and noting that  there exists a universal constant $C>0$ such that for each resolution level $j$ the inequality $\big\|\sum_{k=0}^{2^j-1}|\psi_{jk}|\big\|_{\infty}\leq C 2^{j/2}$ holds,
\begin{align}
\|f_0-\mathbb{E}_{f_0,T}\hat{f}\|_{\infty}&\leq\|\sum_{j=j_n}^{\infty}\sum_{k=0}^{2^j-1}f_{0,jk}\psi_{jk} \|_{\infty}+\|\sum_{j=0}^{j_n} \sum_{k=0}^{2^j-1}\mathbb{E}_{f_0,T}W_{jk}\psi_{jk}\|_{\infty}\nonumber\\
&\leq  \|f_0\|_{B_{\infty,\infty}^s}\sum_{j=j_n}^{\infty}2^{-j(s+1/2)}  \big\|\sum_{k=0}^{2^j-1}|\psi_{jk}|\big\|_{\infty}+ n^{-1/2}\sum_{j=0}^{j_n} \big\|\sum_{k=0}^{2^j-1}|\psi_{jk}|\big\|_{\infty}\nonumber\\
&\lesssim \sum_{j=j_n}^{\infty}2^{-js}+  \sqrt{2^{j_n}/n}\lesssim 2^{-j_ns}+\sqrt{2^{j_n}/n}.\label{eq: UB:bias:Linfty}
\end{align}
Furthermore, in view of \eqref{eq: UB:supZjk},
\begin{align}
\mathbb{E}_{f_0,T}\|\hat{f}-\mathbb{E}_{f_0,T}\hat{f}\|_{\infty}
&\leq \sum_{j=0}^{j_n}   \mathbb{E}_{f_0,T}\max_{k}(|Z_{jk}| +|W_{jk}|)\big\|\sum_{k=0}^{2^j-1}|\psi_{jk}|\big\|_{\infty} \nonumber\\
&\lesssim \sum_{j=0}^{j_n} 2^{j/2}\Big(\sqrt{(\log_2 n)\eta/n}+\sqrt{1/n} \Big)  \lesssim 
\sqrt{2^{j_n}\eta (\log_2 n)/n},\label{eq: UB:var:Linfty}
\end{align}
providing the upper bound in the statement of the lemma.

It remained to prove assertion \eqref{eq: UB:supZjk}. First note that
\begin{align*}
Z_{jk}|T&\sim N (\mu_{n,m,k,T},\sigma^2_{n,m,k,T}),\qquad\text{with}\\
\mu_{n,m,k,T}&=\frac{\eta}{n}\sum_{i\in A_{jk}}\sum_{\ell=1}^{n/m}\psi_{jk}(T^{(i)}_{\ell})f_{0}(T^{(i)}_{\ell})-f_{0,jk}\lesssim 2^{j/2},\\
\sigma^2_{n,m,k,T}&=  (\frac{\eta}{n})^2\sum_{i\in A_{jk}}\sum_{\ell=1}^{n/m}\psi_{jk}^2(T^{(i)}_{\ell})\lesssim 2^j\eta/n.
\end{align*}
Using standard  bounds on the maximum of Gaussian variables (see for instance Lemma 3.3.4 of \cite{gine:nickl:2016}) we have that
\begin{align*}
\mathbb{E}_{f_0|T} \max_{k} |Z_{jk}-\mathbb{E}_{f_0|T}Z_{jk}|\leq \sqrt{2(j+1)}\max_k\sigma_{n,m,k,T}.
\end{align*}
Furthermore, note that for $k\geq 2$
\begin{align*}
\mathbb{E}_{T}\big(\psi_{jk}(T^{(i)}_{\ell})f_{0}(T^{(i)}_{\ell}) \big)_+^k\leq \|f_0\|_{\infty}^k \|\psi_{jk} \|_{\infty}^{k-2}  \mathbb{E}_{T}\psi_{jk}(T^{(i)}_{\ell})^2\lesssim 2^{(k-2)j/2},
\end{align*}
hence in view of Bernstein's inequality (with $c=C 2^{j/2}$ and $v=Cn/\eta$), see Proposition 2.9 of \cite{massart2007concentration} (or Lemma \ref{lem: Bernstein} in the supplement), we get that
\begin{align*}
\mathbb{P}_T\Big(|\mu_{n,m,k,T}|\geq C (\sqrt{\frac{\gamma \eta\log_2 n}{n}}+\frac{2^{j/2}\eta}{n} )\Big)\lesssim (n/\eta)^{-\gamma},
\end{align*}
which implies for $2^j\leq n/\eta$ that
\begin{align}
\mathbb{P}_T\Big(\max_k|\mu_{n,m,k,T}|\geq C_\gamma \sqrt{ (\log_2 n)\eta/n} \Big)\lesssim (n/\eta)^{-\gamma+1}.\label{eq: UB:exp:cond}
\end{align}
Therefore one can deduce that
\begin{align*}
\mathbb{E}_T \Big( \max_{k}|\mu_{n,m,k,T}|\Big)&\leq C_\gamma  \sqrt{(\log_2 n)\eta/n} + 2^{j/2}(n/\eta)^{-\gamma+1}\\
&\lesssim  \sqrt{ (\log_2 n)\eta/n},
 \end{align*}
for large enough choice of $\gamma>0$. 
Combining the above displays leads to
\begin{align*}
\mathbb{E}_{f_0,T}\max_k|Z_{jk}| &=\mathbb{E}_T (\mathbb{E}_{f_0|T}(\max_k|Z_{jk}|) )\\
&\leq \mathbb{E}_T\big( \max_{k}|\mu_{n,m,k,T}| \big)
+\sqrt{2(j+1)}\mathbb{E}_T\max_k\sigma_{n,m,k,T}\\
&\leq c( \sqrt{(\log_2 n)\eta/n}+2^{j/2}\sqrt{j}e^{-c n^{\delta}})\leq C \sqrt{(\log_2 n)\eta/n},
\end{align*}
for some large enough constants $c,C>0$ and $2^{j}\leq n/\eta$, where in the last line we have used that under the event  
$\mathcal{B}_j$ (i.e. the event that in each bin $I_{j,l}=[(l-1)2^{-j},l2^{-j}]$, at most $2n/(\eta 2^{j})$ observations $T_{\ell}^{(i)}$, $i\in A_{jk}$, $\ell=1,...,n/m$, $k=0,...,2^{j}-1$ fall) we have that $\max_k\sigma_{n,m,k,T}^2\leq C$, and $\mathbb{P}_T(\mathcal{B}_j^c)\leq C e^{-cn^{\delta}}$, see \eqref{eq: UB:Bcomp}. 

\subsection{Proof of Theorem \ref{theorem: adaptive:Linfty}}\label{sec: adaptive:Linfty}

The proof of the theorem goes similarly to the proof of Theorem \ref{theorem: adaptive}, here we only highlight the differences. First recall that for every $s,L>0$ and $f_0\in B_{\infty,\infty}^s(L)$ we have $f_{0,jk}\leq L$, for all $j\geq 0, k\in \{0,1,...,2^j-1\}$, hence following from the same argument as in Theorem \ref{theorem: adaptive}, the estimator belongs to $\mathcal{F}_{dist}(B,\ldots,B; B_{\infty,\infty}^{s}(L))$.

Let us next introduce the notations
$B(j,f_0)=2^{-js} \|f_0\|_{B_{\infty,\infty}^s}$ and
\begin{align*}
j^*=\min\big\{ j\in\mathcal{J}:\,  B(j,f_0)\leq  \sqrt{j 2^j/n_j} \big\}.
\end{align*}
Then by the definition of $j^*$ we have
\begin{align*}
\sqrt{(j^*-1)2^{j^*-1}/n_{j^*-1}}< B(j^*-1,f_0)= 2^{s} B(j^*,f_0)\leq  2^{s}\sqrt{  j^*2^{j^*}/n_{j^*}}.
\end{align*}
Distinguish again three cases according to the value of $j^*$ we get that 
\begin{equation}\label{eq: opt:resolution:Linfty}
2^{j^*}\asymp\begin{cases}
(n/\log_2 n)^{1/(1+2s)}, &\text{if $B\geq  C_L (n/\log_2 n)^{1/(1+2s)}\log_2 n$,}\\
B/\log_2 n, &\text{if $ (\frac{n}{\log_2 n})^{\frac{1}{1+2s}}(\log_2 n)^{\frac{2s}{1+2s}}\leq  B< C_L  (\frac{n}{\log_2 n})^{\frac{1}{1+2s}}\log_2 n$,}\\
(nB/\log_2^3 n)^{1/(2+2s)},&\text{if  $B< (n/\log_2 n)^{1/(1+2s)}(\log_2 n)^{\frac{2s}{1+2s}}$,}
\end{cases}
\end{equation}
and
\begin{equation}\label{eq: opt:resolution_noise:Linfty}
n_{j^*}\gtrsim\begin{cases}
n, &\text{if $B\geq  C_L (n/\log_2 n)^{1/(1+2s)}\log_2 n$,}\\
n/\log_2 n, &\text{if $  (\frac{n}{\log_2 n})^{\frac{1}{1+2s}}(\log_2 n)^{\frac{2s}{1+2s}}\leq  B< C_L  (\frac{n}{\log_2 n})^{\frac{1}{1+2s}}\log_2 n$,}\\
(nB/\log_2^{\frac{1+4s}{1+2s}} n)^{\frac{1+2s}{2+2s}},&\text{if  $B<  (n/\log_2 n)^{1/(1+2s)}(\log_2 n)^{\frac{2s}{1+2s}}$,}
\end{cases}
\end{equation}
where $C_L= 4 \big(L^2(1+2s)\big)^{1/(1+2s)}$. Note that in all cases $j^*\leq j_{\max}$ holds.

We split the risk into two parts
\begin{align}
\mathbb{E}_{f_0,T}\|f_0-\hat{f}\|_{\infty}= \mathbb{E}_{f_0,T}\|f_0-\tilde{f}(\hat{j})\|_{\infty} 1_{\hat{j}>j^*}+ \mathbb{E}_{f_0,T}\|f_0-\tilde{f}(\hat{j})\|_{\infty} 1_{\hat{j}\leq j^*}\label{eq: risk:two:parts:Linfty}
\end{align}
and deal with each term on the right-hand side separately. Note that in view of the definition of $\hat{j}$ and assertions \eqref{eq: UB:bias:Linfty} and \eqref{eq: UB:var:Linfty}
\begin{align*}
\mathbb{E}_{f_0,T}\|f_0-\tilde{f}(\hat{j})\|_\infty 1_{\hat{j}\leq j^*}&\leq  \mathbb{E}_{f_0,T}\|\tilde{f}(j^*)-\tilde{f}(\hat{j})\|_{\infty} 1_{\hat{j}\leq j^*}+ \mathbb{E}_{f_0,T}\|\tilde{f}(j^*)-f_0\|_{\infty}\\
&\leq \tau\sqrt{ j^* 2^{j^*}/n_{j^*}}+ \|\mathbb{E}_{f_0,T}\tilde{f}(j^*)-f_0\|_\infty+\mathbb{E}_{f_0,T}\|\tilde{f}(j^*)-\mathbb{E}_{f_0,T} \tilde{f}(j^*)\|_{\infty}\\
&\lesssim  \sqrt{(\log_2 n)2^{j^*}/n_{j^*}}+2^{-j^*s},
\end{align*}
which implies together with \eqref{eq: opt:resolution:Linfty} and \eqref{eq: opt:resolution_noise:Linfty} that
\begin{equation*}
\mathbb{E}_{f_0,T}\|f_0-\hat{f}\|_{\infty} 1_{\hat{j}\leq j^*}\lesssim
\begin{cases}
\Big(\frac{n}{\log_2 n}\Big)^{-\frac{s}{1+2s}},&\text{if  $B\geq  C_L (n/\log_2 n)^{1/(1+2s)}\log_2 n$,}\\
 \sqrt{\frac{B\log_2 n}{n}}, &\text{if $ (\frac{n}{\log_2 n})^{\frac{1}{1+2s}}(\log_2 n)^{\frac{2s}{1+2s}}\leq  B\leq C_L    (\frac{n}{\log_2 n})^{\frac{1}{1+2s}}\log_2 n$,}\\
\Big(\frac{nB}{\log_2^3 n}\Big)^{-\frac{s}{2+2s}},&\text{if  $B\leq (\frac{n}{\log_2 n})^{\frac{1}{1+2s}}(\log_2 n)^{\frac{2s}{1+2s}}$.}
\end{cases}
\end{equation*}

Next we deal with the first term on the right hand side of \eqref{eq: risk:two:parts:Linfty}. First note that in view of \eqref{eq: UB:bias:Linfty},
\begin{align*}
\|f_0-\mathbb{E}_{f_0,T}\tilde{f}(j)\|_{\infty}^2&\lesssim  2^{-2js}+2^{j}/n.
\end{align*}
Furthermore, by using the upper bound $\psi_{lk}^2\lesssim 2^l$
\begin{align*}
\mathbb{E}_{f_0,T}\|\tilde{f}(j)-\mathbb{E}_{f_0,T}\tilde{f}(j)\|_{\infty}^2
&\lesssim\mathbb{E}_{f_0,T}\Big( \sup_{x\in[0,1]}   \sum_{l=0}^{j}  \sum_{k=0}^{2^l-1}|\psi_{lk}(x)|  (|Z_{lk}| +|W_{lk}|)\Big)^2\\
&\lesssim 2^{2j}\mathbb{E}_{f_0,T} \sum_{l=0}^{j}  \sum_{k=0}^{2^l-1} \big(Z_{lk}^2+W_{lk}^2\big)\\
&\lesssim 2^{3j}(\mathbb{E}_{f_0,T}Z_{lk}^2+ n^{-1})\lesssim 2^{3j}.
\end{align*}

Then by Cauchy-Schwarz inequality and Lemma \ref{lem: Lem:8.2.1redo:Linfty} we get that
\begin{align*}
\mathbb{E}_{f_0,T}\|f_0-\hat{f}\|_\infty 1_{\hat{j}> j^*}&\leq \sum_{j=j^*+1}^{j_{\max}} \mathbb{E}_{f_0,T}^{1/2}\|f_0-\tilde{f}(j)\|_\infty^2\, \mathbb{P}_{f_0,T}^{1/2}(\hat{j}=j)\\
& \lesssim \sum_{j=j^*+1}^{j_{\max}} 2^{(3/2)j} \mathbb{P}_{f_0,T}^{1/2}(\hat{j}=j)\lesssim
2^{j^*}e^{-c\tau^2 j^*}+2^{(3/2)j_{\max}} n^{-2}\\
&=o( 2^{- j^*s}+1/\sqrt{n}),
\end{align*}
for sufficiently large choice of $\tau>0$, resulting in the required upper bound and concluding the proof of our statement.

\begin{lemma}\label{lem: Lem:8.2.1redo:Linfty}
Assume that $f_0\in B_{\infty,\infty}^s(L)$, for some $s,L>0$. Then for every $C>0$ there exist positive constants $c>0$ such that for every $j>j^*$ and sufficiently large $\tau>0$ we have
\begin{align*}
\mathbb{P}_{f_0,T}(\hat{j}=j)\lesssim e^{-c\tau^2 j}+n^{-2}.
\end{align*}
\end{lemma}

\begin{proof}
Let us introduce the notation $j^{-}=j-1$ and note that for every $j>j^*$ we have $j^{-}\geq j^*$. Then by the definition of $\hat{j}$
\begin{align*}
\mathbb{P}_{f_0,T}(\hat{j}=j)\leq \sum_{l=j}^{j_{\max}}\mathbb{P}_{f_0,T}\big(\|\tilde{f}(j^-)- \tilde{f}(l)\|_\infty> \tau \sqrt{l2^l/n_l} \big).
\end{align*}
By triangle inequality
\begin{align*}
\|\tilde{f}(j^-)- \tilde{f}(l)\|_\infty&\leq \|\tilde{f}(j^-)-\mathbb{E}_{f_0,T}\tilde{f}(j^-)\|_\infty + \|\tilde{f}(l)-\mathbb{E}_{f_0,T}\tilde{f}(l)\|_\infty\\
&\qquad+\| \mathbb{E}_{f_0,T}\tilde{f}(j^-)-\mathbb{E}_{f_0,T}\tilde{f}(l) \|_\infty.
\end{align*}
We deal with the terms on the right hand side separately. First note that
\begin{align*}
\| \mathbb{E}_{f_0,T}\tilde{f}(j^-)-\mathbb{E}_{f_0,T}\tilde{f}(l) \|_\infty
&\leq\|\sum_{r= j^-}^l\sum_{k=0}^{2^r-1}f_{0,rk}\psi_{rk} \|_{\infty}+\|\sum_{r= j^-}^l\sum_{k=0}^{2^r-1}\mathbb{E}_{f_0,T}W_{rk}\psi_{rk}\|_{\infty}\nonumber\\
&\leq c(\|f_0\|_{B_{\infty,\infty}^s} 2^{-j^-s}+\sqrt{2^{l}/n})\leq C(B(j^-,f_0)+\sqrt{2^{l}/n})\\
&\leq C\big(B(j^*,f_0)+\sqrt{2^{l}/n}\big)\leq C (\sqrt{j^*2^{j^*}/n_{j^*}}+\sqrt{2^{l}/n})\\
&\leq   C \sqrt{l2^{l}/n_l}.
\end{align*}
Furthermore,
\begin{align*}
\|\tilde{f}(l)-\mathbb{E}_{f_0,T}\tilde{f}(l)\|_\infty&\leq \sum_{j=0}^{l}\max_k (|Z_{jk}|+|W_{jk}|)  \sup_{x\in[0,1]}\sum_{k=0}^{2^j-1}|\psi_{jk}(x)|\\
&\leq C\big(\sum_{j=0}^{l}2^{j/2} \max_k |Z_{jk}| +\sqrt{2^l/n} \big).
\end{align*}
We show below that for any $\gamma\geq 1$,
\begin{align}
\mathbb{P}_{f_0,T} \big(\max_{k} (|Z_{lk}| \geq \tau \sqrt{\gamma l/n_l} \big)\lesssim  n_l^{1-\gamma}+ e^{-c\tau^2 l } \label{eq: UB:max:Gaussian}
\end{align}
holds for some sufficiently large $\tau>0$ and sufficiently small $c>0$. By combining the above results we get that 
\begin{align*}
\mathbb{P}_{f_0,T} \Big( \|\tilde{f}(j^-)- \tilde{f}(l)\|_\infty \geq \tau \sqrt{l2^l/n_l} \Big)
&\lesssim \mathbb{P}_{f_0,T} \Big( \|\tilde{f}(l)-\mathbb{E}_{f_0,T}\tilde{f}(l)\|_\infty \geq \frac{\tau-C}{2} \sqrt{l2^l/n_l} \Big)\\
&\lesssim \sum_{j=0}^l \mathbb{P}_{f_0,T} \Big( \max_{k} |Z_{jk}| \geq \frac{\tau-2C}{2C} \sqrt{l/n_l}  \Big)\\
&\leq l \mathbb{P}_{f_0,T} \Big( \max_{k} |Z_{lk}| \geq \frac{\tau-2C}{2C} \sqrt{l/n_l}  \Big)\\
&\lesssim   (\log_2 n)n_l^{1-\gamma}+ e^{-(c/2)\tau^2 l }.
\end{align*}
The above inequality together with the first display of the proof then implies that
\begin{align*}
\mathbb{P}_{f_0,T}(\hat{j}=j)&\lesssim \sum_{l=j}^{j_{\max}}\big( (\log_2 n)n_{l}^{1-\gamma}+e^{-(c/2)\tau^2l}\big)\\
&\lesssim (\log_2  n)^2n_{j_{\max}}^{1-\gamma}+ e^{-(c/2)\tau^2j}\lesssim n^{-2}+e^{-(c/2)\tau^2j},
\end{align*}
for $\gamma\geq 5$, in view of $n_{j_{\max}}\gtrsim (nB/\log_2^2 n)^{\frac{1+2s_{\min}}{2+2s_{\min}}}\geq \sqrt{n}$, for any $s_{\min}>0$, providing the statement of the lemma.

It remained to prove assertion \eqref{eq: UB:max:Gaussian}. Note that by triangle inequality we get that
\begin{align}
\max_{k} |Z_{lk}|\leq \max_{k} |Z_{lk}-\mathbb{E}_{f_0|T} Z_{lk}| + \max_{k} |\mathbb{E}_{f_0|T} Z_{lk}|. \label{eq: help:UB:max}
\end{align}
In view of assertion \eqref{eq: UB:exp:cond} with  $\mathbb{P}_T$-probability at least $1-Cn_l^{1-\gamma}$ the second term on the right hand side is bounded from above by $C\sqrt{\gamma l/n_l}$.
Furthermore recall from the proof of Theorem \ref{theorem: minimaxLinftyUB} (i.e. assertion \eqref{eq: UB:Bcomp}) that $n_l^{-2}\sum_{i\in A_{lk}}\sum_{\ell=1}^{n/m} \psi_{lk}^2( T_\ell^{(i)})\lesssim n_l^{-1}$ holds with $\mathbb{P}_T$- probability at least $1-Ce^{-cn^\delta}$, for some sufficiently small $\delta>0$. Under the above event we have that there exists small enough constant  $c>0$ such that
\begin{align*}
\mathbb{P}_{f_0|T}\big(|Z_{l 1}-\mathbb{E}_{f_0|T}Z_{l1}| \geq \tau \sqrt{l/n_l}\big)\leq \exp\{- c\tau^2 l\}.
\end{align*}
Therefore the first term on the right hand side of \eqref{eq: help:UB:max} is bounded from above by $\tau \sqrt{l/n_l}$ with $\mathbb{P}_{f_0|T}$-probability at least $1-C 2^le^{-c\tau^2 l}\leq 1-Ce^{-(c/2)\tau^2 l}$ on $T\in\mathcal{B}_l$, for some sufficiently large constants $\tau,C>0$ and sufficiently small positive constant $c$.

\end{proof}

\section{Technical lemmas}
In this section we provide the technical lemmas applied in the previous two sections.
\subsection{Proof of Lemma \ref{lem: mutual_regression}}\label{sec: proof: lem:reg}
Without loss of generality we can assume that $T_1^{(i)}\leq T_2^{(i)}\leq...\leq T_{n/m}^{(i)}$, $i=1,...,m$, and let $\ell_k=\ell_k^{(i)}=\max\big\{j\in\{1,...,n/m\} :\, T_j^{(i)}\in I_k\big\}$ denote the index of the largest element $T_{j}^{(i)}$ in the interval $I_k=[(k-1)\tilde{C}2^{-j_n},k\tilde{C}2^{-j_n}]$, $k=1,...,|K_{j_n}|=2^{j_n}/\tilde{C}$. Note that $ T_{\ell_{k-1}+1}^{(i)},..., T_{\ell_{k}}^{(i)}\in I_k$. For convenience let us introduce the following notations
\begin{align*}
X_{[j_1:j_2]}^{(i)}&=(X_{j_1}^{(i)},X_{j_1+1}^{(i)},...,X_{j_2}^{(i)}),\\
d&=|K_{j_n}|,\\
F_{-k}&=(F_1,..,F_{k-1},F_k,...,F_{d}),\\
\delta&=L\delta_n^{1/2}2^{j_n/2}\|\psi\|_\infty,\\
a^2&=\frac{2^5 n \delta^2}{dm/ \log(dm)},\\
\mu_k(t)&=\big(L\delta_n^{1/2}\psi_{j_n,k}(t_j)\big)_{j=(\ell_{k-1}+1),...,\ell_k},\\
B_k(t)&=\big\{x\in\mathbb{R}^{\ell_k-\ell_{k-1}}: \, |\mu_k(t)^T x| \leq a\big\},\\
\mathcal{B}&=\{t\in [0,1]^{n/m}:\, \frac{n}{2dm}\leq \ell_{k}-\ell_{k-1}\leq \frac{2n}{dm},\, k=1,...,d\}.
\end{align*}

Note that $X_{[(\ell_{k-1}+1):\ell_k]}^{(i)}|(T^{(i)},F_k)$ is independent of $F_{-k}$ and
\begin{align*}
X_{[(\ell_{k-1}+1):\ell_k]}^{(i)}|(T^{(i)}=t,F_k=\beta_k)\sim \mathbb{P}_{\beta_k|T^{(i)}=t}^{(i)}=N_{\ell_k-\ell_{k-1}}(\beta_k\mu_k(t),I).
\end{align*}
Furthermore, note that the inequalities $\delta^2\leq \frac{0.4^2  md}{2^7n\log(dm)}$ (in view of $\bar{C}\geq 0.4^{-2} 2^{8}L^2\|\psi\|_{\infty}^2\tilde{C}$)  and $n/m\geq 2^6d\log(n/m)$ (in view of $m=O(n^{\frac{2s}{1+2s}}/\log^2 n)$) hold. 

Then by the definition of $B_k(t)$ we have for all $t\in[0,1]^{n/m}$ and $k=1,...,d$ that 
\begin{align*}
\sup_{x\in B_k(t)} \frac{\phi_{\mu_k(t)}(x)}{\phi_{-\mu_k(t)}(x)}&=
\sup_{{x}\in B_k(t)}\exp\Big\{\frac{\big| \|{x}-\mu_k(t)\|_2^2-\|{x}+\mu_k(t)\|_2^2\big|}{2}\Big\}\\
&= \sup_{{x}\in B_k(t)} \exp\{ 2|{x}^T\mu_k(t)| \}=\exp\{ 2a\},
\end{align*}
where $\phi_{{\mu}}$ denotes the density function of a normal distribution with mean vector ${\mu}$ and identity covariance matrix. Then by Theorem \ref{thm: mutualUB:condindependent} in the supplement (with $\mathcal{F}_0=\big\{{\beta}=(\beta_k )_{k=1..d}:\,\beta_k\in \{-1,1\}, k=1,...,d\big\}$) we have that
\begin{align}
I(F;Y^{(i)})&=\int_{[0,1]^{n/m}}I(F;Y^{(i)}|T^{(i)}=t)dt\label{eq: Ub:mutual:help}\\
&\leq \sum_{k=1}^{d}  (\log 2) \int_{[0,1]^{n/m} } \sqrt{ \mathbb{P}_{\beta_k|T^{(i)}=t}^{(i)}\big( X^{(i)}_{[(\ell_{k-1}+1):\ell_k] }\notin B_k(t)\big)}d t\nonumber\\
&\qquad +  \sum_{k=1}^{d}\int_{[0,1]^{n/m}}\mathbb{P}_{\beta_k|T^{(i)}=t}^{(i)}\big(X^{(i)}_{[(\ell_{k-1}+1):\ell_k] }\notin B_k(t)\big)dt\nonumber\\
&\qquad+ 2{C^2(C-1)^2}  I\big({X}^{(i)};{Y}^{(i)}|T^{(i)}\big),\nonumber
\end{align}
with $C=\exp\{ 2^{7/2}\delta \sqrt{n\log(dm)} / \sqrt{dm}\}$.

Note that $ I\big({X}^{(i)};{Y}^{(i)}|T^{(i)}\big) \leq H({Y}^{(i)}|T^{(i)})\leq H({Y}^{(i)})$. In view of Lemma \ref{lem: design} we have that $\mathbb{P}_T(T^{(i)}\in \mathcal{B})\geq1- 2 de^{-n/(8md)}\geq 1-2(md)^{-4}$ following from the inequality $n/m\geq 2^6d \log(n/m)$. Besides for arbitrary $t\in \mathcal{B}$ we have in view of 
\begin{align*}
\|\mu_k(t)\|_2^2\leq  \sum_{j=\ell_{k-1}+1}^{\ell_k} \delta_n \psi_{j_n,k}(t_j)^2\leq \|\psi\|_{\infty}^2\delta_n2^{j_n}(\ell_k-\ell_{k-1})\leq 2 n\delta^2/(md)
\end{align*}
that
\begin{align*}
 \mathbb{P}_{{f}_k}^{(i)}({X}_{[(\ell_{k-1}+1):\ell_k]}^{(i)}\notin B_k(t)| T^{(i)}=t)&= \mathbb{P}_f^{(i)}( |\mu_k(t)^T {X}_{{[(\ell_{k-1}+1):\ell_k]}}^{(i)}|  > a|T^{(i)}=t)\\
&\leq 2 \exp\Big\{-\frac{ (a-\|\mu_k(t)\|_2^2)^2}{2\|\mu_k(t)\|_2^2}\Big\}\\
&\leq 2 \exp\Big\{-\frac{ a^2}{4\|\mu_k(t)\|_2^2}\Big\}
\leq 2 (md)^{-4}.
\end{align*}
Therefore 
\begin{align*}
\int_{[0,1]^{n/m}}&\sqrt{\mathbb{P}_{\beta_k|T^{(i)}=t}^{(i)}\big(X^{(i)}_{[(\ell_{k-1}+1):\ell_k]}\notin B_k(t)\big)}dt\\
&\leq \int_{\mathcal{B}}\sqrt{\mathbb{P}_{\beta_k|T^{(i)}=t}^{(i)}\big(X^{(i)}_{[(\ell_{k-1}+1):\ell_k]}\notin B_k(t)\big)}dt+\mathbb{P}_T(T^{(i)}\notin \mathcal{B} )\\
&\leq \sqrt{2} (md)^{-2}+ 2(md)^{-4}\leq 2 (md)^{-2},
\end{align*}
and similarly $\int_{[0,1]^{n/m}}\mathbb{P}_{f|T^{(i)}=t}^{(i)}\big( X^{(i)}_{[(\ell_{k-1}+1):\ell_k]}\notin B_k(t)\big)dt\leq 4 (md)^{-4}$. Then by plugging in the above inequalities into \eqref{eq: Ub:mutual:help} and using the inequalities $e^{x}\leq 1+2x$ for $x\leq 0.4$ and ${C^2}\leq 2$  we get that
\begin{align*}
I({F};{Y}^{(i)})\leq  \frac{4\log 2}{m^{2}d} + \frac{2^{12} \delta^2n\log(dm)}{md}  H({Y}^{(i)}).%\label{eq: Ub:mutual:help2}
\end{align*}

Furthermore, from the data-processing inequality and the convexity of the KL divergence 
\begin{align}
I({F};{Y}^{(i)})&\leq I\big({F};(T^{(i)},{X}^{(i)})\big)\leq I\big(F;{X}^{(i)}|T^{(i)}\big)+I\big(F;T^{(i)}  \big)\label{eq: UB:mutual}\\
&=\int_{t\in [0,1]^{n/m}} I\big(F;{X}^{(i)}|T^{(i)}=t\big) dt\nonumber\\
&\leq \int_{t\in [0,1]^{n/m}}  \frac{1}{|\mathcal{F}_0|^2}\sum_{\beta,\beta'\in\mathcal{F}_0}K( \mathbb{P}_{\beta|T^{(i)}=t}^{(i)}\|  \mathbb{P}_{\beta'|T^{(i)}=t}^{(i)})dt\nonumber\\
&\leq\frac{1}{2|\mathcal{F}_0|^2}\sum_{\beta,\beta'\in\mathcal{F}_0} \sum_{\ell=1}^{n/m}\sum_{k\in K_{j_n}}(\beta_k'-\beta_{k})^2L^2\delta_n \int_{0}^1 \psi_{j_n,k}^2(t_{\ell})dt_{\ell}\nonumber\\
&\leq  2\delta^2n/m.\nonumber
\end{align}

Then by combining the previous upper bounds and using the data processing inequality $I({F};{Y})\leq \sum_{i=1}^m I({F};{Y}^{(i)})$ we get that
\begin{align*}
I(F;Y)
&\leq \frac{4\delta^2n}{m}\sum_{i=1}^m\min\Big\{ 2^{10}\log(md)d^{-1} H(Y^{(i)}),1\Big\}+4\log 2\\
&\leq \frac{4L^2\delta_n2^{j_n}\|\psi\|_\infty^2 n}{m}\sum_{i=1}^m\min\Big\{ 2^{10}\log(md)d^{-1} H(Y^{(i)}),1\Big\}+4\log 2.
\end{align*}
Finally we arrive to our statement by using Lemma \ref{lem: Shannon} and $2^{j_n}=\tilde{C}d$.

\begin{remark}
We  note that in \cite{yang:barron:1999} it is sufficient to provide the upper bound \eqref{eq: UB:mutual} for the mutual information as there is no limitation in the amount of transmitted bits. In our setting one has to take into account the code length as well, hence sharper upper bounds are required, which is actually the core and most challenging part of the proof of Lemma \ref{lem: mutual_regression}.
\end{remark}

\begin{lemma}\label{lem: design}
Let us take $X_1,X_2,...,X_n\stackrel{iid}{\sim}MN(1/r,1/r,....,1/r)$ and denote by $\chi_k=\{\ell\in\{1,...,n\}: X_\ell=k\}$ the index set of the observations belonging to the $k$th bin, $k=1,...,r$. Then
\begin{align*}
P( 2^{-1}n/r \leq |\chi_k| \leq 2n/r, k=1,...,r )\geq 1- 2r e^{-n/(8r)}.
\end{align*}
\end{lemma}

\begin{proof}
We start with the proof of the upper bound. Note that by Chernoff's bound
\begin{align*}
P( \sup_{k=1,...,r} |\chi_k| \geq 2n/r )&\leq \sum_{k=1}^r P(|\chi_k| \geq 2n/r )
\leq r  e^{-n/(3r) },
\end{align*}
and similarly for the lower bound
\begin{align*}
P( \inf_{k=1,...,r} |\chi_k| \leq 2^{-1}n/r )&\leq r e^{-n/(8r)}.
\end{align*}
\end{proof}

Given a finite set $\FF_0\subset \mathcal{F}$, let use introduce the notations
\begin{align*}
N_{t}^{\max}= \max_{f\in\mathcal{F}_0}\big| \{ \tilde{f}\in\mathcal{F}_0:\, d(f,\tilde{f})\leq t\} \big|,\\
N_{t}^{\min}= \min_{f\in\mathcal{F}_0}\big|\{ \tilde{f}\in\mathcal{F}_0:\, d(f,\tilde{f})\leq t\}\big|.
\end{align*}
 The following theorem is a slight extension of Corollary 1 of \cite{duchi:wainwright:2013}, see also Theorem \ref{lem: Fano:Wainwright:suppl} with the corresponding proof in the supplement. 

\begin{theorem}\label{lem: Fano:Wainwright}
If $\mathcal{F}$ contains a finite set $\mathcal{F}_0$ { 
 and} $|\mathcal{F}_0|-N_{t}^{\min}>N_t^{\max}$, {then} for all $p, t > 0$,
\begin{align*}
\inf_{\hat f \in \EEE(Y)} \sup_{f \in \FF}
\mathbb{E}_{f} d^p(\hat f, f) \ge t^p 
\Big(1-\frac{I(F;Y)+\log 2}{\log (|\mathcal{F}_0|/N_{t}^{\max})}\Big),
\end{align*} 
where $\EEE(Y)$ denotes the set of measurable functions of $Y$, $I(F;Y)$ is the mutual information between the uniform random variable $F$ (on $\mathcal{F}_0$) and $Y$, 
in the Markov chain $F\to X\to Y$, and $\mathbb{E}_f$ is the expectation with respect to the distribution of $Y$ given $F=f$.
 \end{theorem}

\subsection{Entropy of a finite binary string}\label{sec: Shannon}
In the proof of Theorem \ref{theorem: minimaxL2LB} we need to bound the 
entropy of transmitted finite binary string $Y^{(i)}$. Since we 
{do not want to restrict ourself only to prefix codes}, 
we can not use a standard version of Shannon's source coding 
theorem for this purpose. Instead we use the following result.

\begin{lemma}\label{lem: Shannon}
Let $Y$ be a random finite binary string. 
Its expected length  satisfies the inequality
\begin{align*}
H(Y)\leq 2 \EE l(Y)+1.
\end{align*}
\end{lemma}

\begin{proof}
Let $N = l(Y)$ and consider a random binary string  $U$ 
with distribution $U\big|N =n\sim Unif( \{0,1\}^{n})$. 
 Then for $S$ the collection of finite binary strings, 
\begin{align*}
K(Y,U) & = \sum_{s \in S} \PP(Y = s) \log\frac{\PP(Y = s)}{\PP(U = s)} \\
& = \sum_{s \in S} \PP(Y = s) \log\frac{1}{\PP(U = s)} -H(Y)\\
& = \sum_n \sum_{s \in \{0,1\}^n} \PP(Y = s) \log\frac{1}{\PP(U = s)} - H(Y).
\end{align*}
Now for every $n$ {and $s\in\{0,1\}^n$, we have $\PP(U = s) =\PP(U = s\given N=n)\PP(N=n) = 
2^{-n}\PP(N=n)$}. It follows that 
\[
\sum_{s \in \{0,1\}^n} \PP(Y = s) \log\frac{1}{\PP(U = s)} {=}  \PP(N=n)
\log\frac{2^n}{\PP(N=n)}. 
\]
Hence, 
\[
K(Y, U) \le (\log 2)  \EE N + H(N) - H(Y)
\]
The non-negativity of the KL-divergence thus implies that 
$H(Y)\le  \EE N +H(N)$. 

To complete the proof 
 we show that $H(N)\leq \EE N+1$. To do so 
 consider the index set $I=\{i:\, \PP(N=i)\geq e^{-i}\}$ and note that the function $x\mapsto x\log(1/x)$ is monotone increasing for $x\leq e^{-1}$ . Then
\begin{align*}
H(N)&= \sum_{i\in I} \PP(N=i)\log \frac{1}{\PP(N=i)}+ \sum_{i\in I^c} \PP(N=i)\log \frac{1}{\PP(N=i)}\\
&\leq \sum_{i\in I}  \PP(N=i) i+ \sum_{i\in I^c} e^{-i}i \le \EE N + 1.
\end{align*}
This completes the proof.
\end{proof}

\bibliographystyle{acm}
\bibliography{references}

\begin{thebibliography}{10}

\bibitem{battey:2018}
{\sc Battey, H., Fan, J., Liu, H., Lu, J., and Zhu, Z.}
\newblock Distributed testing and estimation under sparse high dimensional
  models.
\newblock {\em Ann. Statist. 46}, 3 (06 2018), 1352--1382.

\bibitem{chagny:2013}
{\sc Chagny, G.}
\newblock Penalization versus {G}oldenshluger-{L}epski strategies in warped
  bases regression.
\newblock {\em ESAIM: Probability and Statistics 17\/} (2013), 328--358.

\bibitem{cohen:1993}
{\sc Cohen, A., Daubechies, I., and Vial, P.}
\newblock Wavelets on the interval and fast wavelet transforms.
\newblock {\em Applied and Computational Harmonic Analysis 1}, 1 (1993), 54 --
  81.

\bibitem{cover:2012}
{\sc Cover, T.~M., and Thomas, J.~A.}
\newblock {\em Elements of information theory}.
\newblock John Wiley \& Sons, 2012.

\bibitem{daubechies:1992}
{\sc Daubechies, I.}
\newblock {\em Ten Lectures on Wavelets}.
\newblock Society for Industrial and Applied Mathematics, 1992.

\bibitem{deisenroth:2015}
{\sc {Deisenroth}, M.~P., and {Ng}, J.~W.}
\newblock {Distributed Gaussian Processes}.
\newblock {\em ArXiv e-prints\/} (Feb. 2015).

\bibitem{duchi:wainwright:2013}
{\sc {Duchi}, J.~C., and {Wainwright}, M.~J.}
\newblock {Distance-based and continuum Fano inequalities with applications to
  statistical estimation}.
\newblock {\em ArXiv e-prints\/} (Nov. 2013).

\bibitem{fano1961transmission}
{\sc Fano, R.~M., and Hawkins, D.}
\newblock Transmission of information: A statistical theory of communications.
\newblock {\em American Journal of Physics 29}, 11 (1961), 793--794.

\bibitem{gershgorin:1931}
{\sc Gershgorin, S.~A.}
\newblock Uber die abgrenzung der eigenwerte einer matrix.
\newblock {\em Bulletin de l'Acad\'emie des Sciences de l'URSS. Classe des
  sciences math\'ematiques et na}, 6 (1931), 749--754.

\bibitem{gine:nickl:2016}
{\sc Gin{\'e}, E., and Nickl, R.}
\newblock {\em Mathematical foundations of infinite-dimensional statistical
  models}.
\newblock Cambridge series in statistical and probabilistic mathematics. 2016.

\bibitem{Guhaniyogi:2017}
{\sc {Guhaniyogi}, R., {Li}, C., {Savitsky}, T.~D., and {Srivastava}, S.}
\newblock {A Divide-and-Conquer Bayesian Approach to Large-Scale Kriging}.
\newblock {\em ArXiv e-prints\/} (Dec. 2017).

\bibitem{hardle:2012}
{\sc H{\"a}rdle, W., Kerkyacharian, G., Picard, D., and Tsybakov, A.}
\newblock {\em Wavelets, Approximation, and Statistical Applications}.
\newblock Lecture Notes in Statistics. Springer New York, 2012.

\bibitem{kleiner:2014}
{\sc Kleiner, A., Talwalkar, A., Sarkar, P., and Jordan, M.~I.}
\newblock A scalable bootstrap for massive data.
\newblock {\em Journal of the Royal Statistical Society: Series B (Statistical
  Methodology) 76}, 4 (2014), 795--816.

\bibitem{lacour:08}
{\sc Lacour, C.}
\newblock {Adaptive estimation of the transition density of a particular hidden
  Markov chain}.
\newblock {\em {Journal of Multivariate Analysis} 99}, 5 (May 2008), 787--814.

\bibitem{jason:2017}
{\sc Lee, J.~D., Liu, Q., Sun, Y., and Taylor, J.~E.}
\newblock Communication-efficient sparse regression.
\newblock {\em Journal of Machine Learning Research 18}, 5 (2017), 1--30.

\bibitem{massart2007concentration}
{\sc Massart, P.}
\newblock Concentration inequalities and model selection.

\bibitem{rosen2016}
{\sc Rosenblatt, J.~D., and Nadler, B.}
\newblock {On the optimality of averaging in distributed statistical learning}.
\newblock {\em Information and Inference: A Journal of the IMA 5}, 4 (06 2016),
  379--404.

\bibitem{scott:etal:2013}
{\sc Scott, S.~L., Blocker, A.~W., Bonassi, F.~V., Chipman, H., George, E., and
  McCulloch, R.}
\newblock Bayes and big data: The consensus monte carlo algorithm.
\newblock In {\em EFaBBayes 250 conference\/} (2013), vol.~16.

\bibitem{shang:cheng:2015}
{\sc {Shang}, Z., and {Cheng}, G.}
\newblock {A Bayesian Splitotic Theory For Nonparametric Models}.
\newblock {\em ArXiv e-prints\/} (Aug. 2015).

\bibitem{srivastava:2015a}
{\sc Srivastava, S., Cevher, V., Tran-Dinh, Q., and Dunson, D.~B.}
\newblock Wasp: Scalable bayes via barycenters of subset posteriors.
\newblock In {\em AISTATS\/} (2015).

\bibitem{szabo:zanten:2017}
{\sc {Szabo}, B., and {van Zanten}, H.}
\newblock {An asymptotic analysis of distributed nonparametric methods}.
\newblock {\em ArXiv e-prints\/} (Nov. 2017).

\bibitem{tsybakov:2009}
{\sc Tsybakov, A.~B.}
\newblock {\em Introduction to nonparametric estimation}.
\newblock Springer, New York, 2009.

\bibitem{wang:2010}
{\sc Wang, J., Chen, J., and Wu, X.}
\newblock On the sum rate of gaussian multiterminal source coding: New proofs
  and results.
\newblock {\em IEEE Transactions on Information Theory 56}, 8 (Aug 2010),
  3946--3960.

\bibitem{yang:barron:1999}
{\sc Yang, Y., and Barron, A.}
\newblock Information-theoretic determination of minimax rates of convergence.
\newblock {\em Ann. Statist. 27}, 5 (10 1999), 1564--1599.

\bibitem{zhang:2013}
{\sc Zhang, Y., Duchi, J., Jordan, M.~I., and Wainwright, M.~J.}
\newblock Information-theoretic lower bounds for distributed statistical
  estimation with communication constraints.
\newblock In {\em Advances in Neural Information Processing Systems\/} (2013),
  pp.~2328--2336.

\bibitem{zhang2013}
{\sc Zhang, Y., Wainwright, M.~J., and Duchi, J.~C.}
\newblock Communication-efficient algorithms for statistical optimization.
\newblock In {\em Advances in Neural Information Processing Systems 25},
  F.~Pereira, C.~J.~C. Burges, L.~Bottou, and K.~Q. Weinberger, Eds. Curran
  Associates, Inc., 2012, pp.~1502--1510.

\bibitem{zhu:2018}
{\sc {Zhu}, Y., and {Lafferty}, J.}
\newblock {Distributed Nonparametric Regression under Communication
  Constraints}.
\newblock {\em ArXiv e-prints\/} (Mar. 2018).

\end{thebibliography}
\appendix

\section{Information theoretic results}\label{sec: info}

\subsection{Basic definitions and results}\label{sec: inf:theory:facts}

A classical reference for basic concepts and results from information theory is the second chapter of \cite{cover:2012}. 
Statements  are only proved for discrete variables in 
Chapter 2 of \cite{cover:2012}, but several are valid more generally and are used more generally
in this paper. 
All logarithms are base $e$ here.

In this section, where we recall notations and basic results, 
$(X, Y, Z)$ is a random triplet in a space $\XX\times\YY\times \mathcal{Z}$ that 
is nice enough, so that regular versions of all conditional distributions exist (for instance
a Polish space). We denote the joint distribution by $P_{(X, Y, Z)}$, 
the (regular version of the) conditional distribution of $X$ given $Y = y$ 
by $P_{X\given Y=y}$, etcetera.
If $X$ and $Y$ are discrete we denote by $p_{(X,Y)}$ the  joint  probability mass function (pmf) 
of $(X,Y)$, by $p_X$ and $p_Y$ the marginal pmf of $X$ and $Y$, and by $p_{X\given Y=y}(x)
=p_{(X,Y)}(x,y)/p_Y(y)$ the conditional pmf of $X$ given $Y=y$. 
For probability measures $P$ and $Q$ on the same 
space we define the Kullback-Leibler 
divergence as usual as $\K(P,Q) = \int \log ({dP}/{dQ})\,dP$
if $P \ll Q$, and as $+\infty$ if not. 

The {\em mutual information} between $X$ and $Y$ and the {\em conditional mutual information} 
between $X$ and $Y$, given $Z$, are  defined as 
\begin{align*}
I(X; Y) & = \K(P_{(X,Y)}, P_X \times P_Y),\\
I(X; Y \given Z= z) & = \K(P_{(X,Y) \given Z= z}, P_{X\given Z= z} \times P_{Y\given Z= z}),\\
I(X; Y \given Z) & = \int I(X; Y \given Z= z)\,dP_Z(z).
\end{align*}
The (conditional) mutual information is nonnegative and symmetric in $X$ and $Y$. 

If $X$ and $Y$ are discrete we define the {\em entropy} of $X$ and the {\em 
conditional entropy of $X$ given $Y$} by 
\begin{align*}
H(X) & = -\sum_x p_X(x)\log p_X(x),\\
H(X \given Y) & = - \sum_{x,y}  p_{(X,Y)}(x,y)\log p_{X \given Y=y}(x).
\end{align*}
Entropy and conditional entropy are nonnegative. For discrete $X$ and $Y$, 
it holds that $I(X; Y) = H(X) - H(X\given Y)$.  Hence, since  mutual information 
is nonnegative, $H(X \given Y) \le H(X)$ (conditioning reduces entropy). 
If $X$ is a discrete variable on a finite set $\XX$ then $H(X) \le \log |\XX|$, 
with equality if $X$ is uniformly distributed on $\XX$. 
We denote by $H(p) = -p\log p -(1-p)\log(1-p)$ the entropy 
of a Bernoulli variable with parameter $p \in (0,1)$. The function
$p \mapsto H(p)$ is a concave function that is symmetric around $p=1/2$. 
Its  maximum value, attained  at $p=1/2$, equals $\log 2$.

We now recall a number of basic  identities for mutual information.  
First of all,  we have the following rule for a general, not necessarily discrete random triplet.

\begin{prop}[Chain rule for mutual information]\label{prop: chain}
We have 
\[
I(X; (Y, Z)) = I(X; Y \given Z) + I(X; Z).
\]
\end{prop}

%
%\begin{prf}
%By disintegration we have 
%\begin{align*}
%dP_{(X,Y,Z)}(x,y,z) & = dP_{(X,Y) \given Z = z}(x,y)dP_Z(z),\\
%dP_{(X,Z)}(x,z) & = dP_{X \given Z = z}(x)dP_Z(z),\\
%dP_{(Y,Z)}(y,z) & = dP_{Y \given Z = z}(y)dP_Z(z).
%\end{align*}
%These identities imply that 
%\[
%\frac{dP_{(X,Y,Z)}}{d(P_X \times P_{(Y,Z)})}(x,y,z) = 
%\frac{dP_{(X,Y) \given Z = z}}{d(P_{X \given Z = z}\times P_{Y \given Z = z})}(x,y)
%\frac{dP_{(X,Z)}}{d(P_X \times P_Z)}(x,z).
%\]
%It follows that 
%\begin{align*}
%I(X; (Y,Z)) &= \int \log\frac{dP_{(X,Y,Z)}}{d(P_X \times P_{(Y,Z)})}\,dP_{(X,Y,Z)}\\
%& = \int \log \frac{dP_{(X,Y) \given Z = z}}{d(P_{X \given Z = z}\times P_{Y \given Z = z})}(x,y)\,dP_{(X,Y,Z)}(x,y,z)\\ 
%& \quad + \int \frac{dP_{(X,Z)}}{d(P_X \times P_Z)}(x,z) \,dP_{(X,Y,Z)}(x,y,z)\\
%& = I(X;Y \given Z) + I(X; Z).
%\end{align*}
%This completes the proof.
%\end{prf}
%
%

We call the triplet $(X, Y, Z)$ a {\em Markov chain}, and write $X \to Y \to Z$,
if the joint distribution disintegrates as 
\[
dP_{(X,Y,Z)}(x,y,z) = dP_X(x)dP_{Y \given X=x}(y)dP_{Z\given Y=y}(z). 
\]
In this situation we have the following result, which relates the information 
in the different links of the chain. Again, discreteness of the variables
is not necessary for this result.

\begin{prop}[Data-processing inequality]\label{prop: dataproc}
If $X \to Y \to Z$ is a Markov chain, then 
\[
I(X; Y) = I(X;Y \given Z) + I(X; Z).
\]
In particular 
\[
I(X; Z) \le I(X; Y). 
\]
\end{prop}

%
%\begin{prf}
%Applying the chain rule once as in the statement of Proposition \ref{prop: chain}
%and once with $Y$ and $Z$ reversed we obtain
%\[
%I(X;Y \given Z) + I(X; Z) = I(X;Z \given Y) + I(X; Y).
%\]
%By the Markove property, $X$ and $Z$ are conditionally independent given $Y$. 
%By definition of conditional information, this implies that $I(X;Z \given Y) = 0$, 
%and the first statement of the proposition follows. 
%The second follows from the first one, since information is nonnegative.
%\end{prf}

In case of independence, mutual information is sub-additive in the following sense.

\begin{prop}[Role of independence]\label{prop: ind}
If $Y$ and $Z$ are conditionally independent given $X$, then
\begin{align*}
I(X; (Y,Z))  = I(X; Y)  + I(X; Z) - I(Y;Z) \le I(X; Y)  + I(X; Z).
\end{align*}
\end{prop}

%
%\begin{prf}
%By the chain rule, 
%\[
%I(X; (Y,Z))  = I(X;Y \given Z) + I(X; Z).
%\]
%We have 
%\begin{align}
%I(X;Y \given Z) & = \int \log\frac{dP_{(X,Y)\given Z=z}}
%{d(P_{X\given Z=z}\times P_{Y\given Z=z})}(x,y)\,dP_{(X,Y,Z)}(x,y,z).\label{eq: hulp:info1}
%\end{align}
%Note that if $Y$ and $Z$ are independent given $X$, then 
%\begin{align*}
%\frac{dP_{(X,Y)\given Z=z}}
%{d(P_{X\given Z=z}\times P_{Y\given Z=z})}(x,y) &= 
%\frac{dP_{(X,Y,Z)}(x,y,z).}{dP_{X\given Z=z}(x)dP_{(Y,Z)}(y,z)}\\
%& = \frac{dP_{(Y,Z)|X=x}(z,y) dP_{X}(x) dP_Z(z)}{dP_{Z|X=x}(y)dP_{X}(x)dP_{(Y,Z)}(y,z)}
%\end{align*}
%Then the right hand side of \eqref{eq: hulp:info1} is further equal to
%\begin{align*}
%&= \int p(x,y, z)\log\frac{p(y,z \given x)p(z)}{p(z\given x)p(y, z)}\,dxdydz\\ 
%&= \int p(x,y, z)\log\frac{p(y\given x)p(z\given x)p(z)}{p(y, z)p(z,x)p(x)}\,dxdydz\\ 
%&= \int p(x,y, z)\log\frac{p(x, y)p(y)p(z)}{p(x)p(y)p(y, z)}\,dxdydz\\ 
%& = I(X; Y) - \int p(x,y, z)\log\frac{p(y, z)}{p(y)p(z)}\,dxdydz\\
%& = I(X; Y) - I(Y;Z).
%\end{align*}
%
%\end{prf}

Finally we recall  Fano's inequality \cite{fano1961transmission} which we  use in the following form
in this paper.

\begin{prop}[Fano's lemma]
Let $X \to Y \to \hat X$ be a Markov chain, where $X$ and  $\hat X$
are random elements in a finite set $\XX$ and 
$X$ has a uniform distribution on  $\XX$. Then 
\[
\PP(X \not = \hat X) \ge 1 - \frac{\log 2 + I(X;Y)}{\log|\XX|}.
\]
\end{prop}

%
%\begin{prf}
%Consider the Bernoulli variable $E = 1_{X \not = \hat X}$ and let $p = \PP(E)$.
%By the proof of Fano's lemma in [REF] we have 
%\[
%H(p) + p\log|\XX| \ge H(X\given \hat X).
%\]
%Since $X$ and $\hat X$ are discrete, the right-hand side equals $H(X) - I(X; \hat X)$. 
%Because  $X$ has a uniform distribution, we have $H(X) =\log|\XX|$. 
%Hence, since $H(p) \le \log 2$, we have 
%\[
%\log 2 + p\log|\XX| \ge \log|\XX| - I(X; \hat X), 
%\]
%i.e.\
%\[
%p \ge 1 - \frac{\log 2 + I(X; \hat X)}{\log|\XX|}.
%\]
%By the data-processing inequality, $I(X; \hat X) \le I(X;Y)$.
%\end{prf}
%

\subsection{Lower bounds for estimators using  processed data}

In this section we consider a situation in which we have a random 
element $X$ in $\XX$ with a distribution $\PP_f$ depending on a parameter $f$ in a
semimetric space $(\FF, d)$. Let us consider any subset $\FF_0\subset \FF$ and take $F$ to be an 
uniform distribution on $\FF_0$. Moreover,  we assume that we have a 
Markov chain $X \to Y$ defined through a Markov transition kernel 
$Q(dy\given x)$ from $\XX$
to some space $\YY$. 
Note that this includes the case that $Y$ is simply a measurable function $Y = \psi(X)$ of 
the full data $X$.
We view $Y$ as a transformed, or processed version of the full data $X$. This forms a Markov chain
\begin{align}
F\to X\to Y. \label{def: Markov:chain}
\end{align}
Furthermore, let us denote by $I(F;Y)$ the mutual information between $F$ and $Y$ 
in the above Markov chain \eqref{def: Markov:chain}.

We are interested in lower bounds for estimators $\hat f$ for $f$ that 
are only based on the processes data. The collection of all such estimators, 
i.e.\ measurable functions of $Y$, is denoted by $\EEE(Y)$.

% We extend the Markov chain to $F\to X \to Y$, where $F$ is a random variable on the functional space $\mathcal{F}$. We denote by $P_F,P_X,P_Y$ the distribution of the random variables $F,X$, and $Y$. Furthermore we use the notations $P_{X|F=f}$, $P_{X|Y=y}$, $P_{Y|F=f}$ the conditional distributions of $X$ given $F=f$, of $X$ given $Y=y$, and $Y$ given $F=f$, respectively. The joint distribution of $X$ and $Y$ is denoted by $P_{(X,Y)}$. We note that (slightly abusing our notation) $\mathbb{P}_f$ could mean either $P_{X|F=f}$ or $P_{Y|F=f}$. This, however, does not lead to ambiguity and it is clear from the context what it means exactly.

The usual approach of relating lower bounds for estimation
to lower bounds for testing multiple hypotheses, in combination
with Fano's lemma, gives the following useful result in our setting.
%\botond{For completeness we provide a proof. [Remove this sentence]}

\begin{theorem}\label{thm: fano}
\label{theorem: m1}
If $\FF$ contains a finite set $\FF_0$ of  functions that are $2\delta$-separated 
for the semimetric $d$, then for all $p > 0$
\[
\inf_{\hat f \in \EEE(Y)} \sup_{f \in \FF}
\mathbb{E}_{f,Q} d^p(\hat f, f) \ge \delta^p 
\Big(1 - \frac{\log 2 + I(F;Y)}{\log |\FF_0|}\Big).
\]

\end{theorem}

%\begin{proof}
%We have 
%\begin{align}
%\inf_{\hat f \in \EEE(Y)} \sup_{f \in \FF}
%\mathbb{E}_f d^p(\hat f, f) \ge 
%\delta^p \inf_{\hat f \in \EEE(Y)} \sup_{f \in \FF} \PP_f(d(\hat f, f) \ge \delta).\label{eq: LB:risk}
%\end{align}
%Let $\tilde f$ be the $\FF_0$-valued estimator 
%that selects the element of $\FF_0$ that is closest to $\hat f$. Since the 
%functions in $\FF_0$ are $2\delta$-separated, we have the implication
%\[
%\{d(\hat f, f) \ge \delta\} \supset \{\tilde f \not = f\}. 
%\]
%Hence for $\Pi$ the uniform distribution on $\FF$, we have 
%\[
%\max_{f \in \FF_0} \PP_f(d(\hat f, f) \ge \delta) \ge \int 
%\PP_f(\tilde f \not = f )\,\Pi(df) = P_{(Y,F)}(\tilde f \not = F ),
%\]
%where $P_{(Y,F)}$ is the distribution under which $F$ is uniformly distributed  
%on $\FF_0$ and $Y \given( F=f) \sim \PP_f$. 
%Now apply Fano's inequality to the Markov chain $F\to Y \to \tilde f$ 
%to complete the proof.
%\end{proof}

We also use a slight modification of this basic result, where the condition
that the functions in $\FF_0$ are separated is replaced by a condition on the 
minimum and maximum {number of} elements in $\FF_0$ that are contained in small balls.
Given a finite set $\FF_0\subset \mathcal{F}$, we use the notations
\begin{align*}
N_{t}^{\max}= \max_{f\in\mathcal{F}_0}\Big\{\#\{ \tilde{f}\in\mathcal{F}_0:\, d(f,\tilde{f})\leq t\} \Big\},\\
N_{t}^{\min}= \min_{f\in\mathcal{F}_0}\Big\{\#\{ \tilde{f}\in\mathcal{F}_0:\, d(f,\tilde{f})\leq t\} \Big\}.
\end{align*}
 The following theorem is a slight extension of Corollary 1 of \cite{duchi:wainwright:2013}. In the latter corollary it is implicitly assumed that $Y$ is a discrete random variable (the conditional entropy $H(F\given Y)$ is considered), while we can allow continuous random variables as well. For self-containedness we provide the proof below.

\begin{theorem}\label{lem: Fano:Wainwright:suppl}
If $\mathcal{F}$ contains a finite set $\mathcal{F}_0$ { 
 and} $|\mathcal{F}_0|-N_{t}^{\min}>N_t^{\max}$, {then} for all $p, t > 0$,
\begin{align*}
\inf_{\hat f \in \EEE(Y)} \sup_{f \in \FF}
\mathbb{E}_{f,Q} d^p(\hat f, f) \ge t^p 
\Big(1-\frac{I(F;Y)+\log 2}{\log (|\mathcal{F}_0|/N_{t}^{\max})}\Big).
\end{align*} 
\end{theorem}

\begin{proof}
We have
\begin{align*}
\inf_{\hat f \in \EEE(Y)} \sup_{f \in \FF}
\mathbb{E}_{f,Q} d^p(\hat f, f) &\ge 
t^p \inf_{\hat f \in \EEE(Y)} \sup_{f \in \FF_0} \PP_{f,Q}(d(\hat f, f) \ge t)\\
&\geq t^p \inf_{\hat f \in \EEE(Y)} P_{(Y,F)}(d(\hat{f},F)>t),
\end{align*}
therefore it is sufficient to show that
\begin{align*}
\inf_{\hat f \in \EEE(Y)}P_{(Y,F)}(d(\hat{f},F)>t)\geq 1-\frac{I(F;Y)+\log 2}{\log (|\mathcal{F}_0|/N_t^{\max})}.
\end{align*}

By definition, 
\begin{align*}
I(F; Y) & = \int\log\frac{dP_{(F,Y)}}{d(P_F \times P_Y)}dP_{(F,Y)}.
\end{align*}
By disintegration, $dP_{(F,Y)}(f,y) = dP_{Y}(y)dP_{F \given Y=y}(f)$, 
so that by Fubini, 
\begin{equation}\label{eq: ii}
I(F; Y)  = \int\Big(\int\log\frac{dP_{F \given Y=y}}{dP_F}\,
 dP_{F \given Y=y}\Big)\,dP_{Y}(dy).
\end{equation}
Now $P_F$ is the uniform distribution on $\FF_0$. Hence, 
it has density $p(f) = 1/|\FF_0|$ w.r.t.\ the counting measure 
$df$ on $\FF_0$. Define $p(f \given y) = P_{F\given Y=y}(\{f\})$, 
so that $P_{F\given Y=y}$ has density $p(f \given y)$ w.r.t.\ $df$. 
The KL-divergence $\K(P_{F\given Y = y}, P_F)$ 
in the inner integral can then be written as 
\[
\int p(f \given y)\log\frac{p(f\given y)}{p(f)}\,df.
\]
Similarly $\hat{f}$ has some density $\hat{p}(\hat{f})$ w.r.t.\ the counting measure 
$d\hat{f}$ on $\FF_0$ and we define $\hat{p}(\hat{f}|y)=P_{\hat{F}|Y=y}(\{\hat{f}\})$.

Next note that by the data-processing inequality, see Proposition \ref{prop: dataproc}, 
we have $ I(F;Y)-I(F;\hat{f})\geq0$. Then
\begin{align*}
0\leq I(F;Y)-I(F;\hat{f})&=  \int\int p(f|y)\log p(f|y)df\,dP_{Y}(dy)-\int p(f)\log p(f)df\\
&\qquad- \Big(\int\int p(f,\hat{f})\log p(f|\hat{f})dfd\hat{f}-\int p(f)\log p(f)df\Big)\\
&=  \int\int p(f|y)\log p(f|y)df\,dP_{Y}(dy){+}H(F|\hat{f}).
 \end{align*}
Next note that since $F$ is uniform on $\FF_0$ (see Section \ref{sec: inf:theory:facts}), 
\begin{align*}
\int\int p(f,y)\log p(f|y)dfdy=I(F;Y){-}H(F){=} I(F;Y)-\log (|\mathcal{F}_0|).
\end{align*}
We can summarize the above results as
\begin{align}
H(F|\hat{f}){\geq  \log (|\mathcal{F}_0|) - I(F;Y)}.\label{eq: UB:entorpy:estimate}
\end{align}

Since by assumption the conditions of Proposition 1 of \cite{duchi:wainwright:2013} hold, we have
\begin{align*}
H(F|\hat{f})\leq H(\mathbb{P}_T)+\mathbb{P}_T{\log} \frac{|\mathcal{F}_0|-N_t^{\min}}{N_t^{\max}} +\log N_t^{\max},
\end{align*}
where $\mathbb{P}_T=P_{(F,Y)}(d(\hat{f},F)>t)$. Then noting that $\log \mathbb{P}_T\leq \log 2$ and by combining the preceding display with $\eqref{eq: UB:entorpy:estimate}$ we get that
\begin{align*}
{\log (|\mathcal{F}_0|)-I(F;Y)}\leq \log 2+ \mathbb{P}_T {\log}\frac{|\mathcal{F}_0|-N_t^{\min}}{N_t^{\max}} +\log N_t^{\max}.
\end{align*}
Reformulation of the inequality yields
\begin{align*}
P_{(Y,F)}(d(\hat{f},F)>t)&\geq \frac{\log (|\mathcal{F}_0|/N_t^{\max})}{\log\big((|\mathcal{F}_0|-N_t^{\min})/N_t^{\max}\big)}-\frac{I(F;Y)+\log 2}{\log\big((|\mathcal{F}_0|-N_t^{\min})/N_t^{\max}\big)}\\
&\geq 1-\frac{I(F;Y)+\log 2}{\log (|\mathcal{F}_0|/N_t^{\max})},
\end{align*}
which  completes the proof.
\end{proof}

In the next subsections we give  bounds for the mutual information under various assumptions on the random variables $X$ and $f$. 
%\botond{ Remove: and finally we combine the derived results to obtain
% a lower bound for the minimax risk of distributed estimators.}

\subsection{Bounding $I(F; Y)$: bounded likelihood ratios}

We consider again a Markov chain \eqref{def: Markov:chain}
and assume that there exists a constant $C \ge 1$
and a set $\XX'$ that has full mass  under $\PP_f$ for every $f \in \FF_0$, such that 
\begin{equation}\label{eq: bound}
\sup_{x \in \XX'} \max_{f_1, f_2 \in \FF_0} \frac{d\mathbb{P}_{f_1}}{d\mathbb{P}_{f_2}}(x) \le C. 
\end{equation}
This condition bounds the information in the first link $F \to X$ of the chain. 
As a result, it becomes possible to derive an upper bound on $I(F;Y)$ in terms of the 
constant $C$ and the information $I(X;Y)$ in the other link of the chain.

{The following theorem is a slight extension of Lemma 3 of \cite{zhang:2013} {(without the independence assumption, see later)} where we allow the random variables $X$ and $Y$ to be continuous as well, unlike in Lemma 3 of \cite{zhang:2013}, where it was implicitly assumed that they are discrete (by using the entropy $H(X)$ and $H(X|Y)$ in the proof). However, in our manuscript $X$ is continuous and therefore the above mentioned lemma does not apply directly.}

\begin{theorem}\label{theorem: h1}
Assume that \eqref{eq: bound} holds for $C \ge  1$. Then 
\[
I(F; Y) \le 2C^2(C - 1)^2I(X;Y).
\]
\end{theorem}

\begin{proof}
In view of \eqref{eq: ii}
\begin{align}
I(F; Y) & ={\int}\int p(f \given y)\log\frac{p(f\given y)}{p(f)}\,df{dP_Y(dy)}.\label{eq: ii_new}
\end{align}

Since the KL-divergence is nonnegative, we have
\begin{align*}
 -\int p(f)\log\frac{p(f\given y)}{p(f)}\,df \ge 0. 
\end{align*}
It follows that the inner integral in \eqref{eq: ii_new} is bounded by 
\[
\int (p(f \given y) - p(f))\log\frac{p(f\given y)}{p(f)}\,df.
\]
Since $|\log (a/b)| \le |a-b|/(a \wedge b)$, we have the further bound 
\[
\K(P_{F\given Y = y}, P_F) \le \int \frac{|(p(f \given y) - p(f)|^2}
{p(f \given y) \wedge p(f)}\,df.
\]
 We will see ahead that the denominator in the integrand is always strictly positive.

We also define  $p(f \given x) = P_{F\given X=x}(\{f\})$.
Then by conditioning on $X$ we see that 
\begin{align*}
p(f\given y)  = \int p(f\given x)\,dP_{X\given Y=y}(x), \qquad
p(f)  = \int p(f\given x)\,dP_X(x).
\end{align*}
By subtracting these relations and using also that 
\[
0   = \int p(f)\big(dP_{X\given Y=y}(x)- dP_X(x)\big),
\]
we obtain
\begin{align}\label{eq: henk}
|p(f \given y) - p(f)| & = \Big| \int (p(f\given x)- p(f))
\big(dP_{X\given Y=y}(x)- dP_X(x)\big)\Big|\nonumber\\
& \le 2 \sup_{x \in \XX'}|p(f\given x)- p(f)| \ \|P_{X\given Y= y}- P_X\|_\text{TV}.
\end{align}
Now by Bayes' formula and the assumption \eqref{eq: bound} on the likelihood, 
\begin{equation}\label{eq: piet}
\frac{p(f\given x)}{p(f)} = \frac{1}{\int d\mathbb{P}_{f'}/d\mathbb{P}_{f}(x)p(f')\,df'} 
 \in [1/C, C] 
\end{equation}
for all $x \in \XX'$. But then also 
\[
(1-C) p(f) \le (1/C - 1) p(f) \le p(f\given x) - p(f) \le  (C - 1) p(f),
\]
that is, $|p(f\given x) - p(f)| \le  (C - 1) p(f)$. 
Also note that since 
\[
p(f\given y) =\int p(f\given x)\,dP_{X\given Y=y}(dx), 
\]
\eqref{eq: piet} implies that $p(f\given y)/p(f) \in [1/C, C]$
as well. In particular, $p(f) \le C p(f\given y)$. 
Together, we get 
\[
|p(f\given x) - p(f)| \le C(C-1) (p(f) \wedge p(f\given y)) 
\]
for all $x \in \XX'$. 
Combining with what we had above, we get 
\begin{align*}
|p(f \given y) - p(f)|  \le 2C(C - 1) (p(f) \wedge p(f\given y)) \|P_{X\given Y= y}- P_X\|_\text{TV},
\end{align*}
and hence 
\[
\frac{|(p(f \given y) - p(f)|^2}
{p(f \given y) \wedge p(f)} \le 4C^2(C-1)^2 p(f) \|P_{X\given Y= y}- P_X\|^2_\text{TV}.
\]
Integrating w.r.t.\ $f$ this gives the bound 
\[
\K(P_F, P_{F\given Y = y}) \le 4C^2(C-1)^2 
\|P_{X\given Y= y}- P_X\|^2_\text{TV}.
\]
Use  Pinsker's inequality (e.g. \cite{tsybakov:2009}, p.\ 88) 
and integrate w.r.t\ $P_Y$ to arrive at  the statement of the theorem.
\end{proof}

\subsection{Bounding $I(F; Y)$: general case}

Let us consider again the  Markov chain \eqref{def: Markov:chain}, but we drop the condition that we have a 
uniform bound on the likelihood ratio. {The following theorem is a slight extension of Lemma 4 of \cite{zhang:2013} 
{(without the independence assumption, see later)} where we allow again that the random variables $X$ and $Y$ are continuous as well, unlike in Lemma 4 of \cite{zhang:2013}, where it was implicitly assumed that they are discrete.}

\begin{theorem}\label{theorem: h2}
For all $C \ge 1$ we have
\[
I(F; Y) \le \log 2 + \frac{\log|\FF_0|}{|\FF_0|}\sum_{f \in \FF_0}
\PP_f\Big(\max_{f_1, f_2 \in \FF_0} \frac{d \mathbb{P}_{f_1}}{d \mathbb{P}_{f_2}}(X) > C\Big) 
+ 2{C^2(C-1)^2}I(X; Y). 
\]
\end{theorem}

\begin{proof}
Define the set 
\[
B = \Big\{x: \in \XX: \max_{f_1, f_2 \in \FF_0} \frac{d \mathbb{P}_{f_1}}{d \mathbb{P}_{f_2}}(x) \le C\Big\}
\]
and the indicator variable $E = 1_{X \in B}$. With this notation the statement
of the theorem reads
\[
I(F; Y) \le \log 2 + P_{(X,F)}(E=0)\log|\FF_0| + 2{C^2(C-1)^2}I(X; Y),
\]
where $P_{(X,F)}$ is the probability measure defined by the Markov chain, 
i.e.\ the measure  under which $F$ is uniform and $X \given (F=f) \sim \PP_f$.

By the chain rule and the fact that {the mutual} information is nonnegative, 
$I(F; (Y,E)) = I(F; Y) + I(F; E\given Y) \ge I(F,Y)$.
On the other hand, applying the chain rule with $Y$ and $E$ reversed shows that
$I(F; (Y,E)) = I(F; Y \given E) + I(F; E)$.
Hence, we have the inequality
\[
I(F;Y) \le I(F; Y \given E) + I(F; E).
\]
The second term on the right involves only discrete variables and can be bounded by $H(E)$.
This is the entropy of a Bernoulli variable, which is bounded by $\log 2$. 
The first term equals
\[
pI(F; Y \given E=1) + (1-p)I(F; Y \given E=0),
\]
where $p = P_X(E = 1)$. 
Below we prove that 
\begin{equation}\label{eq: tp3}
I(F; Y \given E=1) \le 2{C^2(C-1)^2}I(X,Y \given E=1).
\end{equation}

%\botond{REMOVE: Observe that 
%\[
%p I(X;Y \given E=1) \le p I(X;Y \given E=1) + (1-p) I(X;Y \given E=0)
%= I(X; Y \given E). 
%\]}
By the chain rule, 
$I(X;Y\given E) = I(Y; (X,E)) - I(Y; E) \le I(Y; (X,E))$.
Since $E$ is a function of $X$, the last quantity  equals $I(Y; X)$. 
Next, observe that it follows from the definitions 
that 
\[
I(F; Y \given E= 0) = H(F \given E=0) - 
{\int \int p(f|y)\log\frac{1}{p(f|y)}df \,dP_{Y\given E=0}(y)} \le H(F \given E=0). 
\]
Since $F \given E=0$ lives in the finite set $\FF_0$, this is further bounded 
by $\log|\FF_0|$.

It remains to establish \eqref{eq: tp3}. This essentially follows
from conditioning on $E=1$ everywhere in the proof of Theorem \ref{theorem: h1}. 
Indeed, conditioning in the first part of the proof shows that 
\[
I(F; Y \given E = 1)  \le \int\Big(\int \frac{|(p(f \given y, E=1) - p(f\given E=1)|^2}
{p(f \given y, E=1) \wedge p(f\given E=1)}\,df\Big)\,dP_{Y\given E=1}(dy)
\]
and
\begin{align*}
|p(f \given y, E=1) - p(f\given E=1)|  &\le 2 \sup_{x \in B}\Big|p(f\given x, 
E=1)- p(f\given E=1)\Big|\\
&\qquad\times \ \|P_{X\given Y= y, E=1}- P_{X\given E=1}\|_\text{TV}.
\end{align*}
Now observe that since the likelihood ratio is uniformly bounded by $C$ 
for $x \in B$, Bayes formula implies that 
\[
\frac{p(f\given x, E=1)}{p(f\given E=1)} \in {[1/C, C]}
\]
for all $x \in B$. We can then follow the  rest of the proof of Theorem \ref{theorem: h1}
and arrive at \eqref{eq: tp3}.
\end{proof}

\subsection{Bounding $I(F; Y)$: extra independence assumption}\label{sec: mutual:independence}

Next we consider one additional assumption on the structure of 
the problem. We assume that the
data $X$ is a $d$-dimensional vector of 
the form $X = (X_1, \ldots, X_d)$, 
and that $F$ is a $d$-dimensional vector as well such that 
for all coordinates $j \in \{1, \ldots, d\}$, it holds 
that {$F_j$ and} $X_j|(F_j=f_j)$ are independent of $F_{-j}$.
More precisely, we assume that  for the marginal conditional 
density of $X_j$ it holds that 
\[
{p(x_j|f) = p(x_j\given f_{j})}
\]
for every $j$. Note that this is an assumption on the statistical 
model for the data $X$ and is not related to the distribution of $F$. 

{The following theorem is an extended version of Lemma 3 of \cite{zhang:2013} (now also with the independence assumption) as it holds also for continuous random variables $X$ and $Y$, unlike the result derived in \cite{zhang:2013} .}

\begin{theorem}
Suppose that $F$ and $X$ are $d$-dimensional and that $X_j\given F$ only 
depends on $F_j$ {(i.e.\ $p(x_j|f)=p(x_j|f_j)$)}. Moreover, suppose that for the marginal densities $p(x_j \given f_j)$
it holds that 
\[
\sup_{x_j} \max_{f \not =  f'} \frac{p(x_j \given f_j)}{p(x_j \given f'_j)} \le C
\]
for a constant $C \ge 1$. Then
\[
I(F; Y) \le 2C^2(C - 1)^2\sum_{j=1}^d I(X_j; Y \given F_{1:j-1}).
\]

\end{theorem}

\begin{proof}
By the chain rule,
\[
I(F; Y) = \sum_{j=1}^d I(F_j; Y \given F_{1:j-1}). 
\]
So consider term $j$ in the sum. By definition of conditional mutual information and Fubini's theorem, 
\begin{align*}
& I(F_j; Y \given F_{1:j-1})  = \\
& \int p(f_{1:j-1})
\Big(\int p(y\given f_{1:j-1})
\Big(\int {p(f_j\given y, f_{1:j-1})\log\frac{p(f_j\given y, f_{1:j-1})}
{p(f_j\given f_{1:j-1})}\,df_j\Big)\,dy}\Big)\,df_{1:j-1}.
\end{align*}

We are first going to analyze the inner integral. So fix $f_{1:j-1}$ for now. 
To simplify the notation somewhat
we are going to write densities that are conditional on $f_{1:j-1}$ by $\pp$ instead of $p$.
Then the inner integral becomes 
\[
\int \pp(f_j\given y)\log\frac{\pp(f_j\given y)}
{\pp(f_j)}\,df_j.
\]
Since KL-divergence is nonnegative, it follows that the inner integral  is bounded by 
\[
\int (\pp(f_j \given y) - \pp(f_j))\log\frac{\pp(f_j\given y)}{\pp(f_j)}\,df_j.
\]
In view of the inequality $|\log (a/b)| \le |a-b|/(a \wedge b)$ we obtain 
the further bound 
\[
 \int \frac{|(\pp(f_j \given y) - \pp(f_j)|^2}
{\pp(f_j \given y) \wedge \pp(f_j)}\,df_j.
\]
(we will see ahead that the denominator in the integrand is always strictly positive).
By conditioning on $x_j$ (and still on $f_{1: j-1}$)  we see that 
\begin{align}\label{eq: klaas}
\pp(f_j\given y)  = \int \pp(f_j\given x_j)\pp(x_j\given y)\,dx_j, \qquad
\pp(f_j)  = \int \pp(f_j\given x_j)\pp(x_j)\,dx_j.
\end{align}
By subtracting these relations and using also that 
$0   = \int \pp(f_j)(\pp(x_j\given y) - \pp(x_j))\,dx_j$,
we obtain
\begin{align}\label{eq: henk}
|(\pp(f_j \given y) - \pp(f_j)| & = \Big| \int (\pp(f_j \given x_j) - \pp(f_j))
(\pp(x_j\given y) - \pp(x_j))\,dx_j\Big|\nonumber\\
& \le 2 \sup_{x_j \in \XX'}|\pp(f_j \given x_j) - \pp(f_j)| \ \|\tilde P_{X_j\given Y= y}- 
\tilde P_{X_j}\|_\text{TV}.
\end{align}
Now by Bayes' formula, 
\[
\frac{\pp(f_j \given x_j)}{\pp(f_j)} = \frac{\pp(x_j\given f_j)}
{\int \pp(x_j\given f'_j)\,df'_{j}}.
\]
But by the conditional independence assumption, 
$\pp(x_j\given f_j) = p(x_j\given f_{1:j}) = p(x_j\given f_j)$.
Hence, by the assumed bound on the marginal likelihood-ratio, 
${\pp(f_j \given x_j)}/{\pp(f_j)} \in [1/C, C]$
for all $x_j \in \XX'$. 
But then also 
\[
(1-C) \pp(f_j) \le (1/C - 1) \pp(f_j) \le \pp(f_j\given x_j) - \pp(f_j) \le  (C - 1) \pp(f_j),
\]
that is, $|\pp(f_j\given x_j) - \pp(f_j)| \le  (C - 1) \pp(f_j)$. 
Also note that the first identity in \eqref{eq: klaas} implies that  
 $\pp(f_j\given y)/\pp(f_j) \in [1/C, C]$
as well. In particular, $\pp(f_j) \le C \pp(f_j\given y)$. 
Together, we get 
\[
|\pp(f_j\given x_j) - \pp(f_j)| \le C(C-1) (\pp(f_j) \wedge \pp(f_j\given y)) 
\]
for all $x_j \in \XX'$. 
Combining with what we had above, we get 
\begin{align*}
|\pp(f_j \given y) - \pp(f_j)|  \le 2C(C - 1) 
(\pp(f_j) \wedge \pp(f_j\given y)) \|\tilde P_{X_j\given Y= y}- \tilde P_{X_j}\|_\text{TV},
\end{align*}
and hence 
\[
\frac{|(\pp(f_j \given y) - \pp(f_j)|^2}
{\pp(f_j \given y) \wedge \pp(f_j)} \le 4C^2(C-1)^2 \pp(f_j) 
\|\tilde P_{X_j\given Y= y}- \tilde P_{X_j}\|^2_\text{TV}.
\]
Integrating w.r.t.\ $f_j$ this gives the bound 
\[
\int \pp(f_j\given y)\log\frac{\pp(f_j\given y)}
{\pp(f_j)}\,df_j \le 4C^2(C-1)^2 
\|\tilde P_{X_j\given Y= y}- \tilde P_{X_j}\|^2_\text{TV}.
\]
By Pinsker's inequality, 
\[
\|\tilde P_{X_j\given Y= y}- \tilde P_{X_j}\|^2_\text{TV} \le 
\frac12 K(\tilde P_{X_j\given Y= y}, \tilde P_{X_j}).
\]
Hence, by multiplying by $\pp(y)$ and integrating we find that 
\[
\int \pp(y) \Big(\int \pp(f_j\given y)\log\frac{\pp(f_j\given y)}
{\pp(f_j)}\,df_j \Big)\,dy \le 2C^2(C-1)^2I(X_j; Y \given F_{{1:j-1}} = f_{1:j-1}).
\]
Multiplying by $p(f_{1:j-1})$ and integrating gives
\[
I(F_j; Y \given F_{1:j-1}) \le 2C^2(C-1)^2I(X_j; Y \given F_{1:j-1}).
\]
\end{proof}

%\begin{remark}
%We have
%\[
%p(x_j \given f_{1: j}) = p(x_j \given f_{j}).
%\]
%On the other hand,
%\[
%p(x_j \given f_{1: j})   = p(x_j \given f_{1: j-1}, f_j) = 
%\frac{p(x_j, f_j \given f_{1:j-1})}{p(f_j\given f_{1: j-1})}.
%\]
%Hence, 
%\[
%p(x_j, f_j \given f_{1:j-1}) = p(f_j\given f_{1: j-1})p(x_j \given f_{j}).
%\]
%If the coordinates $f_j$ of $f$ are independent, then 
%$p(f_j\given f_{1: j-1}) = p(f_j)$. In that case, integrating 
%the last identity w.r.t.\ $f_j$ gives $p(x_j \given f_{1:j-1}) = p(x_j)$, 
%i.e.\ $x_j$ is independent of $f_{1:j-1}$. 
%\end{remark}

We also have the version of the preceding result for the case
that we do not have the likelihood ratio bound everywhere. 
{This result is an extended version of Lemma 4 of \cite{zhang:2013}.}

\begin{theorem}\label{thm: mutualUB:independent}
Suppose that $F$ and $X$ are $d$-dimensional and that 
{$F_j$ and $X_j|F_j$ are independent of  $F_{-j}$}. For $C \ge 1$, define
\[
B_j = \Big\{ x_j: \max_{f \not =  f'} \frac{p(x_j \given f_j)}{p(x_j \given f'_j)} \le C\Big\}
\]
for a constant $C \ge 1$. Then
\begin{align*}
I(F; Y) & \le \sum_{j=1}^d \Big((\log 2) \sqrt{P_{X_j}(X_j \not\in B_j)}+
\log|\FF_0| P_{X_j}(X_j \not \in B_j)\Big)+ 2{C^2(C-1)^2} I(X;Y).
\end{align*}
\end{theorem}

\begin{proof}
Again we start with 
\[
I(F; Y) = \sum_{j=1}^d I(F_j; Y \given F_{1:j-1}). 
\]
Now for fixed $j$ we argue as in {Theorem \ref{theorem: h2}} (but with $\log 2$ replaced by $H(1_{X_j\in B_j})$, see the argument above \eqref{eq: tp3}), but conditional on $F_{1:j-1}$, to get
\begin{align*}
I(F_j; Y \given F_{1:j-1})& \le H(1_{X_j \in B_j}\given F_{1:j-1}) + 
\log|\FF_0| P_{X_j}(X_j \not \in B_j)\\
& \qquad+ 2{C^2(C-1)^2}I(X_j,Y \given F_{1:j-1}).
\end{align*}
Since conditioning decreases entropy, 
$$H(1_{X_j \in B_j}\given F_{1:j-1})
\le H(1_{X_j \in B_j}) \le (\log 2) \sqrt{P_{X_j}(X_j \not\in B_j)}.$$
Combining the preceding computations we get that
\begin{align*}
I(F, Y) & \le \sum_{j=1}^d \Big((\log 2) \sqrt{P_{X_j}(X_j \not\in B_j)}+
\log|\FF_0| P_{X_j}(X_j \not \in B_j)\\
&\qquad\qquad+ 2{C^2(C-1)^2}I(X_j,Y \given F_{1:j-1})\Big).
\end{align*}
Then the statement of the theorem follows from Lemma \ref{lem: mutual:combination} (below) and by applying the chain rule of information, i.e.
$$\sum_{j=1}^d I(X_j;Y|F_{1:j-1})\leq \sum_{j=1}^d I(X_j;Y|X_{1:j-1})=I(X;Y).$$

\end{proof}

\begin{lemma}\label{lem: mutual:combination}
Under the assumption that {$X_j|F_j$ and $F_j$ are independent of $F_{-j}$} we have that
\begin{align*}
I(X_j;Y|F_{1:j-1})\leq I(X_j;Y|X_{1:j-1}).
\end{align*}
\end{lemma}
\begin{proof}
\begin{align}
I(X_j;Y|F_{1:j-1})&=\int\int\int p(x_j,y|f_{1:j-1})\log\frac{p(x_j,y|f_{1:j-1})}{p(x_j|f_{1:j-1})p(y|f_{1:j-1})}d x_jd y\, p(f_{1:j-1})df_{1:j-1}\nonumber\\
&=\int\int\int p(x_j,y,f_{1:j-1})\log p(x_j|y,f_{1:j-1})dx_jdyd f_{1:j-1}\nonumber\\
&\qquad-\int\int p(x_j,f_{1:j-1})\log p(x_j|f_{1:j-1})dx_j df_{1:j-1}.\label{eq: UB_mutual}
\end{align}
Next we note that since $X_j$ is independent of $F_{1:j-1}$ we have $p(x_j)=p(x_j|f_{1:j-1})$, furthermore {since} $X_j$ and $X_{1:j-1}$ {are independent} we get $p(x_j)=p(x_j|x_{1:j-1})$. Besides, we show below that
\begin{align}
&\int\int\int p(x_j,y,f_{1:j-1})\log p(x_j|y,f_{1:j-1})dx_jdyd f_{1:j-1}\nonumber\\
&\qquad\leq 
\int\int\int\int p(x_j,y,f_{1:j-1},x_{1:j-1})\log p(x_j|y,f_{1:j-1},x_{1:j-1})dx_jdyd f_{1:j-1}dx_{1:j-1}.\label{eq: hulp1}
\end{align}
Combining the preceding assertions we get that the right hand side of \eqref{eq: UB_mutual} is further bounded from above by
\begin{align*}
&\int\int\int\int p(x_j,y,f_{1:j-1},x_{1:j-1})\log p(x_j|y,f_{1:j-1},x_{1:j-1})dx_jdyd f_{1:j-1}dx_{1:j-1}\\
&\qquad-\int\int p(x_j,x_{1:j-1})\log p(x_j|x_{1:j-1})dx_j dx_{1:j-1}\\
&= \int\int\int p(x_j,y,x_{1:j-1})\log p(x_j|y,x_{1:j-1})dx_jdydx_{1:j-1}\\
&\qquad-\int\int\int p(x_j,y,x_{1:j-1})\log p(x_j|x_{1:j-1})dx_j dy dx_{1:j-1}\\
&=\quad I(X_j;Y|X_{1:j-1}),
\end{align*}
where {in} the first equation we used the Markov property of the chain \eqref{def: Markov:chain} combined with the independence of $X_{j:d}$ and $F_{1:j-1}$, i.e.
\begin{align*}
p(x_j|y,f_{1:j-1},x_{1:j-1})=\frac{p(x_j,y|f_{1:j-1},x_{1:j-1})}{p(y|f_{1:j-1},x_{1:j-1})}=\frac{p(x_j,y|x_{1:j-1})}{p(y|x_{1:j-1})}=p(x_j|y,x_{1:j-1}).
\end{align*}

It remained to verify the inequality \eqref{eq: hulp1}, which follows from
\begin{align*}
&\int\int\int\int p(x_j,y,f_{1:j-1},x_{1:j-1})\log p(x_j|y,f_{1:j-1},x_{1:j-1})dx_jdyd f_{1:j-1}dx_{1:j-1}\\
&\qquad-\int\int\int p(x_j,y,f_{1:j-1})\log p(x_j|y,f_{1:j-1})dx_jdyd f_{1:j-1}\\
&\quad = \int\int\int\int  p(x_j,y,f_{1:j-1},x_{1:j-1})\log \frac{p(x_j|y,f_{1:j-1},x_{1:j-1})}{ p(x_j|y,f_{1:j-1})}dx_jdyd f_{1:j-1}dx_{1:j-1}\\
&\quad = \int\int\int\int  p(x_j,y,f_{1:j-1},x_{1:j-1})\log {\frac{p(x_j,x_{1:j-1}|y,f_{1:j-1})}{ p(x_j|y,f_{1:j-1}) p(x_{1:j-1}|y,f_{1:j-1})}}dx_jdyd f_{1:j-1}dx_{1:j-1}\\
&\quad = {I(X_j;X_{1:j-1}|Y,F_{1:j-1} )}\geq 0.
\end{align*}
\end{proof}

\subsection{Bounding $I(F; Y)$: decomposable Markov chain}\label{sec: mutual:decompose}
Suppose  now in addition that the data 
can be decomposed as $X = (X^{(1)}, \ldots X^{(m)})$
and that under  $\mathbb{P}_f$, 
the $X^{(i)}$ are independent and $X^{(i)} \sim \mathbb{P}^{(i)}_f$. 
This is intended to describe a setting in which the 
data is distributed over
 $m$ different local machines. 
The machines have independent data and each machine 
has its local statistical model $(\mathbb{P}^{(i)}_f: f \in \FF)$.  Next 
we have for every $i$ a Markov chain $X^{(i)} \to Y^{(i)}$
defined by some Markov transition kernel $Q^{(i)}$. In other words, 
every machine processes or transforms its local data in some way. 
Next the processed data is aggregated into a vector $Y=(Y^{(1)}, \ldots, Y^{(m)})$. 
As before we consider the collection $\EEE(Y)$ of all estimators that 
are measurable functions of this aggregated, processed data $Y$. 
In this distributed setting we have by Proposition \ref{prop: ind} that
\begin{align}
I(F; Y) \le \sum_i I(F; Y^{(i)}). \label{eq: UB:mutual:decompose}
\end{align}
Then the statement of Theorem \ref{thm: mutualUB:independent} can be reformulated as

\begin{theorem}\label{thm: mutualUB:independent:general}
Let us assume that the data $X$ is decomposable as above and suppose that $F$ and $X$ are $d$-dimensional and that {$X_j \given F_j$ and $F_j$ are independent of
$F_{-j}$}. For $C \ge 1$, define
\[
B_j = \Big\{ x_j: \max_{f \not =  f'} \frac{p(x_j \given f_j)}{p(x_j \given f'_j)} \le C\Big\}
\]
for a constant $C \ge 1$. Then
\begin{align*}
I(F; Y) & \le \sum_{i=1}^m\sum_{j=1}^d \Big((\log 2) \sqrt{P_{X_j^{(i)}}(X_j^{(i)} \not\in B_j)}+
\log|\FF_0| P_{X_j^{(i)}} ({X_j^{(i)}}\not \in B_j)\Big)\\
&\qquad\qquad+ 2{C^2(C-1)^2} \sum_{i=1}^{m} I(X^{(i)};Y^{(i)}).
\end{align*}
Here $I(X^{(i)};Y^{(i)})$ is the mutual information between $X^{(i)}$ and $Y^{(i)}$ 
in the Markov chain $F \to X^{(i)} \to Y^{(i)}$, where $F$ has a 
uniform distribution on $\FF_0$ and  $X^{(i)} \given (F=f) \sim \mathbb{P}^{(i)}_f=P_{X^{(i)}|F=f}$. 
\end{theorem}

\begin{proof}
The statement of the theorem follows by combining assertion \eqref{eq: UB:mutual:decompose} with Theorem \ref{thm: mutualUB:independent}.
\end{proof}

\subsection{Bounding $I(F;Y)$: Decomposable and conditionally independent Markov chain}
We consider a slightly modified setting compared to the preceding sections, i.e. 
we have a Markov chain $F \to (T,X) \to Y$ 
and we assume that $F$ has a uniform distribution on a 
finite set $\FF_0\subset \mathbb{R}^{d}$ and $F$ is independent of $T$. Suppose furthermore that the data 
can be decomposed as $(T,X) = \big((T^{(1)},X^{(1)}), \ldots, (T^{(m)},X^{(m)})\big)$ and conditionally on $F$ the pairs $(T^{(i)},X^{(i)})$, $i=1,...,m$, are independent. This is intended to describe the random design regression setting in which the 
data is distributed over $m$ different local machines, where $T^{(i)}\in[0,1]^{n/m}$ is the 
vector of observed design points in machine $i$, and $X^{(i)}\in\mathbb{R}^{n/m}$ is the vector 
of corresponding noisy observations of the regression function in these points.

 Next we have for every $i$ a Markov chain $(T^{(i)},X^{(i)}) \to Y^{(i)}$
defined by some Markov transition kernel $Q^{(i)}$.  
In this modified distributed setting we still have by Proposition \ref{prop: ind} that
\begin{align}
I(F; Y) \le \sum_i I(F; Y^{(i)}). \label{eq: UB:mutual:decompose}
\end{align}
Let $X^{(i)}_{[j_1:j_2]}=(X^{(i)}_{j_1},X^{(i)}_{j_1+1},...,X^{(i)}_{j_2})$ and assume that conditionally on $T^{(i)}=t\in[0,1]^{n/m}$ there exist $1=\ell_0\leq\ell_1\leq\ell_2\leq... \ell_d\leq n/m$ indexes such that the local data $X^{(i)}\in\mathbb{R}^{n/m}$ satisfies that $F_j\in \mathbb{R}$ and $X_{[(\ell_{j-1}+1):\ell_j]}^{(i)}|(F_j=f_j, T^{(i)}=t)$ are independent of $F_{-j}\in\mathbb{R}^{d-1}$.
More precisely, we assume that  for the marginal conditional 
density of $X_{[(\ell_{j-1}+1):\ell_j]}^{(i)}\in\mathbb{R}^{\ell_j-\ell_{j-1}}$ it holds that 
\[
{p(x|f,t) = p(x\given f_{j},t)},\quad x\in\mathbb{R}^{\ell_j-\ell_{j-1}}
\]
for every $j$. Note that by data processing inequality and the independence of $T^{(i)}$ and $F$ we get that
\begin{align}
I(F; Y^{(i)})=I(F; Y^{(i)}|T^{(i)}).\label{eq: dataprocess:condind}
\end{align}
Then in view of Theorem \ref{thm: mutualUB:independent:general} we get the following upper bound on the conditional mutual information.

\begin{theorem}\label{thm: mutualUB:condindependent}
Let us assume that conditionally on the local design $T^{(i)}$ the local data $X^{(i)}$ satisfies that $X_{[(\ell_{j-1}+1):\ell_j]}^{(i)} \given (F_j,T^{(i)})$ and $F_j$ are independent of
$F_{-j}$, and $T^{(i)}$ is independent of $F$. For $C \ge 1$, define
\[
B_j(t) = \Big\{ x\in \mathbb{R}^{\ell_j-\ell_{j-1}}: \max_{f \not =  f'\in \mathcal{F}_0} \frac{p(x \given f_j,t)}{p(x\given f'_j,t)} \le C\Big\}.
\]
Then
\begin{align*}
I(F; Y^{(i)}|T^{(i)}=t) & \le \sum_{j=1}^d \Big((\log 2) \sqrt{\mathbb{P}_{f_j}\big(X_{[\ell_{j-1}+1:\ell_j]}^{(i)} \not\in B_j(t)|T^{(i)}=t\big)}\\
&\qquad+\log|\FF_0| \mathbb{P}_{f_j}\big(X_{[\ell_{j-1}+1:\ell_j]}^{(i)} \not\in B_j(t)|T^{(i)}=t\big)\Big)\\
&\qquad+ 2{C^2(C-1)^2}  I(X^{(i)};Y^{(i)}|T^{(i)}=t).
\end{align*}
Here $I(X^{(i)};Y^{(i)}|T^{(i)}=t)$ is the conditional mutual information between $X^{(i)}$ and $Y^{(i)}$ given $T^{(i)}=t$
in the Markov chain $F \to (T^{(i)},X^{(i)}) \to Y^{(i)}$, where $F$ has a 
uniform distribution on $\FF_0$ and  $X^{(i)} \given (F=f, T^{(i)}=t) \sim \mathbb{P}^{(i)}_{f}(\cdot|T^{(i)}=t)=P_{X^{(i)}|(F=f,T^{(i)}=t)}$. \end{theorem}

\section{Definitions and notations for wavelets}\label{sec: wavelets}
In this section we give a brief introduction to wavelets. A more detailed and elaborate description of wavelets can be found for instance in \cite{hardle:2012, gine:nickl:2016}.

In our work we consider the Cohen, Daubechies and Vial construction of compactly supported, orthonormal, $N$-regular wavelet basis of $L_2[0,1]$, see for instance \cite{cohen:1993}. First for any $N\in\mathbb{N}$ one can follow Daubechies' construction of the father $\phi(.)$ and mother $\psi(.)$ wavelets with $N$ vanishing moments and bounded support on $[0,2N-1]$ and $[-N+1,N]$, respectively, see for instance \cite{daubechies:1992}. Then we obtain the basis functions
\begin{align*}
\big\{ \phi_{j_0m},\psi_{jk}:\, m\in\{0,...,2^{j_0}-1\},\quad j> j_0,\quad k\in\{0,...,2^{j}-1\} \big\},
\end{align*}
for some sufficiently large resolution level $j_0\geq0$. The basis functions (on $x\in[0,1]$) are given as $\psi_{jk}(x)=2^{j/2}\psi(2^jx-k)$, for $k\in [N-1,2^j-N]$, and $\phi_{j_0 k}(x)=2^{j_0}\phi(2^{j_0}x-m)$, for $m\in [0,2^{j_0}-2N]$, while for other values of $k$ and $m$, the basis functions are specially constructed, to form a basis with the required smoothness property. For convenience we introduce the notation $\psi_{j_0k}:=\phi_{j_0k}$ for $k=0,...,2^{j_0}-1$. This does not mean, however, that $\phi_{j_0k}(x)=2^{-1}\psi_{j_0+1, k}(2^{-1}x)$. Then the function $f\in L_2[0,1]$ can be represented in the form
\begin{align*}
f=\sum_{j=j_0}^{\infty}\sum_{k=0}^{2^{j}-1}f_{jk}\psi_{jk},
\end{align*}
with $f_{jk}=\langle f,\psi_{jk}\rangle$. Note that in view of the orthonormality of the wavelet basis the {$L_2$-norm} of the function $f$ is equal to
\begin{align*}
\|f\|_2^2=\sum_{j=j_0}^{\infty}\sum_{k=0}^{2^{j}-1}f_{jk}^2.
\end{align*}
For notational convenience we will take $j_0$ to be 0 in our paper, this can be done without loss of generality.

Next we introduce the Besov spaces we are considering in our analysis. Let us define the Besov (Sobolev-type) norm for $s\in(0,N)$ as
\begin{align*}
\|f\|_{B_{2,\infty}^s}^2=\sup_{j\geq j_0} 2^{2js}\sum_{k=0}^{2^j-1}f_{jk}^2.
\end{align*}
 Then the Besov space $B_{2,\infty}^s$ and Besov ball $B_{2,\infty}^s(L)$ of radius $L>0$ are defined as
\begin{align*}
B_{2,\infty}^s=\{f\in L_2[0,1]:\, \|f\|_{B_{2,\infty}^s}<\infty \},\quad\text{and}\quad B_{2,\infty}^s(L)=\{f\in L_2[0,1]:\, \|f\|_{B_{2,\infty}^s}<L \},
\end{align*}
respectively. We note that the present Besov space is larger than the standard Sobolev space where instead of the supremum one would take the sum over the resolution levels $j$. Then we introduce the Besov (H\"older-type) norm for $s\in(0,N)$ as
\begin{align*}
\|f\|_{B_{\infty,\infty}^s}= \sup_{j\geq j_0, k}\{ 2^{j(s+1/2)}|f_{jk}|\}.
\end{align*}
Then similarly to before we define the Besov space $B_{\infty,\infty}^s$ and Besov ball $B_{\infty,\infty}^s(L)$ of radius $L>0$ as
\begin{align*}
&B_{\infty,\infty}^s=\{f\in L_2[0,1]:\, \|f\|_{B_{\infty,\infty}^s}<\infty \}\quad\text{and}\\
&B_{\infty,\infty}^s(L)=\{f\in L_2[0,1]:\, \|f\|_{B_{\infty,\infty}^s}<L \},
\end{align*}
respectively. For $s\neq N$ these spaces are equivalent to the classical H\"older spaces, while for integer $s$ they are equivalent to the so called Zygmond spaces, see \cite{cohen:1993}.

\section{Collection of technical lemmas}
In this section we give a collection of technical lemmas from the literature, we have used in the paper, for easier reference. 

\begin{lemma}[Lemma 5 of \cite{lacour:08}.]\label{lem4:chagny}
Let $\zeta_1,...,\zeta_n$ be iid random variables and define $\nu_n(r)=\sum_{i=1}^n r(\zeta_i)-E r(\zeta_i)$, where $r$ belongs to a countable class $\mathcal{R}$ of real valued measurable functions. Then for $\eps>0$, 
\begin{align*}
E\Big[ \sup_{r\in\mathcal{R}}\nu_n(r)^2-2(1+2\eps)H^2\Big]_+\leq \frac{4}{K_1} \Big\{ \frac{v}{n}e^{-\frac{K_1\eps nH^2}{v}}+\frac{49 M_1^2}{K_1 C^2(\eps)n^2}e^{-\frac{\sqrt{2\eps}K_1C(\eps)nH}{7M_1}} \Big\},
\end{align*}
where $K_1=1/6$, $C(\eps)=(\sqrt{\eps+1}-1)\wedge 1$,
\begin{align*}
\sup_{r\in\mathcal{R}}\|r\|_\infty\leq M_1,\quad E\sup_{r\in\mathcal{R}}|\nu_n(r)|\leq H,\quad\text{and}\quad \sup_{r\in\mathcal{R}}V(r(\zeta_1))\leq v.
\end{align*}
\end{lemma}
Note that standard density arguments allow to apply the above lemma to the unit sphere of a finite dimensional linear space.

\begin{lemma}[Theorem 4.1.9 of \cite{gine:nickl:2016}]\label{lem:chisquare}
For $\zeta_1,...,\zeta_n\stackrel{iid}{\sim}N(0,1)$ and $\lambda_1,...,\lambda_n\in\mathbb{R}$ we have that
\begin{align*}
P\big( \sum_{i=1}^n \lambda_i^2 (\zeta_i^2-1) \geq t\big)\leq\exp\Big\{ -\frac{t^2}{4(\sum_{i}\lambda_i^2+t\max|\lambda_i|)} \Big\}.
\end{align*}
\end{lemma}

The next lemma is Gershgorin circle theorem \cite{gershgorin:1931}.
\begin{lemma}\label{lem:Gershgorin}
Every eigenvalue of a matrix $\mathbb{R}^{n\times n}$ satisfies
\begin{align*}
|\lambda-A_{ii}|\leq \sum_{j\neq i}A_{ij}, \quad i\in\{1,...,n\}.
\end{align*}
\end{lemma}

For convenience we also recall Lemma 5.3 of \cite{chagny:2013}. Let $S(m)=\{\sum_{j=1}^{D_m}b_{j}\phi_j:\, \sum_{j=1}^{D_m}b_{j}^2=1 \}$ denote the unite sphere in the linear subspace spanned by the basis functions $\phi_{j}$, $j=1,...,D_m$. 

\begin{lemma}\label{lem:5.3}
Let $\nu: L^{2}[0,1]\mapsto \mathbb{R}$ be a linear functional. Then 
\begin{align*}
\sup_{g\in S(m)}\nu(g)^2= \sum_{j=1}^{D_m} \nu(\phi_j)^2.
\end{align*}
\end{lemma}

Finally we give a version of Bernstein's inequality, given for instance in Proposition 2.9 of \cite{massart2007concentration}.
\begin{lemma}\label{lem: Bernstein}[Bernstein's inequality]
Let $X_1, ..., X_n$ be independent real valued random variables. Assume that there exist some positive numbers
$v$ and $c$ such that $\sum_{i=1}^n EX_i^2\leq v$ and for all integers $k\geq 3$
\begin{align*}
\sum_{i=1}^n E (X_i)_+^k\leq (k!/2)v c^{k-2}.
\end{align*}
Then 
\begin{align*}
P(\sum_{i=1}^n (X_i-EX_i) \geq \sqrt{2vx}+cx)\leq e^{-x}.
\end{align*}
\end{lemma}
\end{document}